\documentclass{amsart}[11pt]
\usepackage{graphicx,url,amsthm,xspace,enumitem,subcaption}
\usepackage[dvipsnames]{xcolor}
\usepackage{amscd}
\usepackage{multirow}
\usepackage[utf8]{inputenc}
\usepackage{tikz}
\usetikzlibrary{braids}
\graphicspath{{images/}}

\usepackage{amssymb}
\usepackage{graphics}
\usepackage{graphicx}
\usepackage{amscd}
\usepackage{color}
\usepackage{amsmath, amsthm, epsfig,  psfrag}
\usepackage{graphicx,url,amsthm,xspace,enumitem}
\usepackage{pinlabel}
\usepackage{hyperref}
\usepackage{subcaption}
\usepackage[margin=3cm, marginpar=2.5cm]{geometry}

\usepackage[all,pdftex,arc,curve,color,frame]{xy}
\usepackage{xy}
\xyoption{all}

\usepackage{pinlabel}

\usepackage{import}

\newcommand{\TN}{\mathbb{T}}
\newcommand{\I}{\operatorname{I}}
\newcommand{\X}{\operatorname{X}}
\newcommand{\T}{\operatorname{T}}

\newcommand{\C}{\mathbb{C}}

\newcommand{\Q}{\mathbb{Q}}
\newcommand{\Z}{\mathbb{Z}}
\newcommand{\cptwo}{\C\textup{P}^2}

\newcommand{\cptwobar}{\overline{\C\textup{P}}\,\!^2}

\newcommand{\ti}{\tilde}
\renewcommand{\d}{\partial}

\renewcommand{\Im}{\operatorname{Im}}

\DeclareMathOperator{\Int}{Int}

\newtheorem{theorem}{Theorem}[section]
\newtheorem{lemma}[theorem]{Lemma}
\newtheorem{prop}[theorem]{Proposition}

\theoremstyle{definition}

\newtheorem{remark}[theorem]{Remark}
\newtheorem{example}[theorem]{Example}

\numberwithin{equation}{section}

\title{An unexpected rational blowdown}

\author{M\'arton Beke}
\address{Alfréd Rényi Institute of Mathematics, Budapest\\ Budapest University of Technology and Economics, Budapest, Hungary}
\email{bekem@renyi.hu}

\author{Olga Plamenevskaya}
\address{Department of Mathematics, Stony Brook University, Stony Brook, NY,
11794,  U.S.A.}
\email{olga@math.stonybrook.edu}

\author{Laura Starkston}
\address{Department of Mathematics, University of California, Davis, 1 Shields Avenue, Davis, CA, 95616, U.S.A.}
\email{lstarkston@math.ucdavis.edu}

\begin{document}

\begin{abstract}

The rational blowdown operation in 4-manifold topology replaces a neighborhood of a configuration of spheres by a rational homology ball. Such configurations typically arise from resolutions of surface singularities that admit rational homology disk smoothings. 
Conjecturally, all such singularities must be weighted homogeneous and belong to certain specific families: Stipsicz--Szab\'o--Wahl constructed $\Q$HD smoothings for these families and used Donaldson's theorem to obtain very restrictive necessary conditions on the resolution graphs for singularities with this property. In particular, these results, as well as subsequent work of Bhupal--Stipsicz, show that for certain resolution graphs, the canonical contact structure on the link of the singularity cannot admit a symplectic $\Q$HD  filling. 
 By contrast, we exhibit Stein rational homology disk fillings for the contact links of an infinite family of rational singularities that are {\em not} weighted homogeneous, producing a new symplectic rational blowdown. Inspiration for our construction comes from de Jong--van Straten's description of Milnor fibers of sandwiched singularities; we use the symplectic analog of de Jong--van Straten theory developed by the second and third authors.  The unexpected Stein fillings are built using spinal open books and nearly Lefschetz fibrations.      

\end{abstract}

\maketitle

\section{Introduction} In 4-manifold topology, a rational blowdown 
surgery \cite{FintushelStern} is used to create exotic pairs of smooth manifolds, \cite{JPark, StipSzabo};  this operation can be done symplectically \cite{Symington}. 
Rational blowdowns typically come from smoothings of certain surface singularities. If $(X, 0)$ is a normal surface singularity with minimal good resolution $\tilde{X}$, then the standard neighborhood of the exceptional divisor in $\tilde{X}$ is a plumbing of disk bundles over the exceptional curves, according to the dual resolution graph $G$. 
If the corresponding configuration of surfaces is contained in a smooth 4-manifold, it can be replaced by the Milnor fiber of a smoothing of the singularity. To create a manifold with small homology, one wants a Milnor fiber that is a rational homology disk. 

From another perspective, understanding the topology of Milnor fibers is an interesting question in singularity theory. By \cite{Wahl}, a surface singularity must be rational if it admits 
a smoothing that is a rational homology disk. However, few rational singularities have such smoothings. All known examples fall into several specific families of \textit{weighted homogeneous} singularities, that is, they admit a good $\C^*$ action, which implies that the resolution graphs are star-shaped (see \cite[\S5.1.9]{nemethi2022normal}). In the $\Q$HD smoothable case the node of the resolution graph has valency three or four, \cite{SSW}. It has been conjectured   (\cite[Conjecture 8.10]{wahl2011rational})  that the known examples give a complete list of surface singularities with $\Q$HD smoothings.

In~\cite{SSW}, Stipsicz--Szab\'o--Wahl  used Donaldson's theorem to attack the question by examining possible embeddings of the homology lattice of the plumbing into the diagonal lattice of the same rank. This gives very restrictive necessary conditions for the dual resolution graphs whose link can admit a symplectic $\Q$HD filling. Bhupal--Stipsicz (\cite{bhupal2011weighted}) settled additional cases by studying  symplectic fillings of the links of the corresponding singularities with their canonical contact structures: they showed that the only  singularities with star-shaped resolution graphs whose links admit $\Q$HD fillings are those with known 
$\Q$HD smoothings. (This classification is reproved and extended with methods closer to the ones used in this paper in \cite{Beke-inprep} when combined with the results of \cite{PS2}.)

One could perhaps hope that the conjectural list exhausts all options for a symplectic rational blowdown, namely that $\Q$HD symplectic fillings for the contact link only exist for the singularities on this list. In particular, one could conjecture that $\Q$HD symplectic fillings cannot exist for resolution graphs with two or more nodes. However, we prove:   

\begin{theorem}\label{thm:main} Let $G_{k,n}$ be a dual resolution graph as in Figure~\ref{fig:graphs}, 
$(Y_{k,n}, \xi_{k,n})$ the link of the corresponding singularity with its canonical contact structure. 
Then for all $k\geq 0, n\geq 1$, $(Y_{k,n}, \xi_{k,n})$ admits a Stein filling which is a rational homology disk, despite the fact that the surface singularities of the given topological type admit no rational homology disk smoothing. 
\end{theorem}

\begin{figure}[h!]
\def\svgwidth{0,45\columnwidth}
\centering
    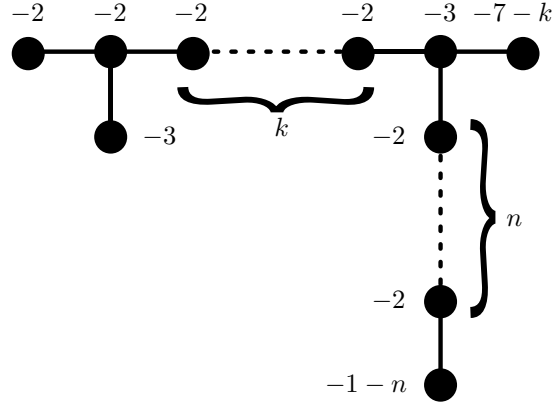

\caption{The family $G_{k,n}$ for $k\geq -1,n\geq 1$, where $G_{k,n}$ has $k+n+6$ vertices, with $k$ vertices between the two nodes. The graph $G_{-1,1}$ has a single node of valency $4$, with 6 vertices total.}
\label{fig:graphs}
\end{figure}

For $k\geq 0$, the graphs of Figure~\ref{fig:graphs} all have two nodes and belong to the family $\mathcal{C}$ of \cite{SSW}, producing a new symplectic rational blowdown. It is interesting to note that for\footnote{This notation means contracting the edge between the two nodes, and giving them framing $a+b+2$ if the two nodes were framed $a,b$ respectively.} $k=-1$, a rational homology disk smoothing actually exists for a
singularity in the corresponding topological type by \cite[Theorem 5(a)]{fowler2013rational}, when
the cross ratio of the intersection points of the node with the four arms
is $9$. 
Our construction produces a $\Q$HD Stein filling in this case as well, but we do not know whether our filling is diffeomorphic to the Milnor fiber. With this in mind, our construction is an extension of the $\mathcal C^4$ family of \cite{bhupal2011weighted}.

It seems plausible that similar $\Q$HD Stein fillings can be constructed for other families of graphs 
of the $\mathcal{A}$, $\mathcal{B}$, $\mathcal{C}$ types of \cite{SSW} by similar extensions of $\mathcal A^4,\mathcal B^4$ or a different arm of $\mathcal C^4$. Possibly, one can find such fillings for graphs with an arbitrarily large number of nodes (our arguments in Section 7 hint at a possible induction process), but we currently lack the methods to show that the corresponding singularities cannot admit any $\mathbb Q$HD deformations. Combinatorial possibilities for many families with  $\Q$HD Stein fillings can be found by a computer search.  However, each new family requires tedious separate analysis, so we only produce one 2-parameter family in this paper. 

These graphs all correspond to {\em sandwiched} singularities (in combinatorial terms, this means that the graph can be augmented, by adding $(-1)$ vertices, to a plumbing graph that can be blown down to a smooth point). Deformation theory for this class of singularities can be understood thanks to a very attractive approach of de Jong--van Straten \cite{djvs}: the surface singularity $(X,0)$ is encoded via the germ of a singular plane curve $(C, 0)$, and
all smoothings of $(X,0)$ correspond to certain {\em picture deformations} of $(C, 0)$, decorated with marked points. Milnor fibers can be reconstructed directly from picture deformations. In particular, a Milnor fiber of a smoothing is a rational homology disk if and only if the number of irreducible components of $(C ,0)$ matches the number of marked points of the corresponding picture deformation.    
The combinatorial features of the arrangement of curves and marked points in a picture deformation depend on the original dual resolution graph. In~\cite{Beke-inprep}, the first author was able to rule out rational homology disk smoothings for certain families of graphs through combinatorial analysis of possible curve and point arrangements. The family  of Figure~\ref{fig:graphs} was discovered in the process: combinatorial obstructions vanish, even though rational homology disk smoothings do not exist in this case by~\cite[Theorem 8.6]{wahl2011rational}, since the graphs of Figure \ref{fig:graphs} are taut of type $(L_2)-(J_1)-(R_1)$ in the sense of Laufer \cite[Section 2.2]{Lauf}.
For symplectic fillings of the contact link of a sandwiched singularity, the second and third authors developed an analog of the de Jong--van Straten theory \cite{PS1, PS2}, showing that all fillings arise in a similar way from certain immersed disk arrangements in $\C^2$ compatible with the germ $(C, 0)$ associated to  the singularity.  The constructions of~\cite{PS2} use spinal open books and nearly Lefschetz fibrations; the symplectic input comes from \cite{BaHa,LVHMW, LVHMW2,  HRW}.

Using the setting of~\cite{PS2}, we will describe the unexpected $\Q$HD fillings by means of these immersed disk arrangements. In turn,
the arrangements are encoded via {\em braided wiring diagrams with tangencies};  braided wiring diagrams were first introduced in \cite{ARVOLA} as a tool to study complex line arrangements. To get a filling of the given link, we need to fix the boundary braid of the arrangement, up to isotopy.  Inspired by \cite{CS},  we develop  diagrammatic moves that preserve the boundary monodromy, and then construct a new family of fillings from standard smoothings via a series of moves. As our examples demonstrate, this gives a useful strategy to create interesting 4-manifolds in certain situations.  These tools may have further applications in the study of complex curves and surfaces in symplectic 4-manifolds.

From the de Jong--van Straten correspondence and non-existence of the corresponding $\Q$HD smoothings, it follows that the combinatorial arrangements producing our $\Q$HD fillings cannot be realized by  analytic deformations of the corresponding curve germs. We do not know how to prove this deformation non-realizability by a direct argument for analytic deformation of plane curves; it would be interesting to investigate this phenomenon further.   

Given the significance of rational blowdown in 4-manifold topology, a natural question is whether our new rational blowdown can be used to construct any interesting exotica. So far we have not been able to produce anything novel. However, we can, for example, embed the corresponding symplectic plumbing into an elliptic fibration and use standard techniques to obtain some basic examples of exotic manifolds.

We have tried to make the paper reasonably self-contained by outlining the de Jong--van Straten construction as well as the symplectic analog of \cite{PS2}  in Section~2. Section 3 describes  braided wiring diagrams with tangencies, Section 4 gives a number of boundary-preserving moves on such diagrams. 
In Section 5, we use these moves to construct the arrangements that yield unexpected Stein fillings for graphs $G_{k, 1}$. We treat this case first for expository reasons: the diagrams are smaller and easier to follow, and we explain all steps in detail. The general 2-parameter family is relegated to Section 7 for the persistent reader; we give the diagrammatic moves with brief explanations. In Section~6, we give a basic example of our rational blowdown producing some 4-manifold exotica.

\vspace{.2cm}

{\bf Acknowledgements.} This work was initiated by the first two authors during the conference ``New structures in low-dimensional topology'' in Budapest, July 2024. We are grateful to Andr\'as N\'emethi, Andr\'as Stipsicz, and Luya Wang for helpful conversations. 
MB is partially supported by the Doctoral Excellence Fellowship Programme (DCEP) funded by the National Research Development and Innovation Fund of the Ministry of Culture and Innovation and the Budapest University of Technology and Economics and by ERC Advanced Grant KnotSurf4d. OP has been partially supported by the NSF grant DMS 2304080.
LS has been partially supported by NSF CAREER grant DMS 2042345 and Sloan grant FG 2021-16254.

\section{Picture deformations, DJVS arrangements, and $4$-manifold constructions}\label{s:djvs}

We briefly describe de Jong--van Straten's deformation theory for sandwiched singularities. We will use de Jong--van Straten's picture deformations to construct a particular Milnor fiber for the link $(Y_{k,n}, \xi_{k,n})$, which we then modify to obtain unexpected fillings.  A picture deformation gives an arrangement of algebraic curves; for the modification, we use {\em DJVS immersed disk arrangements} 
introduced in \cite{PS2}.

A normal surface singularity $(X, 0)$ is {\em sandwiched} if there exists an embedding of the tubular neighborhood of the exceptional set of its resolution $\tilde{X}$ into some blowup of $\C^2$. One can consider arbitrary resolutions here; for simplicity, we always work with the minimal good resolution, where the exceptional curves are smooth and intersect at transverse double points only. (Note that the minimal resolution is good 
for all rational singularities.) A good resolution can be encoded by the dual resolution graph, whose vertices correspond to the irreducible components of the exceptional divisor, and the edges record intersections between the components.
Sandwiched singularities can be defined by a combinatorial condition on the dual resolution graph: a graph is {\em sandwiched} if it can be augmented  by adding new end vertices with 
self-intersection $(-1)$  in a way that the resulting augmented graph can be blown down to a smooth point at $0$.  (The choice of such augmentation is generally not unique.)
The $(-1)$ vertices correspond to a distinguished collection of $(-1)$ curves in a blowup of $\C^2$ with the embedding of $\ti{X}$,  so that the configuration of these $(-1)$ curves 
together with the original exceptional set can be completely blown down.

In \cite{djvs}, the sandwiched singularity is then encoded by an associated singular plane curve germ, as follows. For each distinguished 
$(-1)$ curve, fix  a transverse complex disk $\ti{C_i}$ through a generic point. Blow down the configuration of the exceptional curves of $(X, 0)$ together with these $(-1)$ curves; the image $C_i$ of the disk $\ti{C_i}$ under the 
blowdown is (a germ of) a curve in $\C^2$, possibly singular at the origin. We set   $C = \cup C_i$ and decorate $C$ with a weight $w$, a collection of integers $w_i=w(C_i)$.   For each $C_i$,  the number $w_i$ is the sum of multiplicities of the intersections 
of the image of $\ti{C_i}$ with the exceptional curves  during the blowdown process. (If  $C_i$ is smooth, then  $w_i$ is simply the number of blowdowns that the corresponding component goes through.)
The original singularity can be reconstructed from the decorated germ $(C, w)$, where $C = \cup C_i$ and $w$ is a collection of weights $w_i$: the corresponding blowups recover the embedding of the tubular neighborhood of the exceptional set of the resolution given by $G$ into some blowup of $\C^2$.

 \begin{figure}[h!]
 \def\svgwidth{\columnwidth}

\centering
    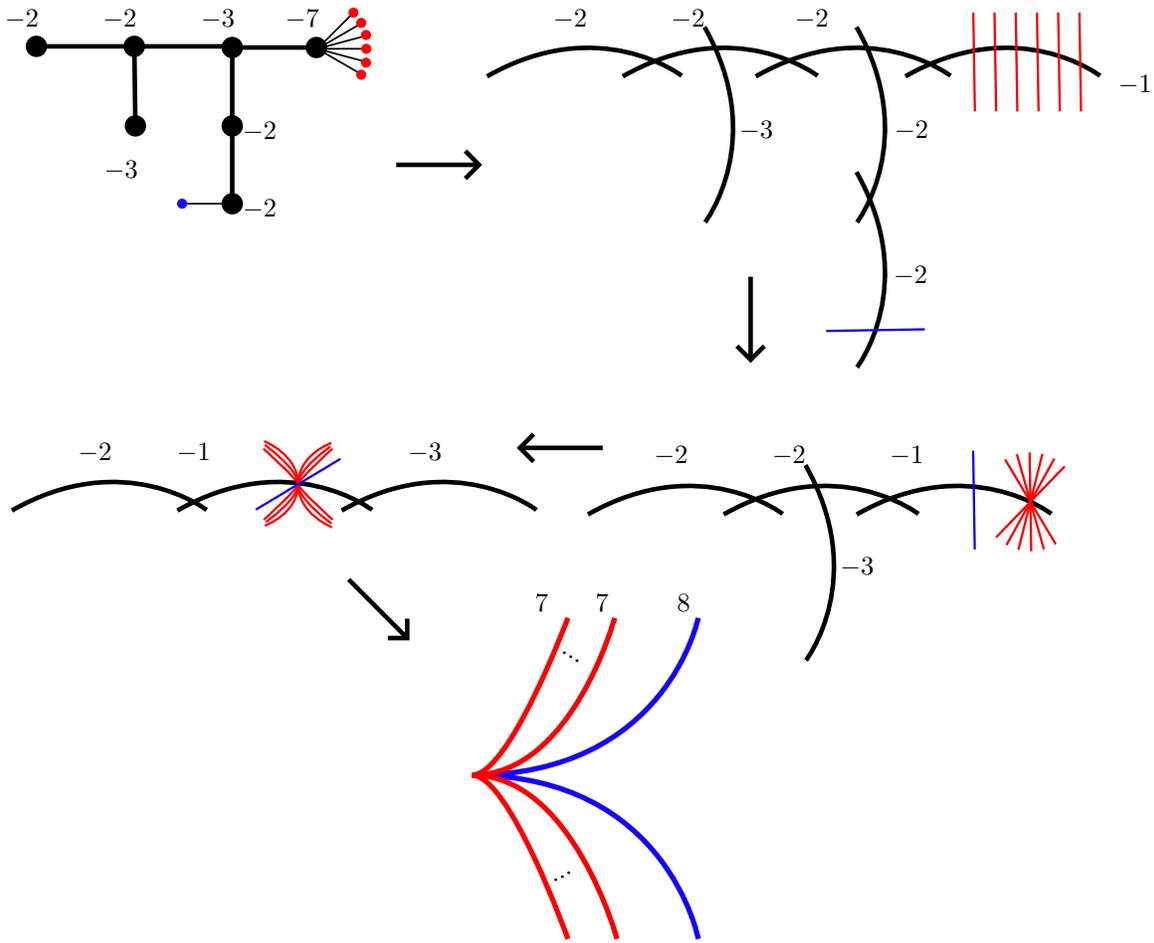

\caption{The decorated germ for the graph $G_{0,1}$. Each component $C_i$ of the germ has a cusp singularity modeled on  $y^2=x^3$. The intersection multiplicities are $C_0 \cdot C_i=7$ for $i=1, \dots, 6$, and $C_i\cdot C_j = 8$ for $1\leq i<j\leq 6$.
 The weights are as indicated, $w(C_0)=8$, $w(C_i)=7$  for $i=1, \dots, 6$.}
\label{fig:djvs}
\end{figure}

Figure~\ref{fig:djvs}  shows the augmented graph, the blowdown process, and the resulting 
decorated germ for the graph $G_{0,1}$ of Figure~\ref{fig:graphs}.  Each component $C_i$ of the germ is modeled
on the cusp $y^2=x^3$ because the previous steps of the blowdown give a familiar embedded resolution of the cuspidal cubic.    We also see that 
components $C_i$ and $C_j$ have intersection multiplicity $8$ when $i, j=1, \dots, 6$, 
and $C_0$ and $C_i$ have multiplicity of intersection $7$ for each $i=1, \dots, 6$. In general, for the graphs $G_{k,n}$ with $k\geq0$ we get cusps $C_1,...,C_{6+k}$ with decoration $7+k$ and pairwise intersection 
multiplicity  $8+k$ for each pair, and 
instead of $C_0$ we have $n$ cusps $C'_1, \dots C'_n$ 
 with decoration $7+k+n$,  with $C'_i\cdot C'_j = 8+k +n$ for $1\leq i <j\leq n$ and    
 $C'_j \cdot C_i = 7+k$ for $i = 1, \dots, k+6$, $j=1, \dots, n$.

\begin{remark} Strictly speaking, our pictures encode topological rather than analytic data. If we consider the exceptional set of the resolution $\tilde{X}$ as an embedded configuration of complex curves, then, once the augmented graph is fixed, $C$ is a germ of a reducible complex curve in $\C^2$ (defined up to analytic equivalence). However, if we only fix the resolution graph and its augmentation, the output of the blowdown process is {\em the topological type} of a germ $C$, that is, the link $C \cap S^3$ of the plane curve singularity $(C, 0)$ obtained by intersecting the curve $C$ with a small sphere centered at the origin. (This is a link in $S^3$ in the knot-theoretic sense, defined up to smooth isotopy.) By~\cite{CPP}, the contact link $(Y, \xi)$ of the singularity $(X, 0)$ depends only on the topological type; to generate symplectic fillings, we will only need to know the isotopy class of 
$C \cap S^3$. 
\end{remark}

De Jong--van Straten's theory describes Milnor fibers of smoothings  of a sandwiched surface singularity $(X, 0)$ via {\em picture deformations} of the decorated  germ.

A {\em picture deformation} of a decorated germ   $(C, w)$ as above is a 1-parameter small analytic deformation $C^t$ equipped with a collection $p$ of marked points (for each $t$) such that

(C-1)    $C^t$ is a $\delta$-constant deformation,

(C-2)   the only singularities of the curve $C^t$ are transverse multiple points for all $t>0$,

(C-3) the marked points $p_j$ are chosen on $C^t$, so that each  self-intersection and each intersection between two irreducible components is marked; additionally, there may be free marked points on each
$C_i$ away from the intersections,

{(C-4) each component $C_i^t$  carries $w_i$ marked points,  
counted with multiplicity of intersection.}

In \cite{djvs},  $p$ is the deformation of the corresponding multiplicity scheme, originally concentrated at $0$ and reduced for $t>0$, but we will think of the marked points simply as a set. 

The $\delta$-constant property implies that the deformation preserves the branches, and each $C^t_i$ is an immersed disk in a Milnor ball $B$, \cite{Tes}. With the appropriate choice of coordinates, we can assume that none of the tangent cones of the branches $C_i$ of the decorated germ are vertical and, moreover,  all vertical tangencies of the projections $C^t_i \to \C_x$ are nondegenerate. This means that outside of self-intersections, each projection is a branched covering with simple branch points only; the degree of the covering is the multiplicity of the corresponding branch $C_i$ at $0$. We fix a good representative of the deformation in a Milnor ball, thought of as the product $B=D_x \times D_y$ of two coordinate disks, with corners smoothed.    All the constructions will take place inside this ball. 
We will also assume that $C \cap \d B$ is contained in $\partial D_x \times D_y$, and think of the link of $C$ as a braid.  

In~\cite{PS2}, the second and third authors introduced immersed disk arrangements that are more general than picture deformations but have similar topological properties. Let 
$(\Gamma, p)$ be an arrangement of immersed smooth disks  
$\Gamma =\cup \Gamma_i$ in $D_x \times D_y$, {with  $\Gamma_i= n_i(D)$ for an immersion 
$n_i: D \to D_x \times D_y$. Let $p=\{p_j\} \subset \Gamma$ be a finite collection of marked points.}

 We say that $(\Gamma, p)$ is 
a {\em DJVS immersed disk arrangement} if

($\Gamma$-1) Away from self-intersections, $\pi_x \circ n_i: D\to D_x$  is a simple branched covering for each $i$;

($\Gamma$-2) All intersections $\Gamma_i \cap \Gamma_k$ and self-intersections of 
the components $\Gamma_i$ are positive transverse multiple points locally modeled on the intersection of complex lines.

($\Gamma$-3) All intersections and self-intersections are marked, and there can be additional free marked points.

($\Gamma$-4) Each disk $\Gamma_i$ intersects the boundary of $D_x \times D_y$ transversely, $\Gamma_i\cap \partial(D_x \times D_y) \subset \partial D_x \times \Int D_y$, all marked points and branch points are contained in the interior of $D_x \times D_y$ and their projections to $D_x$ are distinct.

The {\em weight} $w(\Gamma_i)$ of the component $\Gamma_i$ is the number 
of the marked points  on $\Gamma_i$, again counted with multiplicity of intersection. (For example, if $p_j$ marks a double point of $\Gamma_j$, it is counted twice.)
We will always assume that $\Gamma_i$ is oriented compatibly with the orientation of $D_x \subset \C$.

A DJVS arrangement $(\Gamma, p)$ is {\em compatible} with the decorated germ $(C, w)$ if 
the braids $\d \Gamma$ and $\d C$ are isotopic in $\partial D_x \times \Int D_y$, and $w(\Gamma_i)= w(C_i)$
for the corresponding components. Equivalently, we can assume that  $\d \Gamma_i= \d C_i$, by using braid isotopy to modify the arrangement $\Gamma$ in the collar neighborhood of the boundary (this modification preserves all the topological data used in Theorem~\ref{thm:ps}).

Given a DJVS arrangement $(\Gamma, p)$, we construct a smooth $4$-manifold with boundary $W_{(\Gamma, p)}$: blow up the ball $B$ at all the marked points $p_j$ and take the complement of the tubular neighborhoods of the strict transforms $\ti{\Gamma}_i$ 
in $B \#_n \cptwobar$, 
\begin{equation} \label{eq:DJVSfilling}
W_{(\Gamma,p)} = [(D_x \times D_y)  \#_n \cptwobar] \setminus \cup_i \nu(\tilde{\Gamma}_i).
\end{equation}
By hypothesis ($\Gamma$-2), the strict transforms can be taken as in the standard complex model, and 
after the blowup,  $\ti{\Gamma}_i$ are disjoint smoothly embedded curves. We say that $W_{(\Gamma,p)}$ is obtained from $(\Gamma, p)$ by the DJVS construction. 
Our work builds on the following key results.

\begin{theorem} \label{thm:djvs} \rm{(\cite{djvs})} Every Milnor fiber of a deformation of a sandwiched singularity $(X, 0)$ can be obtained by the DJVS construction from some picture deformation of a fixed decorated germ $(C, w)$ for  
$(X, 0)$. 
\end{theorem}

\begin{theorem} \label{thm:ps} \rm{(\cite{PS2})} (1) For every DJVS arrangement $(\Gamma, p)$ compatible with $(X, 0)$, the 4-manifold $W_{(\Gamma, p)}$ carries a Stein structure and gives a Stein filling of the 
contact link $(Y, \xi)$ of $(X, 0)$.

\noindent (2) Every minimal symplectic filling of the contact link $(Y, \xi)$ can be obtained by the DJVS construction from some DJVS arrangement compatible with $(C, w)$.
\end{theorem}

We will only use the easier part (1) of Theorem~\ref{thm:ps} in this paper: the unexpected $\Q$HD fillings will be constructed from appropriate DJVS arrangements.

Homological invariants of $W_{(\Gamma, p)}$ can be easily computed from the combinatorics of the arrangement: 
\begin{equation}\label{eq:lines-points}
b_1=0,    \qquad b_2= \# \text{(marked points)}- \#\text{(disks)}.
\end{equation}
{Indeed, it is not hard to see that $H_1(W_{\Gamma, p})$ is generated by loops around $\tilde \Gamma_i$’s, $H_2(W)$ can be identified with the kernel of the incidence map
$\mathcal{I}: \Z\langle{p_j}\rangle \to \Z\langle{\Gamma_i}\rangle$, and $H_1(W)$ is the cokernel of this map. We have $b_1=0$ because $W$ consists of 1- and 2-handles only, and  $b_1(Y)=0$ for the link $Y=\d W$ since the singularity is rational,
cf. \cite[Theorem 5.2]{djvs}, \cite[Section 6.1]{PS1}.}
Thus, to produce a $\Q$HD filling, the arrangement $(\Gamma, p)$ must have the number of marked points matching the number of components $\Gamma_i$; this is rare, since typically an arrangement $\Gamma$ has a much larger number of marked points than components.  

To illustrate the discussion, we describe the Milnor fiber of one of the smoothings of the singularities in our family. Consider the germ $C$ of Figure~\ref{fig:djvs} and take its Scott 
deformation, constructed iteratively as follows. We blow up at the singular point at the origin,  perform a small deformation of the strict transform  so that its singularities become disjoint from the exceptional curve $E$, and the branches of the strict transforms are transverse to $E$  (equivalently, we can think of this as ``shifting'' the exceptional curve $E$ off the singularity in the strict transform), and then blow down $E$ to get a deformed curve $C'$ in $\C^2$. 
As a result, the curve $C'$ has a new transverse multiple point as well as  the collection of singularities occurring on the strict transform of $C$ in the blowup of $\C^2$ at $p$. The procedure is repeated until all the singularities at the given stage are transverse multiple points.
See \cite{djvs, ACampo} for details, including the explanation why this procedure can be actually realized by a 1-parameter deformation. Figure~\ref{fig:artin} illustrates the Scott deformation of our decorated germ from Figure~\ref{fig:djvs}.  The marked points on the picture deformation include all the intersections as well as the free points to match the weights on the components of the singular germ $C$. (Decorated germs look similar for all singularities corresponding to graphs $G_{0, n}$, with the number of intersection points increasing with $n$,  although the parity of $n$ affects the order of curve ``tails'' on the right side of the picture.)  For an arbitrary decorated germ, the Scott deformation produces the Artin smoothing under the de Jong--van Straten correspondence, with the Milnor fiber that is diffeomorphic to the minimal resolution of the surface singularity \cite{djvs}. 

\begin{figure}[h!]
\def\svgwidth{0.5\columnwidth}
\centering
    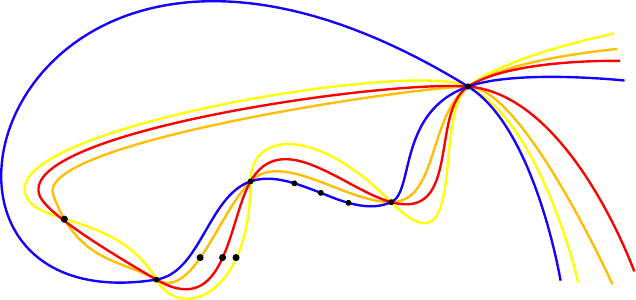

\caption{The Scott deformation of the decorated germ of Figure~\ref{fig:djvs}. The blue curve corresponds to $C_0$ in the decorated germ. Yellow, orange and red curves show deformations of three of the six curves $C_1, \dots, C_6$; each of the other three curves, not pictured, goes through the same intersection points and has one free marked point.}
\label{fig:artin}
\end{figure}

Figure~\ref{fig:artin} shows familiar nodal cubics in the real plane, but it is not very useful when one wants to understand the topology of the arrangement in $\C^2$. In the next section, we will use braided wiring diagrams with tangencies to encode the arrangements.

The isotopy type of the boundary braid of a given germ (and therefore of any compatible arrangement) can be easily read off from the Puiseux coefficients of the plane curve singularity, \cite{BrKnor}.  For the germ corresponding to the graph $G_{0,1}$ of Figure~\ref{fig:graphs}, each component of the boundary braid is a trefoil knot, with the linking number of the knots equal to the intersection multiplicity of the corresponding components $C_i$ of $C$. The topological type of the individual components and the pairwise linking numbers uniquely determine the topological type of an algebraic link, see \cite{BrKnor}.  It is also easy to see, by direct inspection of the arrangement, that the Scott deformation of the curve will have the same boundary braid (as it should).  However, when we try to construct other disk arrangements with the same boundary braid, it is not easy to manipulate the combinatorics of the arrangement and keep track of its boundary braid directly. In the next two sections, we will develop a diagrammatic approach 
that will help us describe the curve arrangements and  modify them while controlling the boundary braid.

\section{Braided wiring diagrams with tangencies} \label{s:diagrams} 

Braided wiring diagrams were introduced by Arvola~\cite{ARVOLA} to study complex line arrangements. 
In \cite{PS2}, we adapted these diagrams to encode DJVS arrangements.   Away from the $x$-values of transverse (self)-intersections,  the projection of a  DJVS arrangement to $D_x$ is a branched covering with simple branch points. The branch points of the projection correspond to points where a component $\Gamma_i$ is tangent to a fiber of $\pi_x$. We refer to these points as tangencies; tangencies and intersections 
are the singular points of $\pi_x: \Gamma \to D_x$. The braided wiring diagram of a DJVS arrangement records the intersections and the tangencies of the arrangement, and how they connect to one another, possibly with some braiding of the components in between. 

For a DJVS arrangement $\Gamma$, let $q_1,\dots, q_N\in D_x$ be the images of intersection points and tangencies (branch points) of $\Gamma$ under the projection $\pi_x$. Choose an embedded piecewise linear arc  $\alpha:I\to D_x$ such that $\alpha(t_1)=q_1,\dots, \alpha(t_N)=q_N$ for $0<t_1<\cdots <t_N<1$ and $\alpha|_{[t_i-\delta,t_i+\delta]}(t) = (t-t_i)+q_i$ runs parallel to the real axis in $D_x$ for some sufficiently small $\delta>0$. Then $\pi_x^{-1}(\alpha(I))\cong [0,1]\times \C$, and $\Gamma\cap \pi_x^{-1}(\alpha(I))$ is a $1$-dimensional braid, except at singular points above $q_1,\dots, q_N$. This restriction of~$\Gamma$ over the preimage of the arc $\alpha$ is  the {\em braided wiring diagram} for $\Gamma$. (The diagram depends on the choice of $\alpha$.) Given a braided wiring diagram, 
the arrangement $\Gamma$ can be reconstructed, up to isotopy, by connecting the local models for the singular points with the thickened braided portions, see \cite[Section 4]{PS2}. The intersection points are modeled on the intersection of complex lines.  
For tangencies, we fix  a local model as in 
\cite[Section 4.3]{PS2}. This involves a choice of perturbation to avoid highly degenerate wiring diagrams. While a different choice of the local model would encode equivalent data, it is important to fix conventions for consistent interpretation of wiring diagrams. Fix $0<\mu\ll \pi/2$. Our model for a tangency is  the complex curve $\{(x,y)\in \C^2\mid e^{i\mu}x = y^2\}$ with the tangency of $\pi_x$ at the origin.  The preimage of the model over the interval in the real $x$-axis includes the singularity and gives the picture for the braided wiring diagram near a tangency, see \cite[Figure 7]{PS2}. If there are two nearby tangencies with nested branches, both in these specific models, then a direct inspection of the model shows that the strands in the corresponding wiring diagram are braided in the neighborhood of the tangencies, as shown in Figure~\ref{fig:tangency-braiding}.  

\begin{figure}[h]
\centering
\def\svgwidth{0.4\columnwidth}
    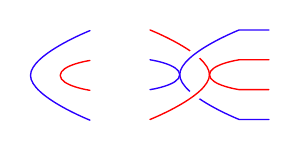

\caption{For two nearby vertical tangencies on curves with nested real parts as on the left, the wiring diagram has braiding of the strands near the tangencies, as shown. The branches of the same curve are shown by strands of the same color.}
\label{fig:tangency-braiding}
\end{figure}

The boundary braid of $\Gamma$ can be read off from the wiring diagram: 
after constructing $\Gamma$ from the local models for the singularities and the thickened braids, the braid $\d \Gamma$ is the restriction of $\Gamma$ over the preimage of the boundary $\d \nu(\alpha)$  of the tubular neighborhood of the arc $\alpha$. The contribution of each combinatorial element of the diagram (intersection point, tangency, braiding) to the boundary braid can be seen directly from the models. Specifically, when the arc $\alpha$ is a segment of the real line in the $x$-plane, the loop $\d \nu(\alpha)$ is formed by the pushoffs of $\alpha$ into the domains $\Im x<0$  (``front'') and  $\Im x >0$ (``back''). 
The braid monodromy of $\d \Gamma$ over the loop $\d \nu(\alpha)$ (traversed counterclockwise)  is obtained by composing the inverse of the braid given by the preimage of $\Gamma$ over the pushoff of $\alpha$ in the positive imaginary direction  (the back of the diagram) with the braid given by the preimage of $\Gamma$ over the pushoff in the negative imaginary direction (the front of the diagram).  By direct inspection, an intersection point in the diagram contributes a positive half-twist on the corresponding strands both to the front and to the inverse of the back portion. (Due to orientations, the contribution to the back is a negative half-twist.) A tangency, with our choice of the local model, contributes a negative half-twist to the back of the diagram (thus a positive
half-twist to its inverse) and two parallel strands in the front. The braided part looks the same in the back and the front pushoffs (contributing to the total boundary braid as partial conjugation due to the sign reversal on the back). If the braided wiring diagram has a unique singular point, we recover the total boundary monodromy of a full positive twist for an intersection point of all strands, and a half-positive twist on two strands for a tangency; the ``front'' and ``back'' analysis allows us to find the monodromy of a diagram with several singular points and braiding.  See \cite[Section 4.3]{PS2} for more details and pictures.

We will use the notational convention where the braids and the braided wiring diagrams are read left to right, and the strands are labeled top to bottom. These conventions differ from those in~\cite{CS} but seem to be more common in low-dimensional topology, although there are no standard conventions in the literature. 

As a warm-up, we illustrate the above discussion with an example of an arrangement that shares some similarities with the arrangements of 
Figure~\ref{fig:artin} but is much simpler.

\begin{example}\label{ex:Scott-warmup} Consider the Scott deformation $C^t$ of a germ $C$  with two irreducible components, each with a simple cusp singularity, intersecting each other with multiplicity $7$. Figure~\ref{fig:Scott-warmup}  shows the real part of the complex curve $C^t$ which is the general fiber of the Scott deformation (left), and the corresponding braided wiring diagram  (right). Note the intertwined strands in the neighborhood of the two tangencies.  
\begin{figure}
	\centering
	\includegraphics[scale=.75]{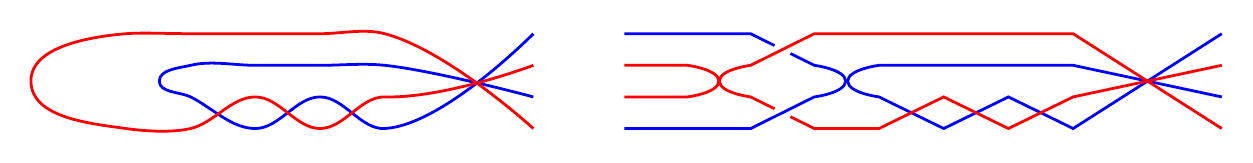}
	\caption{The Scott deformation of the germ  of Example~\ref{ex:Scott-warmup}:
	the real part of the deformed curve (left) and the braided wiring diagram (right). The figure is taken from \cite{PS2}.
	}
	\label{fig:Scott-warmup}
\end{figure}
Constructing the ``back'' and ``front'' pushoff of the wiring diagram as explained above, we can express the boundary braid of the arrangement as [inverse of ``back''][``front''], which gives 
$$
\d C^t = [\Delta_{1,4} \sigma_3^3 \sigma_2 \sigma_3 \sigma_1 \sigma_2] 
[\sigma_1^{-1}\sigma_3^{-1}\sigma_3^3 \Delta_{1,4}],
$$
where $\sigma_i$ stands, as usual, for the positive half-twist between the $i$-th and the $(i+1)$-th strands, and $\Delta_{1,4}$ is the positive half-twist on all strands $1, \dots, 4$, and we used the square brackets to emphasize the contributions from the back and the front of the diagram. The reader can write a simpler expression for the (conjugacy class of) this closed braid and verify that it is consistent with the braid one gets from the Puiseux expansion.
\end{example}

We are now ready to analyze the Scott deformation of the germs corresponding to the family of singularities (topologically) encoded by the graphs $G_{k,1}$.  We ignore the marked points for the moment, since our goal is to understand the boundary braid. The diagrams in Figures~\ref{fig:Scott0}~and~\ref{fig:Scott1} represent the cases $k=-1$ and $k=0$: note that the diagrams are slightly different because the order of the strands on the right of the diagram is affected by parity. These diagrams easily generalize to all $k$ odd resp. $k$ even: increasing $k$ to $k+2$ adds two new germ components (that is, four new strands in the diagram) and two new multipoints involving the ``bottom'' strands of all the components. We hope that the reader is able to see how the patterns in the pictures extend to infinite families.

To be able to write formulas for the diagrams, we introduce notation for certain elements that appear in our examples, see~Figure~\ref{fig:notation}.
The braiding will be written using the standard braid generators, where $\sigma_i$ denotes the positive half-twist between the $i$-th and the $(i+1)$-th strands. (We do not use the word ``crossing'' to avoid confusion with the intersections.) The multipoint intersection of strands $i, i+1, \dots, j$ is denoted $\I^i_{j}$; we only indicate the first and the last strand since the intersection point involves consecutive strands.
For an intersection of only two strands, we also write 
$\I_i=\I^i_{i+1}$.
  We write $\X_{j+1,k}^{i,j}$ for  a configuration of double points where parallel strands  
$i \dots, j$ intersect parallel strands  $j+1, \dots ,k$ in a square grid pattern. (Since the strands are necessarily consecutive,  one of the indices is 
redundant and could be dropped from notation, but we found that the longer notation helps keep track of strands and relations. We will write $\X^{i}_{i+1, k}$ resp. $\X^{i, j}_{j+1}$ when there is 
only one strand in the top or bottom collection.)
The vertical tangency involving the strands $i$  and $(i+1)$ is denoted by
$\T_i$; this notation only makes sense if the two strands belong to the same component of the germ. We encode the diagrams by words in these elements, written from left to right.
 Exponents will mean that the subword is repeated the corresponding number of times; for example,  $(\I^i_{ j})^r$ means that there are $r$ consecutive intersection points involving all strands from $i$ to~$j$. With this notation, the diagram of  Figure~\ref{fig:Scott-warmup} is written as 
 $\T_{2}\sigma_1^{-1}\sigma_3^{-1} \T_{2} (\I_{3})^3 \I^1_{4}$.
\begin{figure}[h]
\centering
    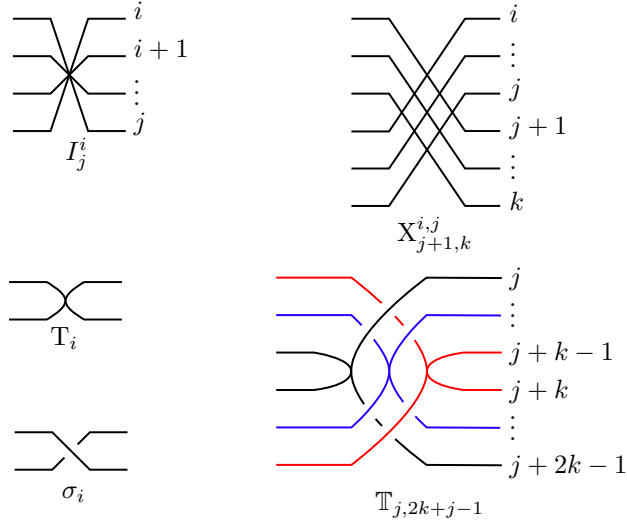

\caption{Notation for elements of braided wiring diagrams with tangencies.}
\label{fig:notation}
\end{figure}
 In our examples that arise from Scott deformation of germs with cuspidal cubic components, we often see a collection of curves with nearby tangencies, with nested pairs of branches and a particular braiding pattern. We call this pattern a \textit{tangency nest} and denote it $\TN_{j, 2k+j-1}$ in a braided wiring diagram, with $k$ tangencies spanning strands from $j$ to $2k+j-1$, see Figure~\ref{fig:notation}. We drop indices and write $\TN$ for the tangency nest on {\em all} strands. For this pattern on $2n$ strands, we have 
$$
\TN_{1, 2n} =  \T_{n}\, \sigma_{n-1}^{-1}\sigma_{n+1}^{-1}\sigma_{n-2}^{-1}\sigma_{n+2}^{-1}\dots 
 \sigma_1^{-1}\sigma_{2n-1}^{-1} \, \T_{n} \, \sigma_{n-1}^{-1}\sigma_{n+1}^{-1}\sigma_{n-2}^{-1}\sigma_{n+2}^{-1}\dots \sigma_2^{-1}\sigma_{2n-2}^{-1}  \dots   \T_{n}\, \sigma_{n-1}^{-1}\sigma_{n+1}^{-1} \, \T_{n},
$$
and the diagram of Figure~\ref{fig:Scott-warmup} is $\TN (\I^3_{4})^3 \I^1_{4}$.
 
\begin{prop}\label{prop:scott-pics}  The braided wiring diagram on $2(k+7)$ strands representing the Scott deformation of the decorated germ with $m=k+7$ components associated to the graph $G_{k,1}$ is given by 
\begin{equation*}
\begin{split}
 \TN_{1, 2m} \I^{m+1+\epsilon}_{ 2m-1+\epsilon} \, 
(\I^{m+1}_{ 2m})^{m-4} \, \I^1_{ 2m},
\end{split}
\end{equation*}
where $\epsilon=0$ for $k$ odd, $\epsilon=1$ for $k$ even. For $k=-1,0$, the diagrams are shown in Figures~\ref{fig:Scott0} and~\ref{fig:Scott1}. 
\end{prop}

\begin{figure}[ht]
	\centering
	\includegraphics[scale=.4]{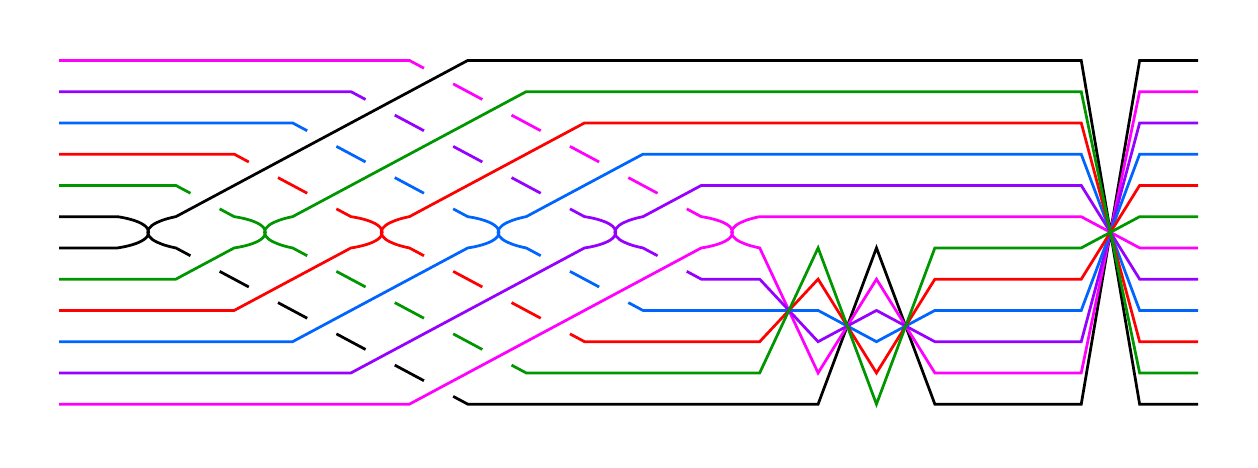}
	\caption{Wiring diagram for the Scott deformation for the germ of the graph $G_{-1,1}$. Here and in subsequent figures, the strands are labeled top to bottom. Strands representing the same irreducible component of the germ are shown in the same color. Diagrams for all odd $k$ follow the same pattern.} 
	\label{fig:Scott0}
\end{figure}

\begin{figure}[ht]
	\centering
	\includegraphics[scale=.3]{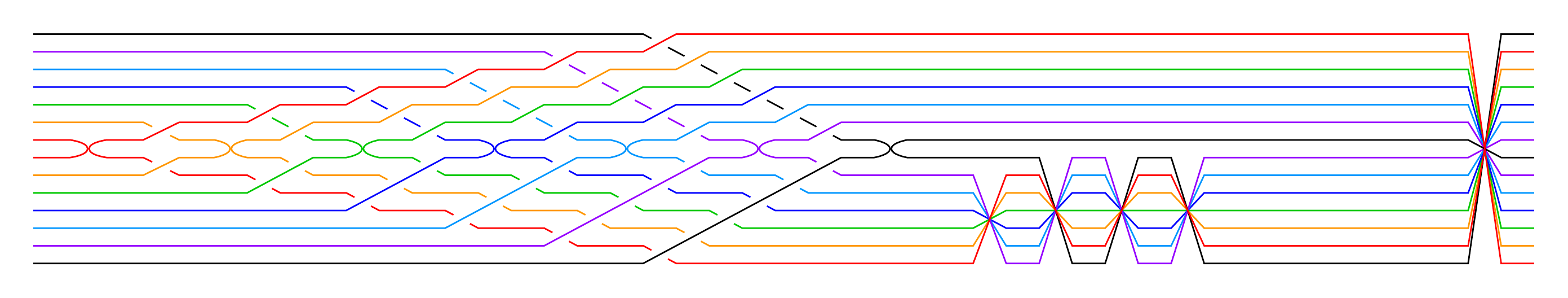}
	\caption{Wiring diagram for the Scott deformation for the germ of the graph $G_{0,1}$.  Diagrams for all even $k$ follow the same pattern.} 
	\label{fig:Scott1}
\end{figure}

\begin{proof} This immediately follows from the structure of the Scott deformation for the given germs (see Figure~\ref{fig:artin}), with the careful treatment of the braiding near the tangencies. 
\end{proof}

\section{Boundary-preserving moves on braided wiring diagrams} 

To prove Theorem~\ref{thm:main}, we need to find DJVS arrangements that are compatible with the corresponding decorated germs and have the number of marked points equal to the number of disks, see~\eqref{eq:lines-points}. 
If one tries to manipulate the combinatorics of the arrangement directly by drawing pictures, it is difficult to control the boundary monodromy that has to remain fixed for compatibility. 
We will use a sequence of simple moves, each giving a useful modification of the diagram while preserving the boundary braid of the arrangement. (The marked points will be ignored during the moves; we only add them at the final stage. All intersection points need to be marked, so the number of intersections must not exceed the given weights, but there can be free marked points if the weights are higher).

For complex line arrangements,  moves on braided wiring diagrams that preserve the monodromy function 
on $\pi_1 (\C_x \setminus \text{singular values})$ were studied by Cohen--Suciu~\cite{CS}. Some of their moves 
are Hurwitz-type modifications relating wiring diagrams  by different choices of
the path $\alpha$ for the same arrangement, although this is not explicit in~\cite{CS}. We will use some of the Cohen--Suciu moves, 
but generally our moves are different in nature: they can change the combinatorics of the arrangement, fixing only the boundary braid. 
	If two braided wiring diagrams have the same boundary braid we will denote this relation by $=_\d$.
Below we describe the moves we use and verify that each move preserves the boundary braid, up to isotopy. (We will only focus on the boundary braid property that we need; stronger properties, such as Hurwitz-type, may hold for some of the moves, but we will not make or verify any stronger claims.) 



\subsection{Commuting elements, expanding $\X$ and $\TN$} Clearly, elements $\I$, $\X$, $\T$, $\sigma$ on disjoint subsets of strands commute 
{since the corresponding arrangements are obviously isotopic through a one-parameter family of DJVS arrangements, with the boundary braid in the same isotopy class.}
In particular,    
\begin{equation}\label{eq:commute}
 \I^i_{j} \I^k_l = \I^k_{l} \I^i_{j}, \qquad \I^i_j \T_k = \T_k \I^i_j, \qquad  \I^i_j \sigma_k = \sigma_k \I^i_j,
\end{equation}
if $i<j<k<l$.
We can draw commuting elements in the diagram without specifying their order. 

A trivial notational relation expands an element $\X^{i, j}_{j+1, l}$ by separating strands:
\begin{equation}\label{eq:expandX}
 \X^{i, j}_{j+1, l} = \X^{i', j}_{j+1, l} \X^{i, i'-1}_{i', i'+l -j-1} = 
 \X^{i, j}_{j+1, l'} \X^{i + l'-j, l'}_{l'+1, l}.
\end{equation}
For example, $\X^{2, m}_{m+1, 2m} = \X^{3, m}_{m+1, 2m} \X^2_{3, m+2}$ means that the configuration of parallel strands $2, \dots m$  intersecting parallel strands $m+1, \dots, 2m$ is the same as 
strands $3, \dots, m$ intersecting $m+1, \dots, 2m$ followed by strand $2$ intersecting the bottom $m$ strands (which have indices $3, \dots, m+2$ after crossing the previous strands).  

Similarly, we can expand a tangency nest: 
\begin{equation}\label{eq:expandTN}
\TN_{1,2n}= \TN_{2,2n-1}\sigma_{2n-1}^{-1}...\sigma_{n+1}^{-1}\sigma_1^{-1}...\sigma_{n-1}^{-1}\T_n.
\end{equation}

\subsection{Splitting/merging multipoints.} A multipoint involving $n$ consecutive strands in the diagram can split in a number of ways. It can split generically into $n\choose 2$ double points, or there can be several multipoints corresponding to the intersection of some subsets of the original set of strands, with other strands intersecting at double points. The inverse move merges two or more multipoints into one. See Figure~\ref{fig:split-multipts}.
	Note that no braiding is introduced between the intersections when splitting a multipoint. One can see this using a model given by complexified real lines, or by directly working with braid relations for the front and back braid pushoffs. Correspondingly, when merging multipoints, it is critical that there is no braiding in between the multipoints which could interfere with deforming them to a single multipoint.
In our notation,  when a multipoint involving strands from $i$ to $k$ splits into a multipoint of intersection of strands from $i$ to $j$ and the other of strands from $j+1$ to $k$, $i\leq j \leq k$, with the pairs of strands from these different groups meeting at double points: 
\begin{equation}\label{eq:split-many}
 \I^i_{k}=_\d \I^i_{j} \I^{j+1}_k \X^{i, j}_{j+1, k} =_\d   \X^{i,j}_{j+1, k} \I^i_{i+k-j-1} \I^{i+k-j}_k.
\end{equation}

\begin{figure}[h]
\centering
    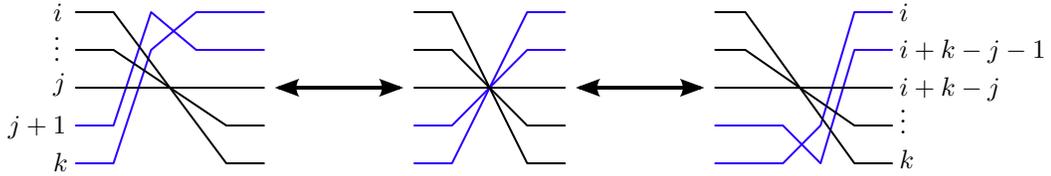

\caption{Splitting a multipoint, as in~\eqref{eq:split-many}.}
\label{fig:split-multipts}
\end{figure}

When just one  strand  (top or bottom) is moved off a multipoint, meeting the other strands in double points instead, we have the relation 
\begin{equation}\label{eq:split-one}
\I^{i}_j=_\d  \I^{i+1}_j \X_{i+1,j}^i =_\d \X_{i+1, j}^i \I^{i}_{j-1}  =_\d   \I^{i}_{j-1} \X_{j}^{i, j-1}
=_\d \X_{j}^{i, j-1} \I^{i+1}_{j}, 
\end{equation}
and similarly
\begin{equation}\label{eq:split-one-middle}
 \I^i_k =_\d \X^j_{j+1, k} \I^i_{k-1} \X^{i+k-j, k-1}_k
\end{equation}
for splitting strand $j$ off an intersection multipoint of strands $i$ to $k$, $i<j<k$. 
There are other variants for splitting into various partitions of the set of strands.

\subsection{Moving multipoints through other lines.} Using an isotopy of individual components,  we can move a multipoint through a collection of adjacent parallel strands that intersect  the strands in the multipoint in double points:
\begin{equation}\label{eq:move-multipts-lines}
\I^i_j \X^{i, j}_{j+1, k} =_\d \X^{i, j}_{j+1, k} \I^{i+k-j}_k, \qquad \I^{j+1}_k \X^{i, j}_{j+1, k} =_\d 
\X^{i, j}_{j+1, k} \I^{i}_{i+k-j-1}.
\end{equation}
see Figure~\ref{fig:move-multipts-lines}(top). Note how the indices of $\I$ and $\X$ must match for this move to be possible. 
This move is modeled by the corresponding move on the complexification of a real line arrangement.
 Note that the multipoint splittings in the previous subsection are related by a sequence of moves of this type. 

We can also switch the order of two multipoints when one involves a subset of strands of the other as in Figure~\ref{fig:move-multipts-lines}(bottom):  
\begin{equation}\label{eq:switch-multipts}
 \I^i_j \I^i_k =_\d \I^i_k \I^{k+i-j}_k, \quad i<j \leq k.
\end{equation}
This move can be obtained as a composition of a multipoint splitting, moving multipoint past lines, and merging back (moves~\eqref{eq:split-many}, its inverse, and move~\eqref{eq:move-multipts-lines}), or proved directly, cf~\cite{CS}.

\begin{figure}[h]
\centering
    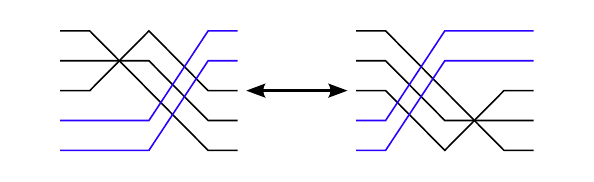

    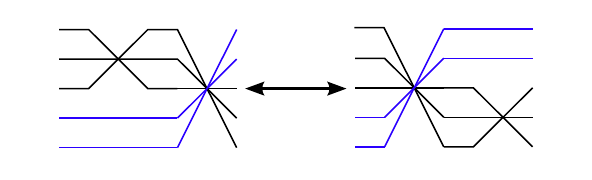

\caption{Top: moving multipoints past other lines, as in~\eqref{eq:move-multipts-lines}. Bottom: switching two multipoints.}
\label{fig:move-multipts-lines}
\end{figure}

\subsection{Switching self-intersections and tangencies on same strands.}  We can interchange a tangency with an adjacent intersection between the same two strands in the diagram:
\begin{equation} \label{eq:swap-tang-crossing}
\I_{i} \T_i =_\d \T_i \I_{i},
\end{equation}
see Figure~\ref{fig:swap-tang-crossing}. The boundary monodromy is given by 
$\sigma_i^3$. 
\begin{figure}[h]
\centering
    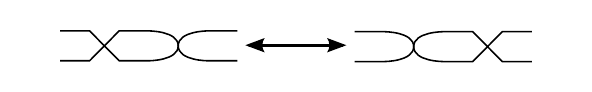

    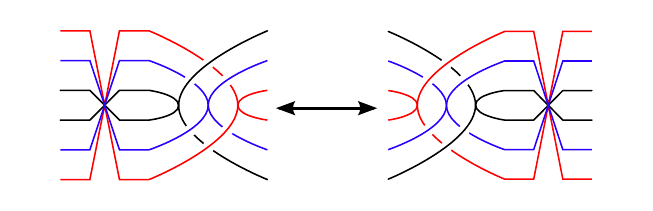

\caption{Swapping self-intersection and tangency, as in~\eqref{eq:swap-tang-crossing}.}
\label{fig:swap-tang-crossing}
\end{figure}
More generally, one can switch a full intersection of $2n$ strands and a tangency nest on these $2n$ strands, 
$$\TN_{1, 2n}\I^1_{2n}=_\d \I^1_{2n}\TN_{1,2n}.$$
This move is modeled by the Scott deformations of two topologically equivalent germs of $n$ cusps with intersection multiplicity 4; the boundary braid is determined by the topological type of the singularity.

\subsection{Moving multipoints through tangencies.} We can move an intersection point through a tangency or a tangency nest, in several ways. Move (i) below is a basic version, moves (ii) and (iii) are its generalizations, and moves (iv) and (v) are variants. See Figure~\ref{fig:all-multipt-tang}.  These moves are more involved than the previous ones: they introduce additional braiding rather than simply switching the elements of the diagram.    Only moves (i) and (ii) will be used in Section~5, the rest are not needed until Section~7. 
For this type of move, one needs to be careful with various symmetric versions: because of the asymmetry in our chosen tangency model, the braiding is somewhat asymmetric.

\noindent {\bf(i)} Moving a double point up or down through a tangency: 
\begin{equation}\label{eq:move-doublept-tang}
 \T_i \I_{i+1} =_\d \sigma_{i+1} \sigma_i^{-1} \T_{i+1} \I_i
  \end{equation}

\noindent {\bf (ii)} Moving a multipoint up or down through a tangency nest on the same strands:
\begin{equation}\label{eq:move-multipt-tang}
 \TN_{1, 2n} \I^1_{n} =_\d \Delta_{1,n} \Delta_{n+1, 2n}^{-1} \TN_{1, 2n} \I^{n+1}_{2n}
\end{equation}
where $\Delta_{j, k}$  stands for a positive half-twist on the collection of strands $j, \dots, k$. We stated the move for the tangency nest on strands $1, \dots, 2n$. 

 \noindent {\bf (iii)} Moving a strand through a tangency nest, with double point intersections: 
\begin{equation}\label{eq:move-strand-tang}
\TN_{1,2n}\I_{2n}\I_{2n-1}\dots\I_{n+1}=_\d \sigma_{2n}\sigma_{2n-1}\dots\sigma_{n+1}\sigma_n^{-1}\sigma_{n-1}^{-1}\dots\sigma_1^{-1}\TN_{2,2n+1}\I_1\I_2\dots\I_{n}.
\end{equation}

\noindent {\bf (iv)} Switching the order of a double point and a tangency:

\begin{equation}\label{eq:move-doublept-tang2}
\T_i\I_{i+1}=_\d \I_{i+1}\sigma_{i+1}\T_i\sigma_{i+1}^{-1}
\end{equation}

\noindent {\bf (v)}  Moving a multipoint to the other side of a tangency nest on the same strands:

\begin{equation}\label{eq:TNIswap2}
\TN_{1,2n}\I^1_n=_\d \I^1_n\Delta_{1,n}\TN_{1,2n}\Delta_{1,n}^{-1}
\end{equation}

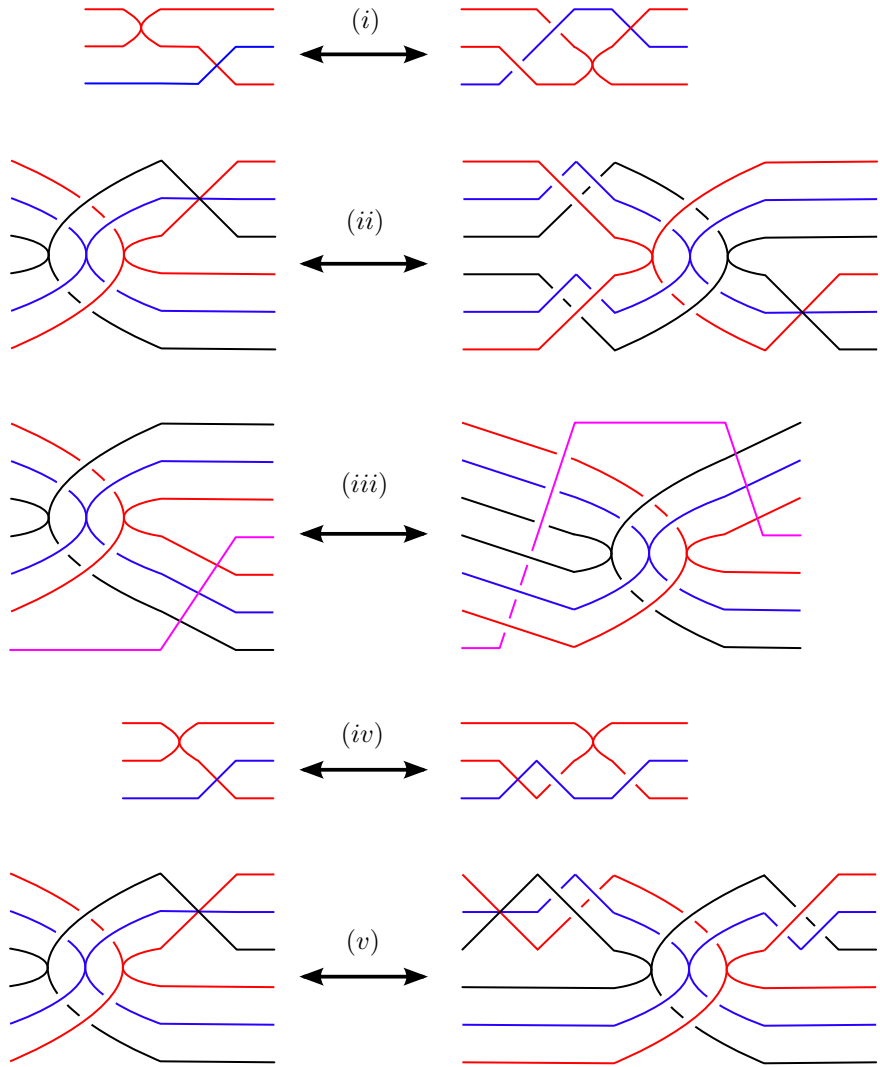
\begin{figure}[h]
\centering
%
\begingroup%
  \makeatletter%
  \providecommand\color[2][]{%
    \errmessage{(Inkscape) Color is used for the text in Inkscape, but the package 'color.sty' is not loaded}%
    \renewcommand\color[2][]{}%
  }%
  \providecommand\transparent[1]{%
    \errmessage{(Inkscape) Transparency is used (non-zero) for the text in Inkscape, but the package 'transparent.sty' is not loaded}%
    \renewcommand\transparent[1]{}%
  }%
  \providecommand\rotatebox[2]{#2}%
  \newcommand*\fsize{\dimexpr\f@size pt\relax}%
  \newcommand*\lineheight[1]{\fontsize{\fsize}{#1\fsize}\selectfont}%
  \ifx\svgwidth\undefined%
    \setlength{\unitlength}{353.59715499bp}%
    \ifx\svgscale\undefined%
      \relax%
    \else%
      \setlength{\unitlength}{\unitlength * \real{\svgscale}}%
    \fi%
  \else%
    \setlength{\unitlength}{\svgwidth}%
  \fi%
  \global\let\svgwidth\undefined%
  \global\let\svgscale\undefined%
  \makeatother%
  \begin{picture}(1,1.20932306)%
    \lineheight{1}%
    \setlength\tabcolsep{0pt}%
    \put(0,0){\includegraphics[width=\unitlength,page=1]{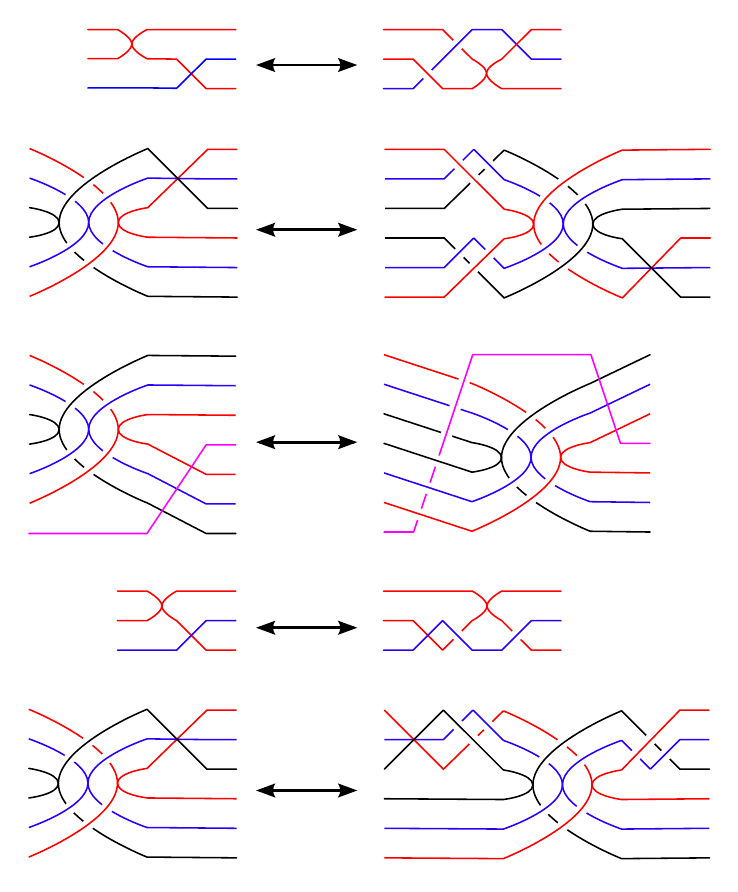}}%
    \put(0.40215752,1.14922022){\color[rgb]{0,0,0}\makebox(0,0)[lt]{\lineheight{1.25}\smash{\begin{tabular}[t]{l}$(i)$\end{tabular}}}}%
    \put(0.39679553,0.93419162){\color[rgb]{0,0,0}\makebox(0,0)[lt]{\lineheight{1.25}\smash{\begin{tabular}[t]{l}$(ii)$\end{tabular}}}}%
    \put(0.39143348,0.65361075){\color[rgb]{0,0,0}\makebox(0,0)[lt]{\lineheight{1.25}\smash{\begin{tabular}[t]{l}$(iii)$\end{tabular}}}}%
    \put(0.39161446,0.38764363){\color[rgb]{0,0,0}\makebox(0,0)[lt]{\lineheight{1.25}\smash{\begin{tabular}[t]{l}$(iv)$\end{tabular}}}}%
    \put(0.39697651,0.16716558){\color[rgb]{0,0,0}\makebox(0,0)[lt]{\lineheight{1.25}\smash{\begin{tabular}[t]{l}$(v)$\end{tabular}}}}%
  \end{picture}%
\endgroup%

\caption{Moving multipoints through tangencies and tangency nests: moves (i)--(v).}
\label{fig:all-multipt-tang}
\end{figure}

It seems plausible that the moves (i)--(v) can be realized by isotopies of individual components. {Move (iv) is a Hurwitz move}.  However, we only prove a weaker claim that we need:  if the local move is performed on a larger arrangement, the elements before and after the move make equivalent contributions into the total boundary braid of the global arrangement, therefore the boundary braid is preserved, up to isotopy. For each move, we  write the ``front'' and ``back'' boundary braids,  as explained in Section~\ref{s:diagrams}, and show that 
the ``front'' braid of the left-hand side is isotopic to the ``front'' of the right-hand side, and the same is true for the ``back'' braids. 
\begin{remark}
	Note that the moves $(i)-(v)$ each have a symmetric pair, where the braiding appears on the other side of the equation. This can be seen by concatenating the diagrams with  the inverse braid (and removing the braiding on one side by a braid isotopy). For example, for move (i) we multiply equation~\eqref{eq:move-doublept-tang} by $\sigma_i \sigma_{i+1}^{-1}$ on the left to get 
$$
	\sigma_i \sigma_{i+1}^{-1} \T_i \I_{i+1} =_\d  \T_{i+1} \I_i.
	$$
	 For moves $(iv), (v)$, it is also useful to observe that $\Delta_{1,n}$ and $\I^1_n$ commute.
	
	Similarly, by mirroring the proofs along two axes, one gets another symmetric version of the moves. This corresponds to reading the formulas right to left and switching the indices $j\leftrightarrow N-j+1$, where $N$ is the number of strands.
\end{remark}
  
\begin{proof}[Proof of Move (i):]    The front and back braids for~\eqref{eq:move-doublept-tang} are given by
$$
\text{ front: }  \sigma_{i+1} = \sigma_{i+1} \sigma_i^{-1} \sigma_i, \qquad \text{ back: }  \sigma_i^{-1} \sigma_{i+1}^{-1}= \sigma_{i+1} \sigma_i^{-1} \sigma_{i+1}^{-1} \sigma_i^{-1}, 
$$
see Figure~\ref{fig:Move6proof}. 
\end{proof}

\begin{figure}[ht]
	\centering
	\includegraphics[scale=.8]{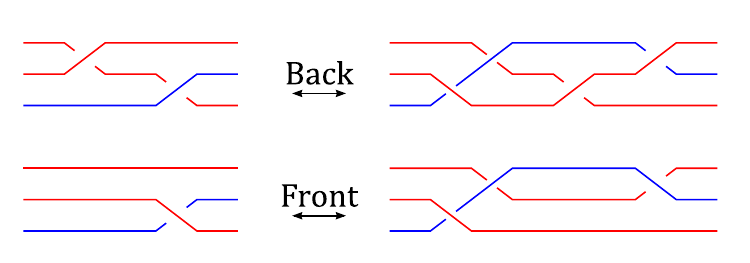}
	\caption{Proof of Move (i): the front and back braids of the diagram, before and after sliding a double point past a tangency.} 
	\label{fig:Move6proof}
\end{figure}

  \begin{proof}[Proof of Move (ii):] For~\eqref{eq:move-multipt-tang},
   the front braid of the diagram before and after the move is
given by 
$$
\text{ front: }  \Delta_{1, n}^{-1} \Delta_{n+1, 2n}^{-1} \Delta_{1, n} = \Delta_{1, n} \Delta_{n+1, 2n}^{-1} \Delta_{1, n}^{-1} \Delta_{n+1, 2n}^{-1} \Delta_{n+1, 2n} = \Delta_{n+1, 2n}^{-1}.
$$ 
The back braids, before and after the move, are 
$$
\text{ back: }   B \Delta_{1, n}^{-1} = \Delta_{1, n} \Delta_{n+1, 2n}^{-1} B \Delta_{n+1, 2n}^{-1},
$$
where 
$B = (\sigma_n^{-1}\sigma_{n-1}^{-1}\dots \sigma_1^{-1})\dots 
(\sigma_{2n-1}^{-1} \dots \sigma_n^{-1})$ is the braid where the parallel strands $1, \dots, n$ cross under the parallel strands 
$(n+1), \dots, 2n$, and the relation above follows from the identities such as  $\Delta_{1, n} B= B \Delta_{n+1, 2n}$  (one can slide any braid on the strands $1, \dots ,n$ past $B$ by an isotopy to get the same braid on the strands $n+1, \dots, 2n$ on the other side). See Figure~\ref{fig:genMove6proof}.
\begin{figure}[ht]
	\centering
	%
\begingroup%
  \makeatletter%
  \providecommand\color[2][]{%
    \errmessage{(Inkscape) Color is used for the text in Inkscape, but the package 'color.sty' is not loaded}%
    \renewcommand\color[2][]{}%
  }%
  \providecommand\transparent[1]{%
    \errmessage{(Inkscape) Transparency is used (non-zero) for the text in Inkscape, but the package 'transparent.sty' is not loaded}%
    \renewcommand\transparent[1]{}%
  }%
  \providecommand\rotatebox[2]{#2}%
  \newcommand*\fsize{\dimexpr\f@size pt\relax}%
  \newcommand*\lineheight[1]{\fontsize{\fsize}{#1\fsize}\selectfont}%
  \ifx\svgwidth\undefined%
    \setlength{\unitlength}{297.24961889bp}%
    \ifx\svgscale\undefined%
      \relax%
    \else%
      \setlength{\unitlength}{\unitlength * \real{\svgscale}}%
    \fi%
  \else%
    \setlength{\unitlength}{\svgwidth}%
  \fi%
  \global\let\svgwidth\undefined%
  \global\let\svgscale\undefined%
  \makeatother%
  \begin{picture}(1,0.66852019)%
    \lineheight{1}%
    \setlength\tabcolsep{0pt}%
    \put(0,0){\includegraphics[width=\unitlength,page=1]{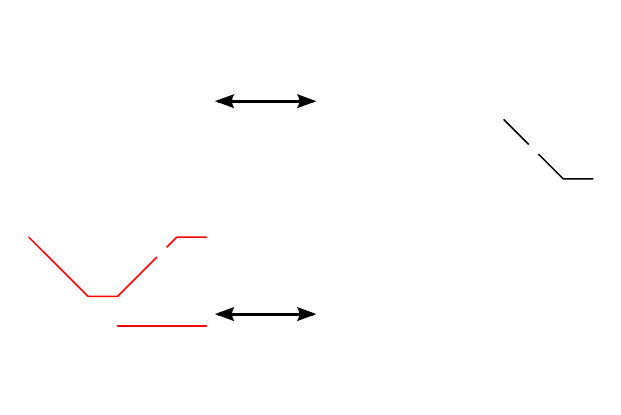}}%
    \put(0.38971613,0.52445779){\color[rgb]{0,0,0}\makebox(0,0)[lt]{\lineheight{1.25}\smash{\begin{tabular}[t]{l}back\end{tabular}}}}%
    \put(0.39057377,0.18769628){\color[rgb]{0,0,0}\makebox(0,0)[lt]{\lineheight{1.25}\smash{\begin{tabular}[t]{l}front\end{tabular}}}}%
    \put(0,0){\includegraphics[width=\unitlength,page=2]{generalizedMove6proof.pdf}}%
  \end{picture}%
\endgroup%

	\caption{Proof of Move (ii): the back (top figure) and the front (bottom figure) braids of the diagram, before and after sliding a multipoint past a tangency nest.} 
	\label{fig:genMove6proof}
\end{figure}
\end{proof}

  \begin{proof}[Proof of Move (iii):] this can be seen by induction from Move (i) or by comparing the front and back braids directly.  We have 
  $$
  \text{ front: }  \Delta_{1, n}^{-1} \Delta_{n+1, 2n}^{-1}  \sigma_{2n} \dots \sigma_{n+1}
  = 
  \sigma_{2n} \dots \sigma_{n+1} \sigma_n^{-1} \dots \sigma_1^{-1} \Delta_{2, n+1}^{-1}  \Delta_{n+2, 2n+1}^{-1} \sigma_1 \dots \sigma_n,
  $$
  where the identity holds because the braid $\sigma_1 \dots \sigma_n$ represents strand 1 crossing over strands 
  $2, \dots n+1$, so
  $\Delta_{2, n+1}^{-1}  \sigma_1 \dots \sigma_n = \sigma_1 \dots \sigma_n \Delta_{1, n}^{-1}$, and similarly 
  $\Delta_{n+1, 2n}^{-1} \sigma_{2n} \dots \sigma_{n+1}=\sigma_{2n} \dots \sigma_{n+1} \Delta_{n+2, 2n+1}^{-1}$. The back braids are examined similarly (the reader should draw a picture).  \end{proof}

  \begin{proof}[Proof of Move (iv):] this is very similar to Move (i), 
$$
\text{ front: }  \sigma_{i+1} = \sigma_{i+1} \sigma_{i+1} \sigma_{i+1}^{-1}, \qquad \text{ back: }  \sigma_i^{-1} \sigma_{i+1}^{-1}= \sigma_{i+1}^{-1} \sigma_{i+1} \sigma_i^{-1} \sigma_{i+1}^{-1}, 
$$
\end{proof}

\begin{proof}[Proof of Move (v):]   using the braid $B$ as in Move (ii), we have
\begin{equation*}
\begin{split}
\text{ front: } &  \Delta_{n+1, 2n}^{-1}\Delta_{1, n}^{-1}\Delta_{1, n}=\Delta_{1, n}\Delta_{1, n}\Delta_{1, n}^{-1}\Delta_{n+1, 2n}^{-1}\Delta_{1, n}^{-1} \\
\text{ back: } &  B\Delta_{1, n}^{-1}=\Delta_{1, n}^{-1}\Delta_{1, n} B\Delta_{1, n}^{-1}.
\end{split}
\end{equation*}
\end{proof}

\begin{remark}
	Note that $\Delta_{1,n}\Delta_{n+1,2n}B^{-1}=\Delta_{1,2n}$, which can also be used to justify the general move of Figure~\ref{fig:swap-tang-crossing}.
\end{remark}

\subsection{Moving braiding past other elements.} We can move a braid element through an intersection or tangency that involves the corresponding strands. For a generator, 
we have 
\begin{equation}\label{eq:move-braid-multipt}
\sigma_j \I^i_k=_\d \I^i_k \sigma_{i+k-j-1}, \qquad i\leq j <k.
 \end{equation}
This is Cohen--Suciu move 5(d), see~\cite{CS}; for a proof in our setting, compare back and front braids before and after. 

There is another type of move, for switching a braid element with tangencies; we leave it to the reader since this move will not be needed in this paper.

\subsection{Removing braiding at the edges} \label{ss:braid-at-edge} If there is a braid element at the beginning or the end of a diagram, this element obviously cancels {in the \emph{closed} braid} when we compose the inverse of the back braid with the front braid of the diagram. Therefore, we can get rid of any braid element at the edge of the diagram. (In fact, the entire disk arrangement can be isotoped to remove the braiding, cf \cite[Section 5.6 (1)-(2)]{CS}.)

\section{$\Q$HD fillings via diagrammatic moves}
In this section, we give a proof of Theorem~\ref{thm:main} for graphs $G_{k,1}$, postponing the general case until 
Section~\ref{s:twoparam}.
Starting with the Scott deformation of the germ corresponding to the family of singularities in Theorem~\ref{thm:main}, we would like to  obtain an arrangement with the same boundary braid, such that the number of intersection points equals the number of curves. The latter requirement holds for the combinatorial arrangement where the curve $\Gamma_0$ of weight $8+k$ has its double point disjoint from all the other curves, and for each of the remaining curves $\Gamma_1, \dots, \Gamma_{k+6}$, all the other curves  go through the double point of every $\Gamma_i$ with  multiplicity $1$ for $i=1,\dots,
k+6$. To generate the immersed disk arrangement with prescribed combinatorial data and the correct boundary, we use the diagrammatic moves from the previous section. The output of the moves, at the end of the procedure, will be the wiring diagrams of the 
desired arrangements; these are shown in 
Figure~\ref{fig:QHD0} for $k$ odd and  Figure~\ref{fig:QHD1}   for $k$ even. 
(The marked points for these DJVS arrangements are the intersection points in the diagram.) 
\begin{proof}[Proof of Theorem~\ref{thm:main} for $n=1$] The proof consists of presenting a sequence of moves changing the diagrams as required; we give this sequence  for $k$ odd.  The moves for $k$ even are similar. For a reader willing to accept a proof by picture, the entire proof is in Figure~\ref{fig:ProofMoves} illustrating the moves connecting Figure~\ref{fig:Scott0} to Figure~\ref{fig:QHD0} for $k=-1$; the pattern extends in a straightforward way to all odd $k$. For a more rigorous argument, we spell out the required moves in formulas below.

\begin{figure}[htbp!]
\def\svgwidth{0.95\textwidth}
	\centering
%
\begingroup%
  \makeatletter%
  \providecommand\color[2][]{%
    \errmessage{(Inkscape) Color is used for the text in Inkscape, but the package 'color.sty' is not loaded}%
    \renewcommand\color[2][]{}%
  }%
  \providecommand\transparent[1]{%
    \errmessage{(Inkscape) Transparency is used (non-zero) for the text in Inkscape, but the package 'transparent.sty' is not loaded}%
    \renewcommand\transparent[1]{}%
  }%
  \providecommand\rotatebox[2]{#2}%
  \newcommand*\fsize{\dimexpr\f@size pt\relax}%
  \newcommand*\lineheight[1]{\fontsize{\fsize}{#1\fsize}\selectfont}%
  \ifx\svgwidth\undefined%
    \setlength{\unitlength}{1381.33117676bp}%
    \ifx\svgscale\undefined%
      \relax%
    \else%
      \setlength{\unitlength}{\unitlength * \real{\svgscale}}%
    \fi%
  \else%
    \setlength{\unitlength}{\svgwidth}%
  \fi%
  \global\let\svgwidth\undefined%
  \global\let\svgscale\undefined%
  \makeatother%
  \begin{picture}(1,1.4346213)%
    \lineheight{1}%
    \setlength\tabcolsep{0pt}%
    \put(0,0){\includegraphics[width=\unitlength,page=1]{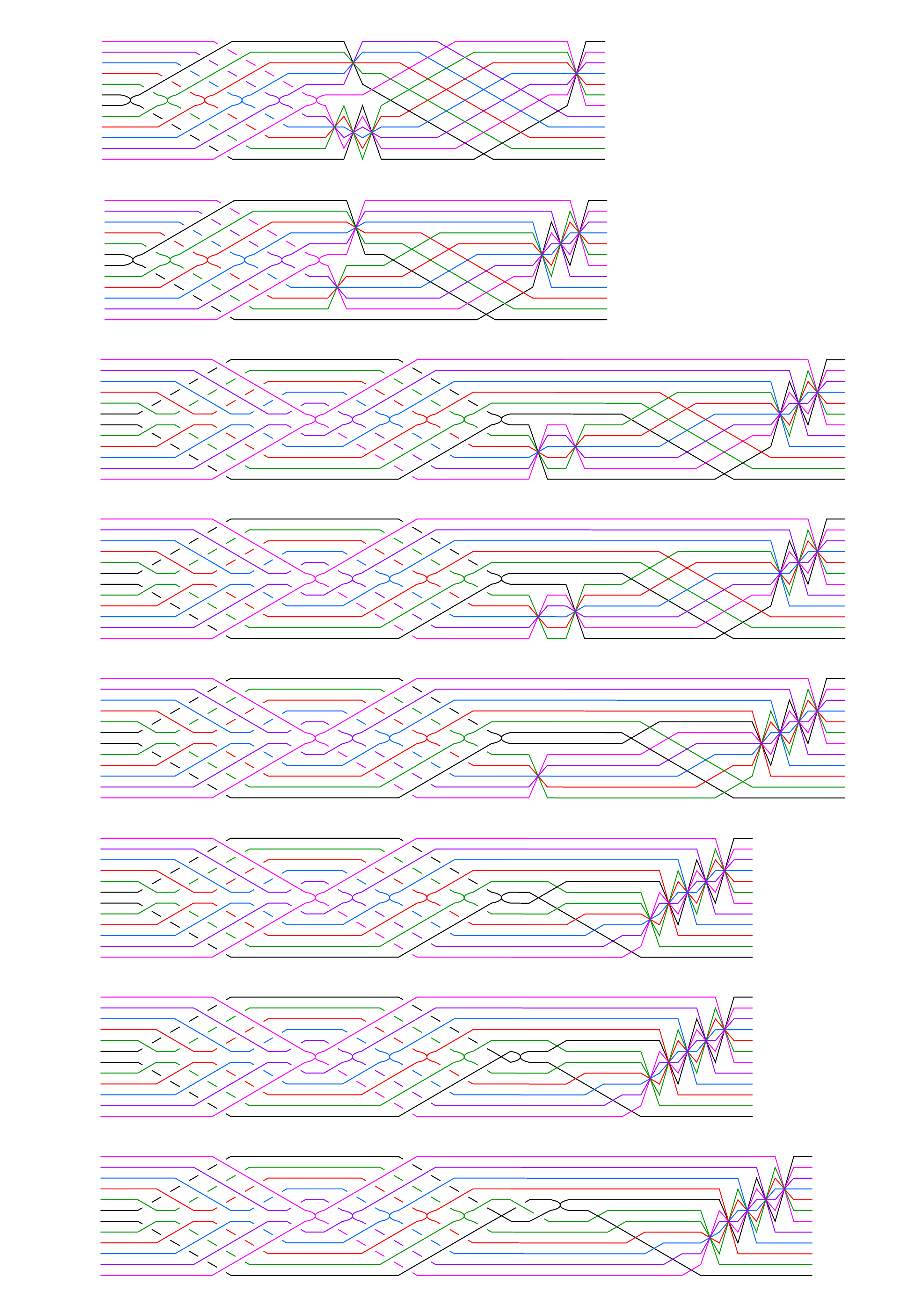}}%
    \put(0.0214114,1.30719333){\color[rgb]{0,0,0}\makebox(0,0)[lt]{\lineheight{1.25}\smash{\begin{tabular}[t]{l}$(a)$\end{tabular}}}}%
    \put(0.0214114,1.13344788){\color[rgb]{0,0,0}\makebox(0,0)[lt]{\lineheight{1.25}\smash{\begin{tabular}[t]{l}$(b)$\end{tabular}}}}%
    \put(0.0214114,0.95922914){\color[rgb]{0,0,0}\makebox(0,0)[lt]{\lineheight{1.25}\smash{\begin{tabular}[t]{l}$(c)$\end{tabular}}}}%
    \put(0.0214114,0.78548373){\color[rgb]{0,0,0}\makebox(0,0)[lt]{\lineheight{1.25}\smash{\begin{tabular}[t]{l}$(d)$\end{tabular}}}}%
    \put(0.0179409,0.61221158){\color[rgb]{0,0,0}\makebox(0,0)[lt]{\lineheight{1.25}\smash{\begin{tabular}[t]{l}$(e)$\end{tabular}}}}%
    \put(0.02234967,0.43846614){\color[rgb]{0,0,0}\makebox(0,0)[lt]{\lineheight{1.25}\smash{\begin{tabular}[t]{l}$(f)$\end{tabular}}}}%
    \put(0.0214114,0.26424733){\color[rgb]{0,0,0}\makebox(0,0)[lt]{\lineheight{1.25}\smash{\begin{tabular}[t]{l}$(g)$\end{tabular}}}}%
    \put(0.0214114,0.09050188){\color[rgb]{0,0,0}\makebox(0,0)[lt]{\lineheight{1.25}\smash{\begin{tabular}[t]{l}$(h)$\end{tabular}}}}%
  \end{picture}%
\endgroup%

\caption{The sequence of moves.}
\label{fig:ProofMoves}
\end{figure}
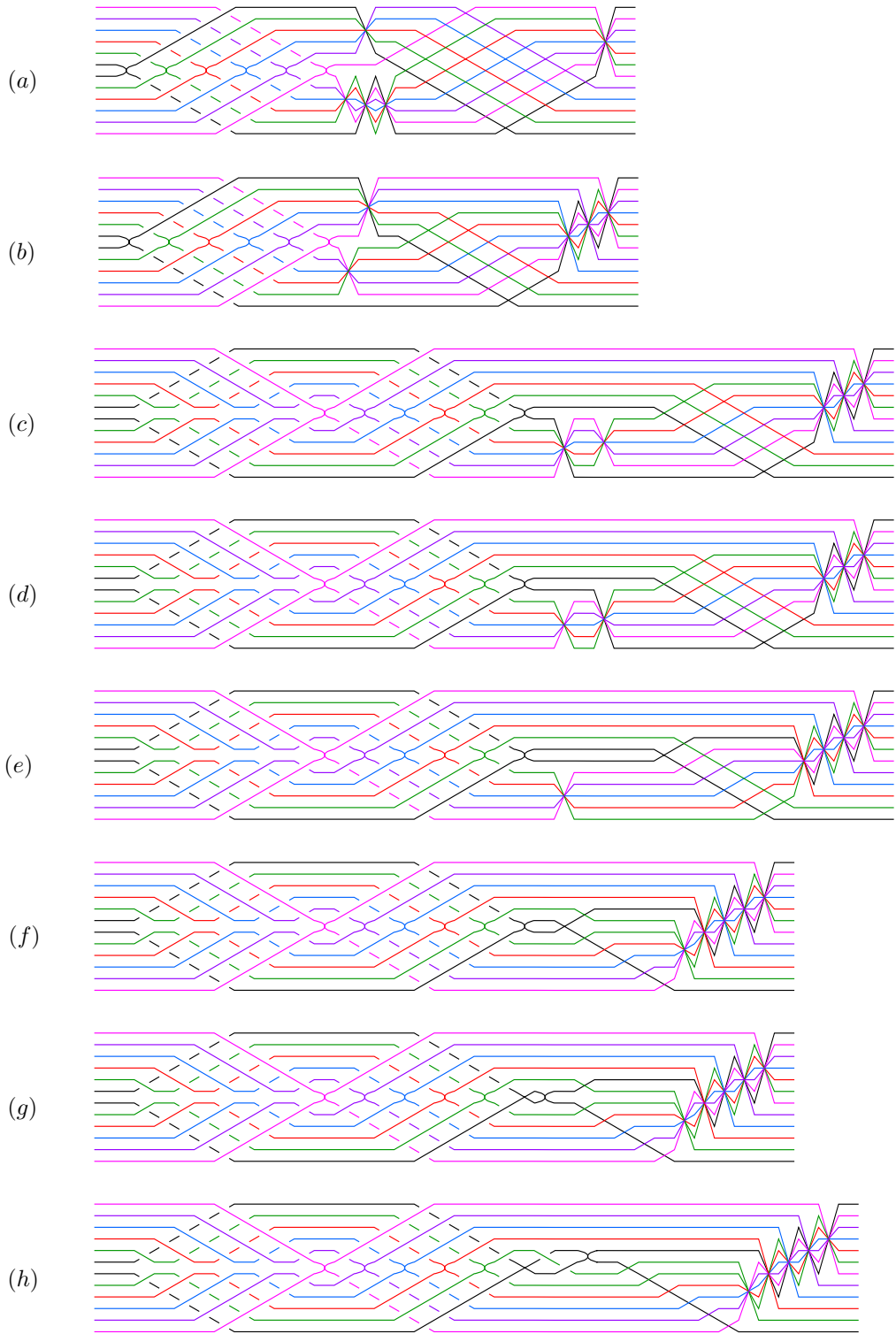

By Proposition~\ref{prop:scott-pics}, the Scott deformation is given by
the diagram  
$$
\TN \I^{m+1}_{ 2m-1} (\I^{m+1}_{ 2m})^{m-4} \I^1_{2m}
$$
on $2m$ strands, with $m=k+7$ even. Use~\eqref{eq:split-many} to split the multipoint on the right as 
$\I^1_{2m}= \I^1_{ m-1} \X^{1, m-1}_{m, 2m} \I^1_{ m+1}$, which gives 
$$
\TN \I^{m+1}_{ 2m-1} (\I^{m+1}_{ 2m})^{m-4} \I^1_{ m-1} \X^{1, m-1}_{m, 2m} \I^1_{ m+1},$$ Figure~\ref{fig:ProofMoves}(a). For the next step, combine the multipoint $\I^1_{ m-1}$ 
with the adjacent strand of double points (the pink strand in Figure~\ref{fig:ProofMoves}(a)): expand 
$\X^{1, m-1}_{m, 2m}= \X^{1,m-1}_m \X^{2,m}_{m+1, 2m}$ by~\eqref{eq:expandX} and then replace
$\I^1_{ m-1} \X^{1, m-1}_m = \I^1_{m}$ by~\eqref{eq:split-one} to get 
$$
\TN \I^{m+1}_{ 2m-1} (\I^{m+1}_{ 2m})^{m-4} \I^1_{ m} \X^{2,m}_{m+1, 2m} \I^1_{ m+1},
$$ which equals 
$\TN \I^{m+1}_{2m-1}\I^{1}_m (\I^{m+1}_{2m})^{m-4}  \X^{2,m}_{m+1, 2m} \I^1_{m+1}$
because the corresponding multipoints commute.  Then slide the $(m-4)$ bottom multipoints
 $(\I^{m+1}_{ 2m})^{m-4}$ up and combine them, one by one, with the double points on strands $2, \dots, m-3$ going down:
the two multipoints in Figure~\ref{fig:ProofMoves}(a) are combined with the purple and blue strands of double points, respectively. More formally, expanding $\X^{2,m}_{m+1, 2m}= \X^{3, m}_{m+1, 2m} \X^2_{3, m+2}$, 
then moving a multipoint past parallel strands by~\eqref{eq:move-multipts-lines}, and merging it with double points by \eqref{eq:split-one} gives

$$
\I^{m+1}_{ 2m} \X^{2,m}_{m+1, 2m} =_\d \I^{m+1}_{ 2m} \X^{3, m}_{m+1, 2m} \X^2_{3, m+2} =_\d 
\X^{3, m}_{m+1, 2m} \I^3_{ m+2} \X^2_{3, m+2} =_\d \X^{3, m}_{m+1, 2m} \I^2_{ m+2},
$$
which takes care of the rightmost multipoint. We then repeat this procedure inductively:
\begin{equation*}
\begin{split}
& \TN  \I^{m+1}_{ 2m-1}\I^1_{ m} (\I^{m+1}_{ 2m})^{m-4}  \X^{2,m}_{m+1, 2m} \I^1_{ m+1} =_\d 
\TN \I^{m+1}_{ 2m-1}\I^1_{ m} (\I^{m+1}_{ 2m})^{m-5}  \X^{3,m}_{m+1, 2m}  \I^2_{ m+2} \I^1_{ m+1} =_\d \\
& \TN  \I^{m+1}_{ 2m-1}\I^1_{ m} (\I^{m+1}_{ 2m})^{m-5} 
\X^{4,m}_{m+1, 2m} \X^3_{4, m+3} \I^2_{ m+2} \I^1_{ m+1} =_\d \\ 
& \TN  \I^{m+1}_{ 2m-1}\I^1_{ m} (\I^{m+1}_{ 2m})^{m-6} 
\X^{4,m}_{m+1, 2m}   \I^4_{ m+3} \X^3_{4, m+3} \I^2_{ m+2} \I^1_{ m+1} =_\d \\  
 & \TN  \I^{m+1}_{ 2m-1}\I^1_{ m} (\I^{m+1}_{2m})^{m-6} 
\X^{4,m}_{m+1, 2m}   \I^3_{ m+3} \I^2_{ m+2} \I^1_{ m+1}=_\d \dots =_\d \\
& \TN   \I^{m+1}_{ 2m-1}\I^1_{ m}   \X^{m-2,m}_{m+1, 2m}    \I^{m-3}_{ 2m-3} \dots \I^2_{m+2} \I^1_{ m+1}.
\end{split}
\end{equation*}
The resulting diagram is in~Figure~\ref{fig:ProofMoves}(b). The multipoints 
$\I^1_{ m}$ and  $\I^{m+1}_{ 2m-1}$ commute, so we rewrite
$$
\TN \I^1_{ m} \I^{m+1}_{ 2m-1}  \X^{m-2,m}_{m+1, 2m}    \I^{m-3}_{ 2m-3} \dots \I^2_{ m+2} \I^1_{ m+1}
$$
and use~\eqref{eq:move-multipt-tang} to slide $\I^1_m$ down past the tangency nest to get 
$$
   \Delta_{1,m} \Delta_{m+1, 2m}^{-1} \TN \I^{m+1}_{2m}    \I^{m+1}_{ 2m-1}  \X^{m-2,m}_{m+1, 2m}    \I^{m-3}_{ 2m-3} \dots \I^2_{ m+2} \I^1_{ m+1},
$$
shown in~Figure~\ref{fig:ProofMoves}(c). Next, we move strand $(m+1)$ from 
$\I^{m+1}_{ 2m}$  to  $\I^{m+1}_{ 2m-1}$ by  split and then merge moves~\eqref{eq:split-one}, 
$$
\I^{m+1}_{ 2m}    \I^{m+1}_{ 2m-1} =_\d  \I^{m+2}_{ 2m} \X^{m+1}_{m+2, 2m} \I^{m+1}_{ 2m-1}  =_\d \I^{m+2}_{ 2m}    \I^{m+1}_{ 2m}. 
$$
The result is the diagram in Figure~\ref{fig:ProofMoves}(d),
$$
  \Delta_{1,m} \Delta_{m+1, 2m}^{-1} \TN \I^{m+2}_{ 2m} 
  \I^{m+1}_{ 2m}  \X^{m-2,m}_{m+1, 2m}    \I^{m-3}_{ 2m-3} \dots \I^2_{ m+2} \I^1_{ m+1}.
$$
Now these two multipoints at the bottom can be moved up and combined with strands of double points (red and green strands in the figure).   Namely, expand 
$\X^{m-2,m}_{m+1, 2m} =  \X^{m-1, m}_{m+1, 2m} \X^{m-2}_{m-1, 2m-2}$ by~\eqref{eq:expandX},  move the multipoint by~\eqref{eq:move-multipts-lines}, and merge with double points by~\eqref{eq:split-one},  
$$
\I^{m+1}_{ 2m}  \X^{m-2,m}_{m+1, 2m} =_\d \I^{m+1}_{ 2m} \X^{m-1, m}_{m+1, 2m} \X^{m-2}_{m-1, 2m-2} =_\d \X^{m-1, m}_{m+1, 2m}  \I^{m-1}_{ 2m-2} \X^{m-2}_{m-1, 2m-2} =_\d 
\X^{m-1, m}_{m+1, 2m} \I^{m-2}_{ 2m-2},
$$
producing the diagram 
$$
\Delta_{1,m} \Delta_{m+1, 2m}^{-1} \TN \I^{m+2}_{ 2m} \X^{m-1, m}_{m+1, 2m} \I^{m-2}_{ 2m-2}  \I^{m-3}_{ 2m-3} \dots \I^2_{ m+2} \I^1_{ m+1}
$$
in Figure~\ref{fig:ProofMoves}(e). Then similarly move up $\I^{m+2}_{ 2m}$ and combine with double points to get the diagram  
\begin{equation*}
 \begin{split}
&\Delta_{1,m} \Delta_{m+1, 2m}^{-1} \TN   \X^{m}_{m+1, 2m} \I^{m-1}_{m}  \I^m_{2m-1}   \I^{m-2}_{ 2m-2}  \I^{m-3}_{ 2m-3}\dots \I^2_{ m+2} \I^1_{ m+1} =_\d  \\ 
& \Delta_{1,m} \Delta_{m+1, 2m}^{-1} \TN  \I^m_{m+1} \X^{m+1}_{m+2, 2m} \I^{m-1}_{m}  \I^m_{ 2m-1}   \I^{m-2}_{2m-2}  \I^{m-3}_{ 2m-3} \dots \I^2_{ m+2} \I^1_{ m+1}
\end{split}
\end{equation*}
in Figure~\ref{fig:ProofMoves}(f). In this last figure, the strands $m$ and $m+1$ 
have a tangency and an adjacent intersection, and we can use move~\eqref{eq:swap-tang-crossing}. To write this in formulas, separate the innermost tangency from the tangency nest:
$$
\TN =_\d \TN_{2, 2m-1} \sigma_1^{-1} \dots \sigma_{m-1}^{-1} \sigma_{2m-1}^{-1} \dots \sigma_{m+1}^{-1} \T_m =_\d \TN_{2, 2m-1}
\sigma_{top}^{-1}  \sigma_{bot}^{-1} \T_m,
$$
where we introduced notation $\sigma_{top}^{-1}$ and   $\sigma_{bot}^{-1}$ to shorten the formulas. Then after move~\eqref{eq:swap-tang-crossing}, we get the diagram

$$
\Delta_{1,m} \Delta_{m+1, 2m}^{-1}   \TN_{2, 2m-1}
\sigma_{top}^{-1}  \sigma_{bot}^{-1}  \I^m_{m+1} \T_m  \X^{m+1}_{m+2, 2m}  \I^{m-1}_{m}  \I^m_{2m-1}   \I^{m-2}_{2m-2}  \I^{m-3}_{2m-3} \dots \I^2_{m+2} \I^1_{m+1} 
$$
shown in Figure~\ref{fig:ProofMoves}(g). The elements $\X^{m+1}_{m+2, 2m}$ and  
$\I^{m-1}_{m}$ commute, so we rewrite  
$\T_m  \X^{m+1}_{m+2, 2m} \I^{m-1}_{m} =_\d \T_m \I^{m-1}_{m} \X^{m+1}_{m+2, 2m}$ and then 
slide the intersection point $\I^{m-1}_{m}$ down past the tangency $\T_m$ by~\eqref{eq:move-doublept-tang} to get the diagram 
\begin{equation*}
\begin{split}
& \Delta_{1,m} \Delta_{m+1, 2m}^{-1}   \TN_{2, 2m-1}
\sigma_{top}^{-1}  \sigma_{bot}^{-1}   \I^m_{m+1} \sigma_{m-1} \sigma_{m}^{-1} \T_{m-1}\I^m_{m+1}  \X^{m+1}_{m+2, 2m} \\
 & \hspace{3in} \I^m_{2m-1}  \I^{m-2}_{ 2m-2}  \I^{m-3}_{ 2m-3} \dots \I^2_{ m+2} \I^1_{ m+1} =_\d \\
 & \Delta_{1,m} \Delta_{m+1, 2m}^{-1}   \TN_{2, 2m-1}
\sigma_{top}^{-1}  \sigma_{bot}^{-1}  \I^m_{m+1} \sigma_{m-1} \sigma_{m}^{-1} \T_{m-1} \X^{m}_{m+1, 2m} \I^m_{2m-1}  \I^{m-2}_{ 2m-2}  \I^{m-3}_{ 2m-3} \dots \I^2_{m+2} \I^1_{m+1}
\end{split}
\end{equation*}
in Figure~\ref{fig:ProofMoves}(h). Finally, use~\eqref{eq:split-one} to combine  the remaining strand of double points with the adjacent multipoint, $\X^{m}_{m+1, 2m} \I^m_{ 2m-1} = \I^m_{ 2m}$, which gives the diagram
$$
\Delta_{1,m} \Delta_{m+1, 2m}^{-1}   \TN_{2, 2m-1}
\sigma_{top}^{-1}  \sigma_{bot}^{-1}  \I^m_{ m+1} \sigma_{m-1} \sigma_{m}^{-1} \T_{m-1} \I^m_{ 2m}  \I^{m-2}_{2m-2}  \I^{m-3}_{ 2m-3} \dots \I^2_{m+2} \I^1_{m+1}
$$
shown in Figure~\ref{fig:QHD0}. When all intersection points are marked, this final diagram satisfies the combinatorial condition~\eqref{eq:lines-points}, as desired,  and therefore produces a $\Q$HD filling via the DJVS construction. 
The diagram can be slightly simplified to 
$$\TN_{2, 2m-1}
\sigma_{top}^{-1}  \sigma_{bot}^{-1}  \I^m_{m+1} \sigma_{m-1} \sigma_{m}^{-1} \T_{m-1} \I^m_{2m}  \I_{2m-2}^{m-2}  \I^{m-3}_{2m-3} \dots \I^2_{ m+2} \I^1_{ m+1}
$$
by removing the braiding in the beginning, see subsection~\ref{ss:braid-at-edge}.
\begin{figure}[ht]
\includegraphics[scale=.4]{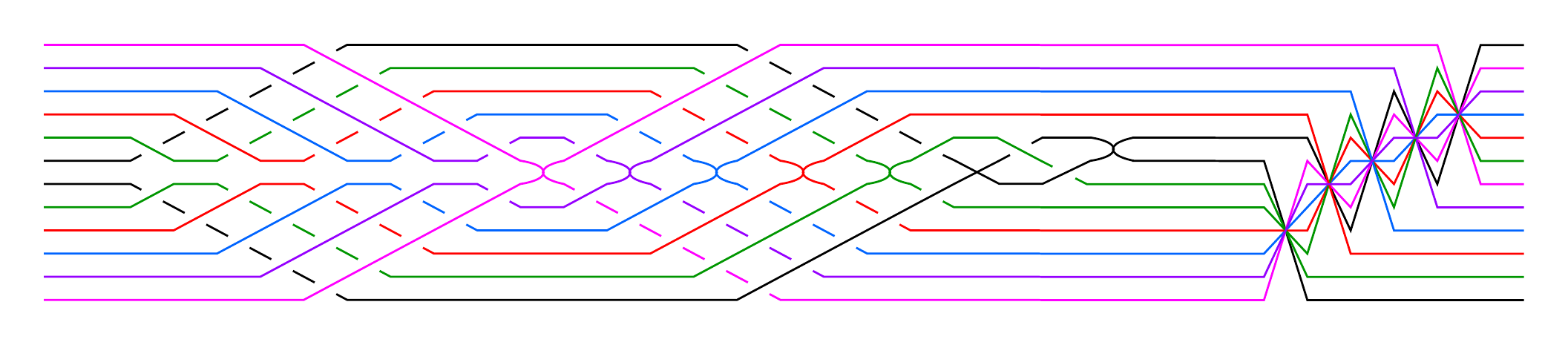}
\caption{The pattern for arrangements producing a $\Q$HD filling for $k$ odd ($k=-1$ is shown).}
\label{fig:QHD0}
\end{figure}
 
For $k$ even, a DJVS arrangement with the desired properties is shown in Figure~\ref{fig:QHD1}. We omit the moves for this case, the final factorization is as follows:
$$
\TN_{2, 2m-1}\sigma_{bot}^{-1}\sigma_{top}^{-1}\I^m_{m+1}\sigma_{m+1}\sigma_{m}^{-1}\T_{m+1}\I^1_{m+1}\I^3_{m+3}\dots\I_{2m}^m.$$
As for $k$ odd, one of the moves slides an intersection point past a tangency nest,  adding some braiding at the beginning of the diagram. To simplify the picture, we gave an equivalent diagram with this braiding removed. 
\begin{figure}[ht]
\includegraphics[scale=.3]{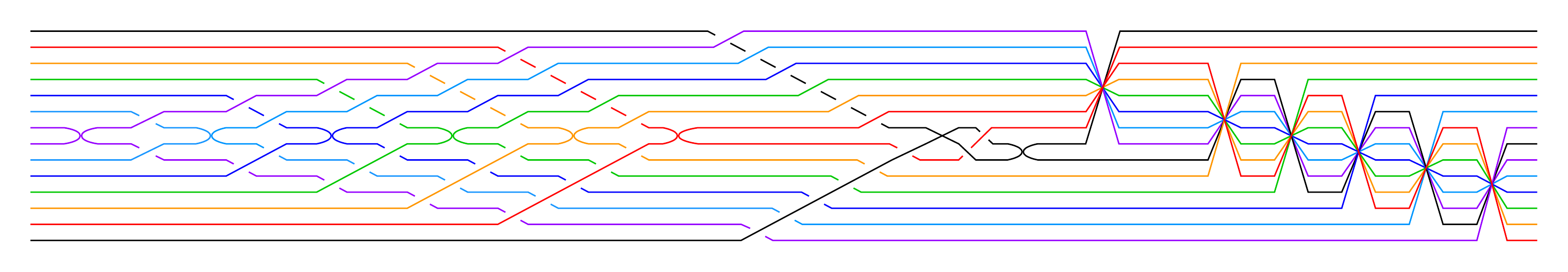}
\caption{The pattern for arrangements producing 
a $\Q$HD filling for $k$ even ($k=0$ is shown).}
\label{fig:QHD1}
\end{figure}
 \end{proof}

\section{Constructing exotic 4-manifolds with new rational blowdowns}

Our new rational blowdown can in principle be used to create exotic $4$-manifolds in the same way as the classical rational blowdown, although we have yet to  find any interesting exotica that cannot be constructed by classical means. To give a very basic example, we show how the configuration of spheres corresponding to our blowdown graph can be found in a blowup of $E(3)$.
We will use an elliptic fibration that has an $I_7$ fiber; recall that this fiber is formed by a cyclic chain of seven 2-spheres, $S_1, \dots, S_7$, such that $S_i$ and $S_j$ intersect transversely at one point iff $i=j \pm 1 \mod 7$, and are disjoint for distinct $i$, $j$ otherwise. The configuration corresponding to $G_{0,1}$ will be formed by (the strict transforms of) a linear chain of 5 spheres $S_1, \dots S_5$ in the $I_7$ fiber together with two disjoint sections, one  of them intersecting $S_3$ and the other $S_4$ at a generic point. Additionally, we will need a third section disjoint from the above configuration (and intersecting $S_7$ at a generic point). This is provided by the next lemma:

\begin{lemma} \label{lm:E3} There is an elliptic fibration $E(3)$ that contains a configuration of smoothly embedded spheres shown in Figure~\ref{fig:E3}: a necklace of seven $(-2)$ spheres forming an $I_7$ fiber, and two disjoint sections that intersect 
two adjacent spheres in the necklace. Further, we can find a fibration as above with an $I_2$ fiber and a third section that is disjoint from the other two and intersects the $I_7$ fiber at another prescribed sphere.

The spheres in the configuration can be assumed to be symplectic with respect to a compatible symplectic structure on $E(3)$.  
\end{lemma}

\begin{proof} Consider a twice-punctured torus with two meridians $\alpha_1$ and $\alpha_2$, a longitude $\beta$, and the two boundary parallel curves $\delta_1$ and $\delta_2$, as in Figure~\ref{fig:2holetorus}. For brevity, we will use the same notation for the curve and for the positive Dehn twist around this curve. By~\cite[Section 3.2]{KorOz}, we have a relation  
$$
\delta_1 \delta_2 = (\alpha_1 \alpha_2 \beta)^4
$$
in the mapping class group of the $2$-holed torus.

\begin{figure}[ht]
	\def\svgwidth{0,25\columnwidth}
	\centering
	%
\begingroup%
  \makeatletter%
  \providecommand\color[2][]{%
    \errmessage{(Inkscape) Color is used for the text in Inkscape, but the package 'color.sty' is not loaded}%
    \renewcommand\color[2][]{}%
  }%
  \providecommand\transparent[1]{%
    \errmessage{(Inkscape) Transparency is used (non-zero) for the text in Inkscape, but the package 'transparent.sty' is not loaded}%
    \renewcommand\transparent[1]{}%
  }%
  \providecommand\rotatebox[2]{#2}%
  \newcommand*\fsize{\dimexpr\f@size pt\relax}%
  \newcommand*\lineheight[1]{\fontsize{\fsize}{#1\fsize}\selectfont}%
  \ifx\svgwidth\undefined%
    \setlength{\unitlength}{241.45591916bp}%
    \ifx\svgscale\undefined%
      \relax%
    \else%
      \setlength{\unitlength}{\unitlength * \real{\svgscale}}%
    \fi%
  \else%
    \setlength{\unitlength}{\svgwidth}%
  \fi%
  \global\let\svgwidth\undefined%
  \global\let\svgscale\undefined%
  \makeatother%
  \begin{picture}(1,1.35219418)%
    \lineheight{1}%
    \setlength\tabcolsep{0pt}%
    \put(0,0){\includegraphics[width=\unitlength,page=1]{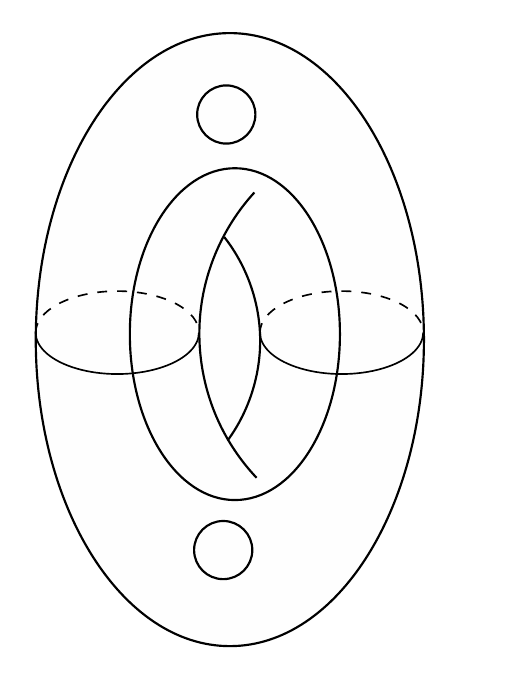}}%
    \put(0.53040773,1.11528166){\color[rgb]{0,0,0}\makebox(0,0)[lt]{\lineheight{1.25}\smash{\begin{tabular}[t]{l}$\delta_1$\end{tabular}}}}%
    \put(0.52628242,0.24651711){\color[rgb]{0,0,0}\makebox(0,0)[lt]{\lineheight{1.25}\smash{\begin{tabular}[t]{l}$\delta_2$\\\end{tabular}}}}%
    \put(0.7023796,0.54001226){\color[rgb]{0,0,0}\makebox(0,0)[lt]{\lineheight{1.25}\smash{\begin{tabular}[t]{l}$\alpha_2$\end{tabular}}}}%
    \put(0.11951451,0.52829135){\color[rgb]{0,0,0}\makebox(0,0)[lt]{\lineheight{1.25}\smash{\begin{tabular}[t]{l}$\alpha_1$\end{tabular}}}}%
    \put(0.69467956,0.90142541){\color[rgb]{0,0,0}\makebox(0,0)[lt]{\lineheight{1.25}\smash{\begin{tabular}[t]{l}$\beta$\end{tabular}}}}%
    \put(0,0){\includegraphics[width=\unitlength,page=2]{delta12.pdf}}%
  \end{picture}%
\endgroup%

	\caption{Curves $\alpha_1$, $\alpha_2$, $\beta$, $\delta_1$, $\delta_2$ on a 2-holed torus.} 
	\label{fig:2holetorus}
\end{figure}

Using relations $\alpha_i \beta \alpha_i = \beta \alpha_i \beta$ and $\alpha_1 \alpha_2 = \alpha_2 \alpha_1$, and writing $x^y = y^{-1}x y$ for conjugates, we have
\begin{equation*}
\begin{split}
 & \delta_1 \delta_2 = (\alpha_1 \alpha_2 \beta)^4 =
 \alpha_1 \alpha_2 \beta \alpha_1 \alpha_2 \beta \alpha_1 \underline{\alpha_2 \beta \alpha_2} \alpha_1 \beta =
 \alpha_1 \alpha_2 \beta \underline{\alpha_1 \alpha_2} \beta \alpha_1 \beta \alpha_2 \underline{\beta \alpha_1 \beta} =
  \alpha_1 \underline{\alpha_2 \beta \alpha_2} \alpha_1 \beta \alpha_1 \beta \alpha_2 \underline{\alpha_1 \beta \alpha_1}  = \\
 & \alpha_1 \beta \alpha_2 \underline{\beta \alpha_1 \beta}  \alpha_1 \beta \alpha_2 \alpha_1^2  \beta^{\alpha_1}= 
 \alpha_1 \beta \alpha_2 \alpha_1 \beta  \alpha_1^2 \underline{\beta \alpha_2 \alpha_1^2}  \beta^{\alpha_1} = 
 \alpha_1 \beta \alpha_2 \alpha_1 \underline{\beta \alpha_1^4} \beta^{\alpha_1^2} \alpha_2 \beta^{\alpha_1} = 
 \alpha_1 \underline{\beta \alpha_2   \alpha_1^5} \beta^{\alpha_1^4} \beta^{\alpha_1^2} \alpha_2 \beta^{\alpha_1} = \\
 & \alpha_1^6 \underline{\beta^{\alpha_1^5} \alpha_2}    \beta^{\alpha_1^4} \beta^{\alpha_1^2} \alpha_2 \beta^{\alpha_1} = \alpha_1^6 \alpha_2 (\beta^{\alpha_1^5})^{\alpha_2}  \beta^{\alpha_1^4} \beta^{\alpha_1^2} \alpha_2 \beta^{\alpha_1},
  \end{split}
\end{equation*}
with the subwords to be changed at each step underlined for ease of reading. Taking the cube of this relation and performing similar manipulations to move $\alpha_1$ and $\alpha_2$ at the cost of taking further conjugates of the Dehn twists around the longitudinal curve $\beta$, we obtain a relation of the form 
\begin{equation} \label{eq:delta12}
\alpha_1^{18} \alpha_2^{3}\, (\text{15 other positive Dehn twists}) = \delta_1^3 \delta_2^3.
\end{equation}
 In turn, this gives
 a relation  
$$
\alpha_1^{18} \alpha_2^{3}\, (\text{15 other positive Dehn twists})=1
$$
on the torus when the two holes are capped off, and therefore 
gives an elliptic fibration over $S^2$.  Treating each Dehn twist as a vanishing cycle for a Lefschetz singularity, we have a Lefschetz fibration on $E(3)$ with torus fiber and 36 nodal fibers.  The punctures in the original relation~\eqref{eq:delta12} 
correspond to sections of this fibration, with self-intersection $-3$.  Alternatively,
we can replace some of the Lefschetz singularities by more complicated singular fibers:
we isolate the subword  $\alpha_1^6 \alpha_2$ and think of it as the monodromy around a singular point corresponding to an $I_7$ fiber (see \cite[Table 1.9]{HarKir} or \cite[Section 2.1]{stipsicz07}).  
The $I_7$ fiber is given by the necklace of seven spheres obtained by collapsing $\alpha_2$ and six parallel 
copies of $\alpha_1$. By carefully keeping track of the vanishing cycles separating the intersection of the fiber with two sections (represented by the boundary components $\delta_1, \delta_2$) we see that for the chosen $I_7$ fiber, the two sections of $E(3)$ intersect it in adjacent spheres of the necklace of length 7.

To find a third section in the complement of the previous two, we can go through a similar argument leveraging the ``star relation'' 
$$
  \delta_1 \delta_2 \delta_3 =  (\alpha_1 \alpha_2 \alpha_3 \beta)^3  
$$
in the mapping class group of the torus with 3 holes, where $\delta_1$, $\delta_2$, $\delta_3$ are the boundary-parallel curves, and $\alpha_1$, $\alpha_2$, $\alpha_3$ are three meridians separating them,  \cite{Ger, KorOz}. Manipulating the cube of this relation, we can similarly construct an elliptic fibration on $E(3)$ with an $I_7$ fiber and three sections intersecting prescribed spheres in the $I_7$ necklace. Isolating an additional  $\alpha_1 \alpha_2$ term together with the factor $\alpha_1^6 \alpha_2$ used to create $I_7$, we can find the required $I_2$ fiber.



 By \cite{gompf2005locally}, \cite[Theorem~10.2.18]{gompf19994}, 
we can assume that the (components of) the fibers are symplectic, and since our fibration admits a section, a finite collection of sections can also be assumed to be symplectic.
\end{proof}

\begin{figure}[ht]
	\centering
    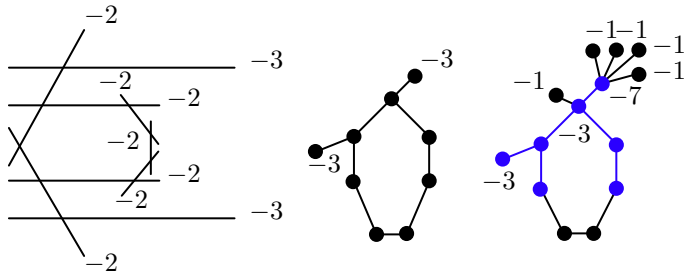

	\caption{Left: A configuration of embedded spheres in $E(3)$. Right: after the blowups, there is a subgraph  $G_{0,1}$. Unlabeled vertices are understood to have framing $-2$.} 
	\label{fig:E3}
\end{figure}

\begin{prop} The blowup $E(3) \# 5\cptwobar$ of the elliptic fibration from Lemma~\ref{lm:E3} contains a plumbing given by the graph $G_{0,1}$. Blowing down this graph produces an exotic copy of  $5 \cptwo \# 27 \cptwobar$.  
\end{prop}

\begin{proof} Using the dual graph given by the lemma, we blow up 4 times at one of the $(-3)$ vertices, and do one more blowup at a $(-2)$-vertex adjacent to the $(-3)$-vertex, as in the figure. The resulting configuration contains the graph $G_{0,1}$.   Since  this is a plumbing of symplectic disk bundles, we have the canonical contact structure induced on the boundary of the plumbing $G_{0,1}$. 
Surgering out this symplectic plumbed $4$-manifold and replacing it with the Stein rational homology disk filling $W$ provided by Theorem~\ref{thm:main}, we get a new closed symplectic manifold $X = X' \cup_{Y} W$, where $X'$ is the complement of the plumbing in the blowup of $E(3)$, and $Y$ is its boundary.  The intersection form of $X$ is odd: in the elliptic fibration constructed in Lemma~\ref{lm:E3}, we can find a section disjoint from all the components of the $G_{0,1}$ curve configuration; this section is unaffected by the surgery and thus gives a homology class in $X$ of self-intersection $-3$. We check that $X$ is simply connected. By construction, the filling $W$ is the complement of $\cup_i \nu(\Gamma_i)$ in the blowup of the 4-ball. Fixing a basepoint in $Y$, we have
$\pi_1(Y) \twoheadrightarrow \pi_1(W)$, because $\pi_1(W)$ is generated by the loops around the cocore disks of $\nu(\Gamma_i)$'s. It remains to check that $X'$ is simply connected. Since $X'$ is the complement of the tree plumbing of spheres in a simply connected manifold, $\pi_1(X')$ is normally generated by the loops that are boundaries of the normal disks to the spheres of the plumbing; moreover, it suffices to consider only the loops that correspond to the leaves of the tree. For our embedding of the configuration $G_{0,1}$, two of the leaves are the spheres of $I_7$ in the original elliptic fibration, and the other two correspond to the sections.  For the first pair of leaves, the loops that are boundaries of the normal disks to the spheres of $I_7$  are contractible because they bound disks in  the two punctured spheres given by the ``leftover'' part of the $I_7$ fiber (in the complement of our configuration). For the leaves that are sections, we need to check that the loops corresponding to the curves $\delta_1$ and $\delta_2$ in the $I_7$ fiber are contractible in $X'$. This follows from the fact that the curves $\alpha_1$ and $\alpha_2$ bound disks in $X'$, since by construction the elliptic fibration of Lemma~\ref{lm:E3} has an $I_2$ fiber where these curves in a nearby regular fiber are collapsed to points: using these disks, we can find a nullhomotopy for $\delta_1$ and $\delta_2$. 

   Since $\sigma(E(3))= -24$ and $b_2(E(3))=34$, after $5$ blowups and a blowdown of a negative definite $7$-vertex tree, we have $\sigma(X)= -22$ and $b_2(X)= 32$, so $b_2^+(X)=5$, $b_2^-(X)=27$.  It now follows that $X$ is homeomorphic to $5 \cptwo \# 27 \cptwobar$, but the two manifolds are not diffeomorphic because $X$ is symplectic, and  $5 \cptwo \# 27 \cptwobar$ has vanishing Seiberg--Witten invariant. \end{proof}

\section{$\Q$HD fillings for the two-parameter family}\label{s:twoparam}

In this section, we prove Theorem~\ref{thm:main} for the general case of the graphs $G_{k,n}$. We give more schematic diagrams with brief comments, but this should be sufficient as the steps are similar to those in Section 5. We only treat the case $k$ even; $k$ odd is similar.

The decorated germ corresponding to the graph $G_{k,n}$ has $n+k+6$ irreducible components, see Section~\ref{s:djvs}.    We set $s=n+k+6$ for brevity. The Scott deformation is obtained similarly to the previous examples, where the  curves corresponding to $C_1, \dots C_{k+6}$ are shown in blue, and the curves corresponding to $C'_1, \dots C'_n$ are shown in red. Assume $n\geq 2$, since the case $n=1$ was done previously.

{
		\fontsize{9pt}{5pt}
		\def\svgwidth{\textwidth}
		%
\begingroup%
  \makeatletter%
  \providecommand\color[2][]{%
    \errmessage{(Inkscape) Color is used for the text in Inkscape, but the package 'color.sty' is not loaded}%
    \renewcommand\color[2][]{}%
  }%
  \providecommand\transparent[1]{%
    \errmessage{(Inkscape) Transparency is used (non-zero) for the text in Inkscape, but the package 'transparent.sty' is not loaded}%
    \renewcommand\transparent[1]{}%
  }%
  \providecommand\rotatebox[2]{#2}%
  \newcommand*\fsize{\dimexpr\f@size pt\relax}%
  \newcommand*\lineheight[1]{\fontsize{\fsize}{#1\fsize}\selectfont}%
  \ifx\svgwidth\undefined%
    \setlength{\unitlength}{968.12284695bp}%
    \ifx\svgscale\undefined%
      \relax%
    \else%
      \setlength{\unitlength}{\unitlength * \real{\svgscale}}%
    \fi%
  \else%
    \setlength{\unitlength}{\svgwidth}%
  \fi%
  \global\let\svgwidth\undefined%
  \global\let\svgscale\undefined%
  \makeatother%
  \begin{picture}(1,0.29998002)%
    \lineheight{1}%
    \setlength\tabcolsep{0pt}%
    \put(0,0){\includegraphics[width=\unitlength,page=1]{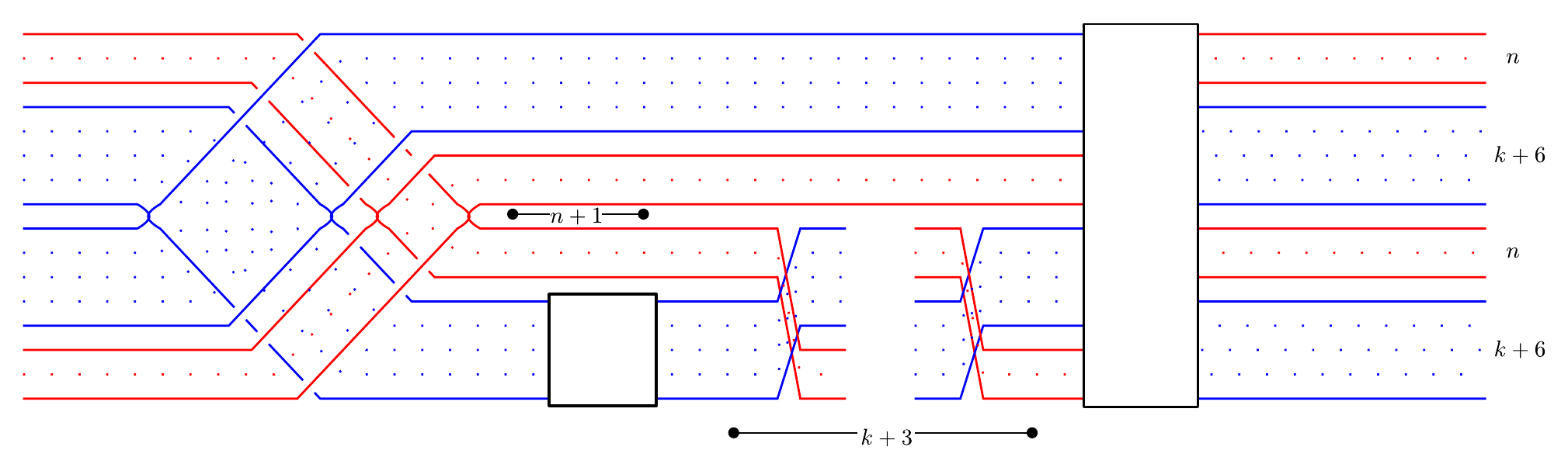}}%
    \put(0.71464799,0.15424256){\color[rgb]{0,0,0}\makebox(0,0)[lt]{\lineheight{1.25}\smash{\begin{tabular}[t]{l}$\I^1_{2s}$\end{tabular}}}}%
    \put(0,0){\includegraphics[width=\unitlength,page=2]{Familieskeven1.pdf}}%
    \put(0.48342664,0.08935115){\color[rgb]{0,0,0}\makebox(0,0)[lt]{\lineheight{1.25}\smash{\begin{tabular}[t]{l}$\I_{2s}^{s+1}$\end{tabular}}}}%
    \put(0.35611293,0.07286441){\color[rgb]{0,0,0}\makebox(0,0)[lt]{\lineheight{1.25}\smash{\begin{tabular}[t]{l}$\I_{2s}^{s+n+1}$\end{tabular}}}}%
    \put(0,0){\includegraphics[width=\unitlength,page=3]{Familieskeven1.pdf}}%
    \put(0.39263634,0.13369563){\color[rgb]{0,0,0}\makebox(0,0)[lt]{\lineheight{1.25}\smash{\begin{tabular}[t]{l}$\I^{s+1}_{s+n}$\end{tabular}}}}%
    \put(0,0){\includegraphics[width=\unitlength,page=4]{Familieskeven1.pdf}}%
    \put(0.31391151,0.13344944){\color[rgb]{0,0,0}\makebox(0,0)[lt]{\lineheight{1.25}\smash{\begin{tabular}[t]{l}$\I^{s+1}_{s+n}$\end{tabular}}}}%
    \put(0,0){\includegraphics[width=\unitlength,page=5]{Familieskeven1.pdf}}%
    \put(0.60009465,0.08935111){\color[rgb]{0,0,0}\makebox(0,0)[lt]{\lineheight{1.25}\smash{\begin{tabular}[t]{l}$\I_{2s}^{s+1}$\end{tabular}}}}%
  \end{picture}%
\endgroup%

}
In the  diagram of the Scott deformation, we use move \eqref{eq:split-many} to split the multipoint $\I^1_{2s}$:

{
	\fontsize{9pt}{5pt}
	\def\svgwidth{\textwidth}
	%
\begingroup%
  \makeatletter%
  \providecommand\color[2][]{%
    \errmessage{(Inkscape) Color is used for the text in Inkscape, but the package 'color.sty' is not loaded}%
    \renewcommand\color[2][]{}%
  }%
  \providecommand\transparent[1]{%
    \errmessage{(Inkscape) Transparency is used (non-zero) for the text in Inkscape, but the package 'transparent.sty' is not loaded}%
    \renewcommand\transparent[1]{}%
  }%
  \providecommand\rotatebox[2]{#2}%
  \newcommand*\fsize{\dimexpr\f@size pt\relax}%
  \newcommand*\lineheight[1]{\fontsize{\fsize}{#1\fsize}\selectfont}%
  \ifx\svgwidth\undefined%
    \setlength{\unitlength}{968.12283325bp}%
    \ifx\svgscale\undefined%
      \relax%
    \else%
      \setlength{\unitlength}{\unitlength * \real{\svgscale}}%
    \fi%
  \else%
    \setlength{\unitlength}{\svgwidth}%
  \fi%
  \global\let\svgwidth\undefined%
  \global\let\svgscale\undefined%
  \makeatother%
  \begin{picture}(1,0.29998003)%
    \lineheight{1}%
    \setlength\tabcolsep{0pt}%
    \put(0,0){\includegraphics[width=\unitlength,page=1]{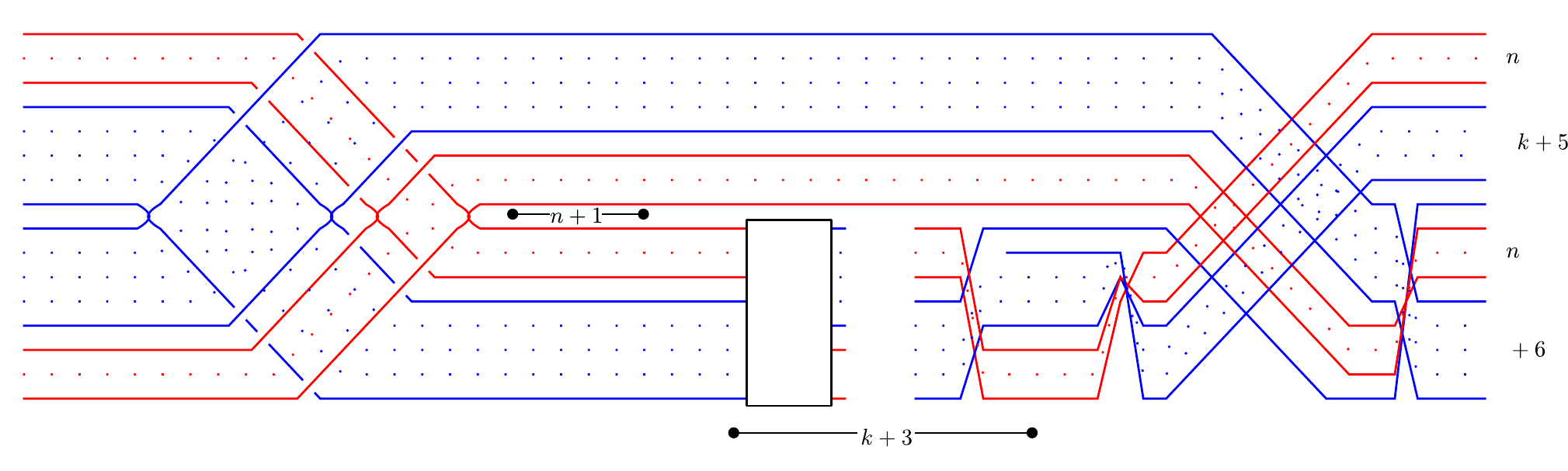}}%
    \put(0.48342813,0.08934886){\color[rgb]{0,0,0}\makebox(0,0)[lt]{\lineheight{1.25}\smash{\begin{tabular}[t]{l}$\I_{2s}^{s+1}$\end{tabular}}}}%
    \put(0,0){\includegraphics[width=\unitlength,page=2]{Familieskeven2.pdf}}%
    \put(0.60009619,0.08934886){\color[rgb]{0,0,0}\makebox(0,0)[lt]{\lineheight{1.25}\smash{\begin{tabular}[t]{l}$\I_{2s}^{s+1}$\end{tabular}}}}%
    \put(0,0){\includegraphics[width=\unitlength,page=3]{Familieskeven2.pdf}}%
    \put(0.68759645,0.08934886){\color[rgb]{0,0,0}\makebox(0,0)[lt]{\lineheight{1.25}\smash{\begin{tabular}[t]{l}$\I_{2s}^{s+2}$\end{tabular}}}}%
    \put(0,0){\includegraphics[width=\unitlength,page=4]{Familieskeven2.pdf}}%
    \put(0.37446228,0.07198021){\color[rgb]{0,0,0}\makebox(0,0)[lt]{\lineheight{1.25}\smash{\begin{tabular}[t]{l}$\I_{2s}^{s+n+1}$\end{tabular}}}}%
    \put(0,0){\includegraphics[width=\unitlength,page=5]{Familieskeven2.pdf}}%
    \put(0.31546419,0.13342059){\color[rgb]{0,0,0}\makebox(0,0)[lt]{\lineheight{1.25}\smash{\begin{tabular}[t]{l}$\I^{s+1}_{s+n}$\end{tabular}}}}%
    \put(0,0){\includegraphics[width=\unitlength,page=6]{Familieskeven2.pdf}}%
    \put(0.39509449,0.13342059){\color[rgb]{0,0,0}\makebox(0,0)[lt]{\lineheight{1.25}\smash{\begin{tabular}[t]{l}$\I^{s+1}_{s+n}$\end{tabular}}}}%
    \put(0,0){\includegraphics[width=\unitlength,page=7]{Familieskeven2.pdf}}%
    \put(0.88965257,0.10775857){\color[rgb]{0,0,0}\makebox(0,0)[lt]{\lineheight{1.25}\smash{\begin{tabular}[t]{l}$\I_{2s}^{s}$\end{tabular}}}}%
    \put(0,0){\includegraphics[width=\unitlength,page=8]{Familieskeven2.pdf}}%
  \end{picture}%
\endgroup%

}
Then use merge move~\eqref{eq:split-one} to combine the $\I^{s+2}_{2s}$ multipoint  with crossings of the (blue) strand $(s+1)$ passing above:

{
	\fontsize{9pt}{5pt}
	\def\svgwidth{\textwidth}
	%
\begingroup%
  \makeatletter%
  \providecommand\color[2][]{%
    \errmessage{(Inkscape) Color is used for the text in Inkscape, but the package 'color.sty' is not loaded}%
    \renewcommand\color[2][]{}%
  }%
  \providecommand\transparent[1]{%
    \errmessage{(Inkscape) Transparency is used (non-zero) for the text in Inkscape, but the package 'transparent.sty' is not loaded}%
    \renewcommand\transparent[1]{}%
  }%
  \providecommand\rotatebox[2]{#2}%
  \newcommand*\fsize{\dimexpr\f@size pt\relax}%
  \newcommand*\lineheight[1]{\fontsize{\fsize}{#1\fsize}\selectfont}%
  \ifx\svgwidth\undefined%
    \setlength{\unitlength}{968.12283325bp}%
    \ifx\svgscale\undefined%
      \relax%
    \else%
      \setlength{\unitlength}{\unitlength * \real{\svgscale}}%
    \fi%
  \else%
    \setlength{\unitlength}{\svgwidth}%
  \fi%
  \global\let\svgwidth\undefined%
  \global\let\svgscale\undefined%
  \makeatother%
  \begin{picture}(1,0.29998003)%
    \lineheight{1}%
    \setlength\tabcolsep{0pt}%
    \put(0,0){\includegraphics[width=\unitlength,page=1]{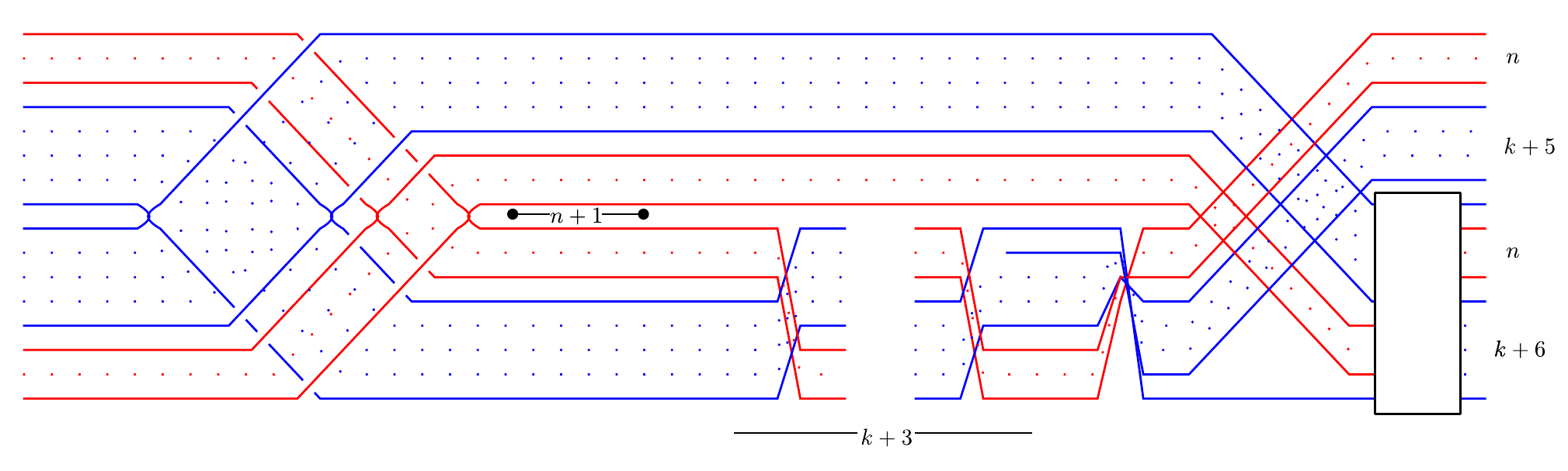}}%
    \put(0.88965258,0.10775857){\color[rgb]{0,0,0}\makebox(0,0)[lt]{\lineheight{1.25}\smash{\begin{tabular}[t]{l}$\I_{2s}^{s}$\end{tabular}}}}%
    \put(0,0){\includegraphics[width=\unitlength,page=2]{Familieskeven3.pdf}}%
    \put(0.70453282,0.09081236){\color[rgb]{0,0,0}\makebox(0,0)[lt]{\lineheight{1.25}\smash{\begin{tabular}[t]{l}$\I_{2s}^{s+1}$\end{tabular}}}}%
    \put(0,0){\includegraphics[width=\unitlength,page=3]{Familieskeven3.pdf}}%
    \put(0.47987125,0.08789656){\color[rgb]{0,0,0}\makebox(0,0)[lt]{\lineheight{1.25}\smash{\begin{tabular}[t]{l}$\I_{2s}^{s+1}$\end{tabular}}}}%
    \put(0,0){\includegraphics[width=\unitlength,page=4]{Familieskeven3.pdf}}%
    \put(0.61156942,0.09081236){\color[rgb]{0,0,0}\makebox(0,0)[lt]{\lineheight{1.25}\smash{\begin{tabular}[t]{l}$\I_{2s}^{s+1}$\end{tabular}}}}%
    \put(0,0){\includegraphics[width=\unitlength,page=5]{Familieskeven3.pdf}}%
    \put(0.37445832,0.07080776){\color[rgb]{0,0,0}\makebox(0,0)[lt]{\lineheight{1.25}\smash{\begin{tabular}[t]{l}$\I_{2s}^{s+n+1}$\end{tabular}}}}%
    \put(0,0){\includegraphics[width=\unitlength,page=6]{Familieskeven3.pdf}}%
    \put(0.30987053,0.13355789){\color[rgb]{0,0,0}\makebox(0,0)[lt]{\lineheight{1.25}\smash{\begin{tabular}[t]{l}$\I^{s+1}_{s+n}$\end{tabular}}}}%
    \put(0,0){\includegraphics[width=\unitlength,page=7]{Familieskeven3.pdf}}%
    \put(0.39043882,0.13342059){\color[rgb]{0,0,0}\makebox(0,0)[lt]{\lineheight{1.25}\smash{\begin{tabular}[t]{l}$\I^{s+1}_{s+n}$\end{tabular}}}}%
  \end{picture}%
\endgroup%

}
Note that $k+4$ is even and use move~\eqref{eq:switch-multipts} to switch multipoints, 
moving the $(k+4)$ copies of $\I^{s+1}_{2s}$ to the left: 

{
	\fontsize{9pt}{5pt}
	\def\svgwidth{\textwidth}
	%
\begingroup%
  \makeatletter%
  \providecommand\color[2][]{%
    \errmessage{(Inkscape) Color is used for the text in Inkscape, but the package 'color.sty' is not loaded}%
    \renewcommand\color[2][]{}%
  }%
  \providecommand\transparent[1]{%
    \errmessage{(Inkscape) Transparency is used (non-zero) for the text in Inkscape, but the package 'transparent.sty' is not loaded}%
    \renewcommand\transparent[1]{}%
  }%
  \providecommand\rotatebox[2]{#2}%
  \newcommand*\fsize{\dimexpr\f@size pt\relax}%
  \newcommand*\lineheight[1]{\fontsize{\fsize}{#1\fsize}\selectfont}%
  \ifx\svgwidth\undefined%
    \setlength{\unitlength}{968.12283325bp}%
    \ifx\svgscale\undefined%
      \relax%
    \else%
      \setlength{\unitlength}{\unitlength * \real{\svgscale}}%
    \fi%
  \else%
    \setlength{\unitlength}{\svgwidth}%
  \fi%
  \global\let\svgwidth\undefined%
  \global\let\svgscale\undefined%
  \makeatother%
  \begin{picture}(1,0.29998003)%
    \lineheight{1}%
    \setlength\tabcolsep{0pt}%
    \put(0,0){\includegraphics[width=\unitlength,page=1]{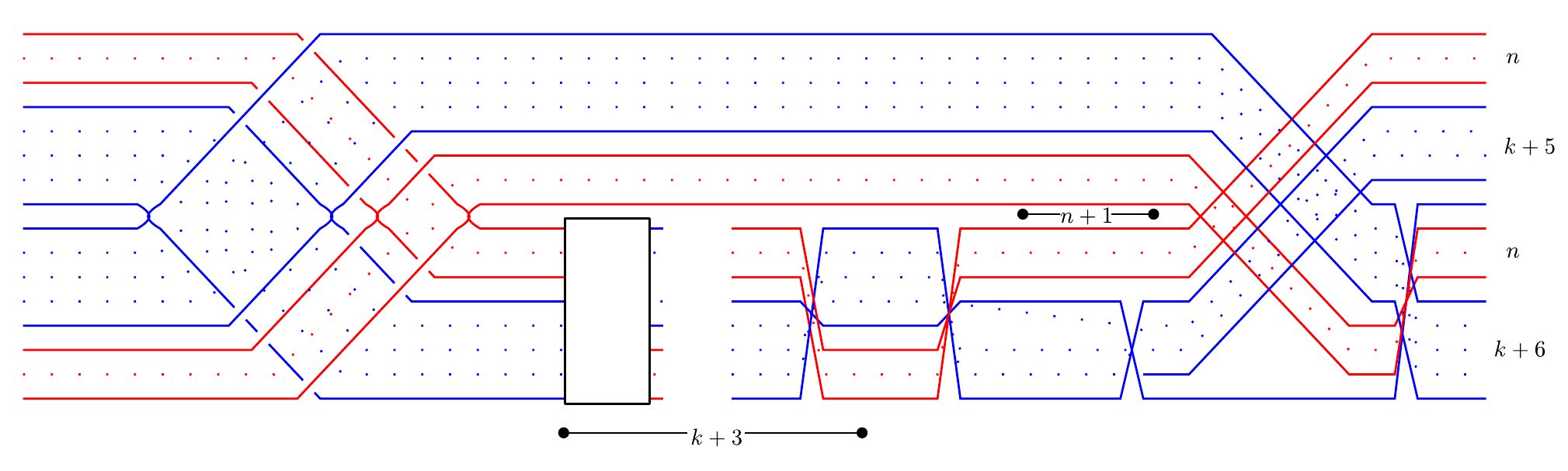}}%
    \put(0.36758696,0.09081236){\color[rgb]{0,0,0}\makebox(0,0)[lt]{\lineheight{1.25}\smash{\begin{tabular}[t]{l}$\I_{2s}^{s+1}$\end{tabular}}}}%
    \put(0,0){\includegraphics[width=\unitlength,page=2]{Familieskeven4.pdf}}%
    \put(0.49545813,0.09177003){\color[rgb]{0,0,0}\makebox(0,0)[lt]{\lineheight{1.25}\smash{\begin{tabular}[t]{l}$\I_{2s}^{s+1}$\end{tabular}}}}%
    \put(0,0){\includegraphics[width=\unitlength,page=3]{Familieskeven4.pdf}}%
    \put(0.57283467,0.09081236){\color[rgb]{0,0,0}\makebox(0,0)[lt]{\lineheight{1.25}\smash{\begin{tabular}[t]{l}$\I_{2s}^{s+1}$\end{tabular}}}}%
    \put(0,0){\includegraphics[width=\unitlength,page=4]{Familieskeven4.pdf}}%
    \put(0.64298941,0.13355789){\color[rgb]{0,0,0}\makebox(0,0)[lt]{\lineheight{1.25}\smash{\begin{tabular}[t]{l}$\I^{s+1}_{s+n}$\end{tabular}}}}%
    \put(0,0){\includegraphics[width=\unitlength,page=5]{Familieskeven4.pdf}}%
    \put(0.71322666,0.13348924){\color[rgb]{0,0,0}\makebox(0,0)[lt]{\lineheight{1.25}\smash{\begin{tabular}[t]{l}$\I^{s+1}_{s+n}$\end{tabular}}}}%
    \put(0,0){\includegraphics[width=\unitlength,page=6]{Familieskeven4.pdf}}%
    \put(0.68820983,0.07080776){\color[rgb]{0,0,0}\makebox(0,0)[lt]{\lineheight{1.25}\smash{\begin{tabular}[t]{l}$\I_{2s}^{s+n+1}$\end{tabular}}}}%
    \put(0,0){\includegraphics[width=\unitlength,page=7]{Familieskeven4.pdf}}%
    \put(0.8886755,0.10593092){\color[rgb]{0,0,0}\makebox(0,0)[lt]{\lineheight{1.25}\smash{\begin{tabular}[t]{l}$\I_{2s}^{s}$\end{tabular}}}}%
  \end{picture}%
\endgroup%

}
Then move the $(k+4)$  copies of $\I^{s+1}_{2s}$ through the tangency nest on the left, using move~\eqref{eq:move-multipt-tang} $k+4$ times (note the braiding $\Delta^{\pm(k+4)}$ on the left of the next diagram):

{
	\fontsize{9pt}{5pt}
	\def\svgwidth{\textwidth}
	%
\begingroup%
  \makeatletter%
  \providecommand\color[2][]{%
    \errmessage{(Inkscape) Color is used for the text in Inkscape, but the package 'color.sty' is not loaded}%
    \renewcommand\color[2][]{}%
  }%
  \providecommand\transparent[1]{%
    \errmessage{(Inkscape) Transparency is used (non-zero) for the text in Inkscape, but the package 'transparent.sty' is not loaded}%
    \renewcommand\transparent[1]{}%
  }%
  \providecommand\rotatebox[2]{#2}%
  \newcommand*\fsize{\dimexpr\f@size pt\relax}%
  \newcommand*\lineheight[1]{\fontsize{\fsize}{#1\fsize}\selectfont}%
  \ifx\svgwidth\undefined%
    \setlength{\unitlength}{968.12283325bp}%
    \ifx\svgscale\undefined%
      \relax%
    \else%
      \setlength{\unitlength}{\unitlength * \real{\svgscale}}%
    \fi%
  \else%
    \setlength{\unitlength}{\svgwidth}%
  \fi%
  \global\let\svgwidth\undefined%
  \global\let\svgscale\undefined%
  \makeatother%
  \begin{picture}(1,0.29998003)%
    \lineheight{1}%
    \setlength\tabcolsep{0pt}%
    \put(0,0){\includegraphics[width=\unitlength,page=1]{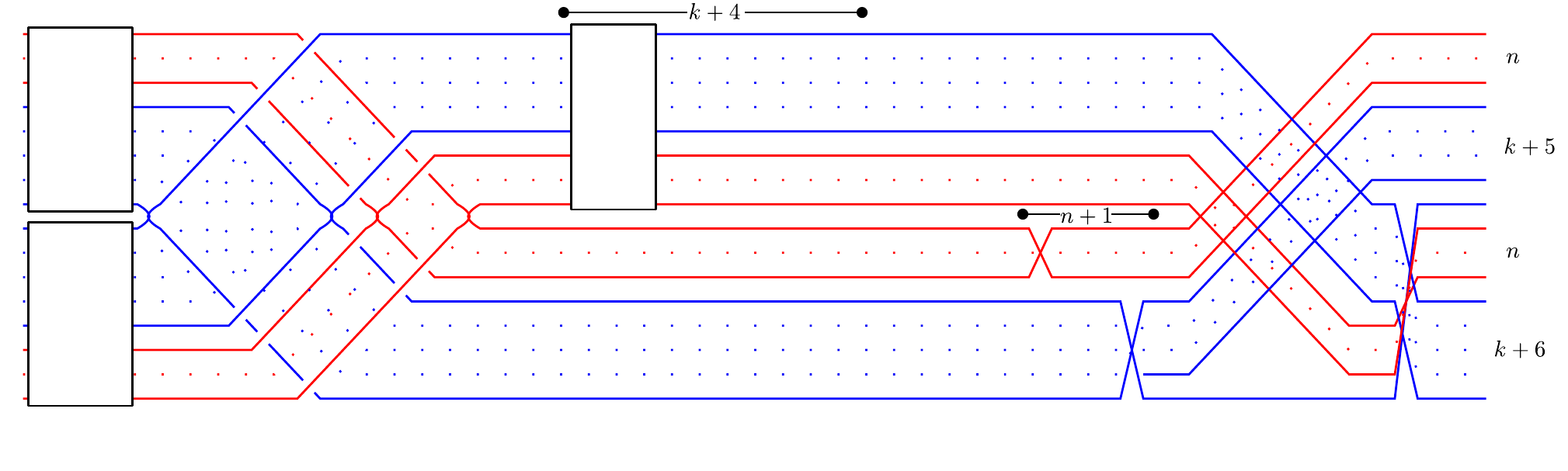}}%
    \put(0.37960059,0.22251052){\color[rgb]{0,0,0}\makebox(0,0)[lt]{\lineheight{1.25}\smash{\begin{tabular}[t]{l}$\I_s^1$\end{tabular}}}}%
    \put(0,0){\includegraphics[width=\unitlength,page=2]{Familieskeven5.pdf}}%
    \put(0.510881,0.22413798){\color[rgb]{0,0,0}\makebox(0,0)[lt]{\lineheight{1.25}\smash{\begin{tabular}[t]{l}$\I_s^1$\end{tabular}}}}%
    \put(0,0){\includegraphics[width=\unitlength,page=3]{Familieskeven5.pdf}}%
    \put(0.63911593,0.13355789){\color[rgb]{0,0,0}\makebox(0,0)[lt]{\lineheight{1.25}\smash{\begin{tabular}[t]{l}$\I^{s+1}_{s+n}$\end{tabular}}}}%
    \put(0,0){\includegraphics[width=\unitlength,page=4]{Familieskeven5.pdf}}%
    \put(0.71322666,0.13433258){\color[rgb]{0,0,0}\makebox(0,0)[lt]{\lineheight{1.25}\smash{\begin{tabular}[t]{l}$\I^{s+1}_{s+n}$\end{tabular}}}}%
    \put(0,0){\includegraphics[width=\unitlength,page=5]{Familieskeven5.pdf}}%
    \put(0.6920833,0.07080776){\color[rgb]{0,0,0}\makebox(0,0)[lt]{\lineheight{1.25}\smash{\begin{tabular}[t]{l}$\I_{2s}^{s+n+1}$\end{tabular}}}}%
    \put(0,0){\includegraphics[width=\unitlength,page=6]{Familieskeven5.pdf}}%
    \put(0.88965257,0.10775857){\color[rgb]{0,0,0}\makebox(0,0)[lt]{\lineheight{1.25}\smash{\begin{tabular}[t]{l}$\I_{2s}^{s}$\end{tabular}}}}%
    \put(0.02324085,0.21797601){\color[rgb]{0,0,0}\makebox(0,0)[lt]{\lineheight{1.25}\smash{\begin{tabular}[t]{l}$\Delta^{-k-4}$\end{tabular}}}}%
    \put(0.02243518,0.10177175){\color[rgb]{0,0,0}\makebox(0,0)[lt]{\lineheight{1.25}\smash{\begin{tabular}[t]{l}$\Delta^{k+4}$\end{tabular}}}}%
  \end{picture}%
\endgroup%

}
Using moves~\eqref{eq:move-multipts-lines} and~\eqref{eq:split-one}, move the $k+4$ multipoints 
$\I^1_{s}$ down past the crossings with red strands, and combine each multipoint with the double point crossings of one additional blue strand, forming multipoints $\I^{n+2}_{s+n+2}, \dots, \I^{s-1}_{2s-1}$: 

{
	\fontsize{9pt}{5pt}
	\def\svgwidth{\textwidth}
	%
\begingroup%
  \makeatletter%
  \providecommand\color[2][]{%
    \errmessage{(Inkscape) Color is used for the text in Inkscape, but the package 'color.sty' is not loaded}%
    \renewcommand\color[2][]{}%
  }%
  \providecommand\transparent[1]{%
    \errmessage{(Inkscape) Transparency is used (non-zero) for the text in Inkscape, but the package 'transparent.sty' is not loaded}%
    \renewcommand\transparent[1]{}%
  }%
  \providecommand\rotatebox[2]{#2}%
  \newcommand*\fsize{\dimexpr\f@size pt\relax}%
  \newcommand*\lineheight[1]{\fontsize{\fsize}{#1\fsize}\selectfont}%
  \ifx\svgwidth\undefined%
    \setlength{\unitlength}{968.12283325bp}%
    \ifx\svgscale\undefined%
      \relax%
    \else%
      \setlength{\unitlength}{\unitlength * \real{\svgscale}}%
    \fi%
  \else%
    \setlength{\unitlength}{\svgwidth}%
  \fi%
  \global\let\svgwidth\undefined%
  \global\let\svgscale\undefined%
  \makeatother%
  \begin{picture}(1,0.29998003)%
    \lineheight{1}%
    \setlength\tabcolsep{0pt}%
    \put(0,0){\includegraphics[width=\unitlength,page=1]{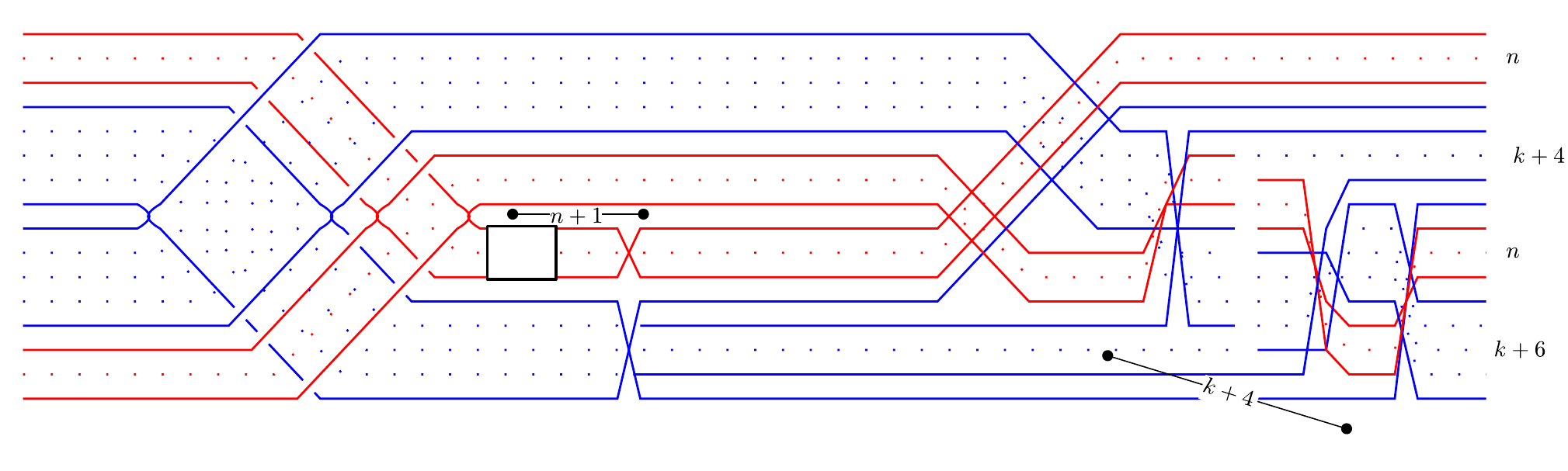}}%
    \put(0.31383334,0.13346856){\color[rgb]{0,0,0}\makebox(0,0)[lt]{\lineheight{1.25}\smash{\begin{tabular}[t]{l}$\I^{s+1}_{s+n}$\end{tabular}}}}%
    \put(0,0){\includegraphics[width=\unitlength,page=2]{Familieskeven6.pdf}}%
    \put(0.38742937,0.13346856){\color[rgb]{0,0,0}\makebox(0,0)[lt]{\lineheight{1.25}\smash{\begin{tabular}[t]{l}$\I^{s+1}_{s+n}$\end{tabular}}}}%
    \put(0,0){\includegraphics[width=\unitlength,page=3]{Familieskeven6.pdf}}%
    \put(0.3747134,0.07370826){\color[rgb]{0,0,0}\makebox(0,0)[lt]{\lineheight{1.25}\smash{\begin{tabular}[t]{l}$\I_{2s}^{s+n+1}$\end{tabular}}}}%
    \put(0,0){\includegraphics[width=\unitlength,page=4]{Familieskeven6.pdf}}%
    \put(0.89568485,0.10670875){\color[rgb]{0,0,0}\makebox(0,0)[lt]{\lineheight{1.25}\smash{\begin{tabular}[t]{l}$\I_{2s}^{s}$\end{tabular}}}}%
    \put(0,0){\includegraphics[width=\unitlength,page=5]{Familieskeven6.pdf}}%
    \put(0.81730331,0.12180016){\color[rgb]{0,0,0}\makebox(0,0)[lt]{\lineheight{1.25}\smash{\begin{tabular}[t]{l}$\I_{2s-1}^{s-1}$\end{tabular}}}}%
    \put(0,0){\includegraphics[width=\unitlength,page=6]{Familieskeven6.pdf}}%
    \put(0.72004868,0.14786549){\color[rgb]{0,0,0}\makebox(0,0)[lt]{\lineheight{1.25}\smash{\begin{tabular}[t]{l}$\I_{s+n+2}^{n+2}$\end{tabular}}}}%
    \put(0,0){\includegraphics[width=\unitlength,page=7]{Familieskeven6.pdf}}%
    \put(0.02324085,0.21797601){\color[rgb]{0,0,0}\makebox(0,0)[lt]{\lineheight{1.25}\smash{\begin{tabular}[t]{l}$\Delta^{-k-4}$\end{tabular}}}}%
    \put(0.02243518,0.10177175){\color[rgb]{0,0,0}\makebox(0,0)[lt]{\lineheight{1.25}\smash{\begin{tabular}[t]{l}$\Delta^{k+4}$\end{tabular}}}}%
  \end{picture}%
\endgroup%

}
Use move~\eqref{eq:move-multipts-lines} to slide $(n-1)$ copies of $\I^{s+1}_{s+n}$  to the right of the diagram, where they become multipoints $\I^1_n$.  Then use move~\eqref{eq:move-multipt-tang} to slide one of the two remaining copies of $\I^{s+1}_{s+n}$ up through the tangency nest, creating additional braiding on the red strands, then combine this multipoint with the second copy of  $\I^{s+1}_{s+n}$ and all the double point crossings on the red strands  to make $\I^{s+1-n}_{s+n}$ via move~\eqref{eq:split-many}. Slide multipoint $\I^{s+n+1}_{2s}$ on the blue strands to the left under the red strands, where it becomes 
$\I^{s+1}_{2s-n}$:

{
	\fontsize{9pt}{5pt}
	\def\svgwidth{\textwidth}
	%
\begingroup%
  \makeatletter%
  \providecommand\color[2][]{%
    \errmessage{(Inkscape) Color is used for the text in Inkscape, but the package 'color.sty' is not loaded}%
    \renewcommand\color[2][]{}%
  }%
  \providecommand\transparent[1]{%
    \errmessage{(Inkscape) Transparency is used (non-zero) for the text in Inkscape, but the package 'transparent.sty' is not loaded}%
    \renewcommand\transparent[1]{}%
  }%
  \providecommand\rotatebox[2]{#2}%
  \newcommand*\fsize{\dimexpr\f@size pt\relax}%
  \newcommand*\lineheight[1]{\fontsize{\fsize}{#1\fsize}\selectfont}%
  \ifx\svgwidth\undefined%
    \setlength{\unitlength}{969bp}%
    \ifx\svgscale\undefined%
      \relax%
    \else%
      \setlength{\unitlength}{\unitlength * \real{\svgscale}}%
    \fi%
  \else%
    \setlength{\unitlength}{\svgwidth}%
  \fi%
  \global\let\svgwidth\undefined%
  \global\let\svgscale\undefined%
  \makeatother%
  \begin{picture}(1,0.3003096)%
    \lineheight{1}%
    \setlength\tabcolsep{0pt}%
    \put(0,0){\includegraphics[width=\unitlength,page=1]{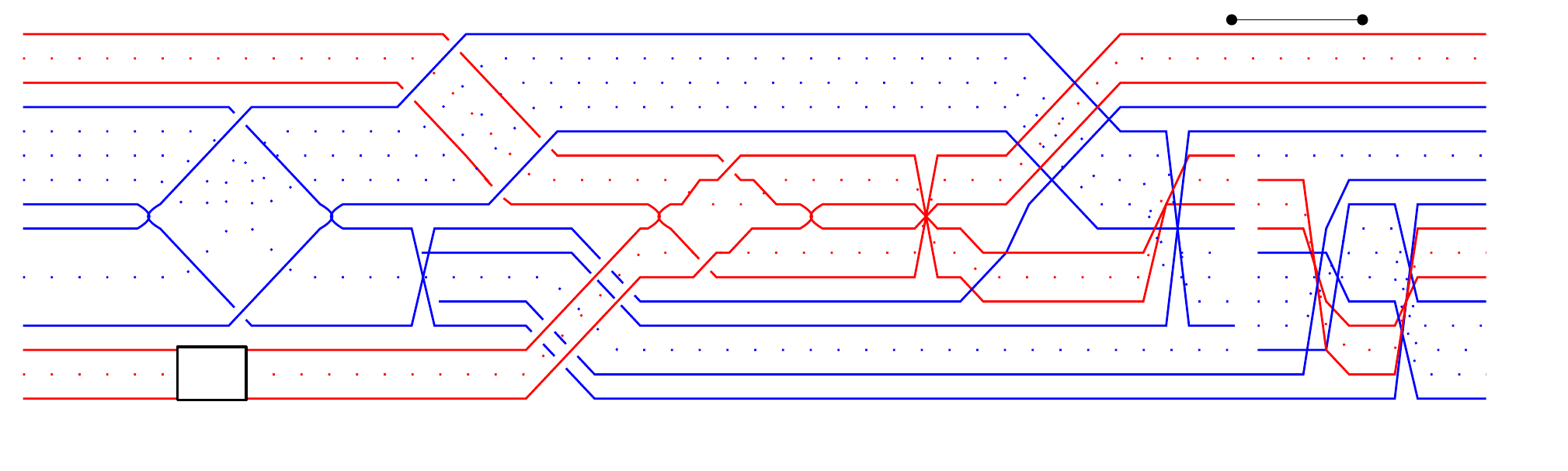}}%
    \put(0.1195517,0.05650155){\color[rgb]{0,0,0}\makebox(0,0)[lt]{\lineheight{1.25}\smash{\begin{tabular}[t]{l}$\Delta$\end{tabular}}}}%
    \put(0,0){\includegraphics[width=\unitlength,page=2]{Familieskeven7.pdf}}%
    \put(0.25415032,0.1177844){\color[rgb]{0,0,0}\makebox(0,0)[lt]{\lineheight{1.25}\smash{\begin{tabular}[t]{l}$\I_{2s-n}^{s+1}$\end{tabular}}}}%
    \put(0,0){\includegraphics[width=\unitlength,page=3]{Familieskeven7.pdf}}%
    \put(0.89009288,0.10681115){\color[rgb]{0,0,0}\makebox(0,0)[lt]{\lineheight{1.25}\smash{\begin{tabular}[t]{l}$\I_{2s}^{s}$\end{tabular}}}}%
    \put(0,0){\includegraphics[width=\unitlength,page=4]{Familieskeven7.pdf}}%
    \put(0.82043344,0.12229102){\color[rgb]{0,0,0}\makebox(0,0)[lt]{\lineheight{1.25}\smash{\begin{tabular}[t]{l}$\I_{2s-1}^{s-1}$\end{tabular}}}}%
    \put(0,0){\includegraphics[width=\unitlength,page=5]{Familieskeven7.pdf}}%
    \put(0.72190899,0.15332658){\color[rgb]{0,0,0}\makebox(0,0)[lt]{\lineheight{1.25}\smash{\begin{tabular}[t]{l}$\I_{s+n+2}^{n+2}$\end{tabular}}}}%
    \put(0,0){\includegraphics[width=\unitlength,page=6]{Familieskeven7.pdf}}%
    \put(0.78560372,0.25773994){\color[rgb]{0,0,0}\makebox(0,0)[lt]{\lineheight{1.25}\smash{\begin{tabular}[t]{l}$\I^{1}_{n}$\end{tabular}}}}%
    \put(0,0){\includegraphics[width=\unitlength,page=7]{Familieskeven7.pdf}}%
    \put(0.11609907,0.25773994){\color[rgb]{0,0,0}\makebox(0,0)[lt]{\lineheight{1.25}\smash{\begin{tabular}[t]{l}$\Delta^{-1}$\end{tabular}}}}%
    \put(0,0){\includegraphics[width=\unitlength,page=8]{Familieskeven7.pdf}}%
    \put(0.8520779,0.25782919){\color[rgb]{0,0,0}\makebox(0,0)[lt]{\lineheight{1.25}\smash{\begin{tabular}[t]{l}$\I^{1}_{n}$\end{tabular}}}}%
    \put(0,0){\includegraphics[width=\unitlength,page=9]{Familieskeven7.pdf}}%
    \put(0.53601358,0.1590712){\color[rgb]{0,0,0}\makebox(0,0)[lt]{\lineheight{1.25}\smash{\begin{tabular}[t]{l}$\I^{s+1-n}_{s+n}$\end{tabular}}}}%
    \put(0,0){\includegraphics[width=\unitlength,page=10]{Familieskeven7.pdf}}%
    \put(0.02459334,0.21837981){\color[rgb]{0,0,0}\makebox(0,0)[lt]{\lineheight{1.25}\smash{\begin{tabular}[t]{l}$\Delta^{-k-4}$\end{tabular}}}}%
    \put(0.0237884,0.10228074){\color[rgb]{0,0,0}\makebox(0,0)[lt]{\lineheight{1.25}\smash{\begin{tabular}[t]{l}$\Delta^{k+4}$\end{tabular}}}}%
  \end{picture}%
\endgroup%

}
Apply move~\eqref{eq:move-multipt-tang} to slide $\I^{s+1}_{2s-n}$ up through the tangency nest on blue strands; the result is $\I^{n+1}_s$ and the extra braiding on the left:  

{
	\fontsize{9pt}{5pt}
	\def\svgwidth{\textwidth}
	%
\begingroup%
  \makeatletter%
  \providecommand\color[2][]{%
    \errmessage{(Inkscape) Color is used for the text in Inkscape, but the package 'color.sty' is not loaded}%
    \renewcommand\color[2][]{}%
  }%
  \providecommand\transparent[1]{%
    \errmessage{(Inkscape) Transparency is used (non-zero) for the text in Inkscape, but the package 'transparent.sty' is not loaded}%
    \renewcommand\transparent[1]{}%
  }%
  \providecommand\rotatebox[2]{#2}%
  \newcommand*\fsize{\dimexpr\f@size pt\relax}%
  \newcommand*\lineheight[1]{\fontsize{\fsize}{#1\fsize}\selectfont}%
  \ifx\svgwidth\undefined%
    \setlength{\unitlength}{1050bp}%
    \ifx\svgscale\undefined%
      \relax%
    \else%
      \setlength{\unitlength}{\unitlength * \real{\svgscale}}%
    \fi%
  \else%
    \setlength{\unitlength}{\svgwidth}%
  \fi%
  \global\let\svgwidth\undefined%
  \global\let\svgscale\undefined%
  \makeatother%
  \begin{picture}(1,0.26)%
    \lineheight{1}%
    \setlength\tabcolsep{0pt}%
    \put(0,0){\includegraphics[width=\unitlength,page=1]{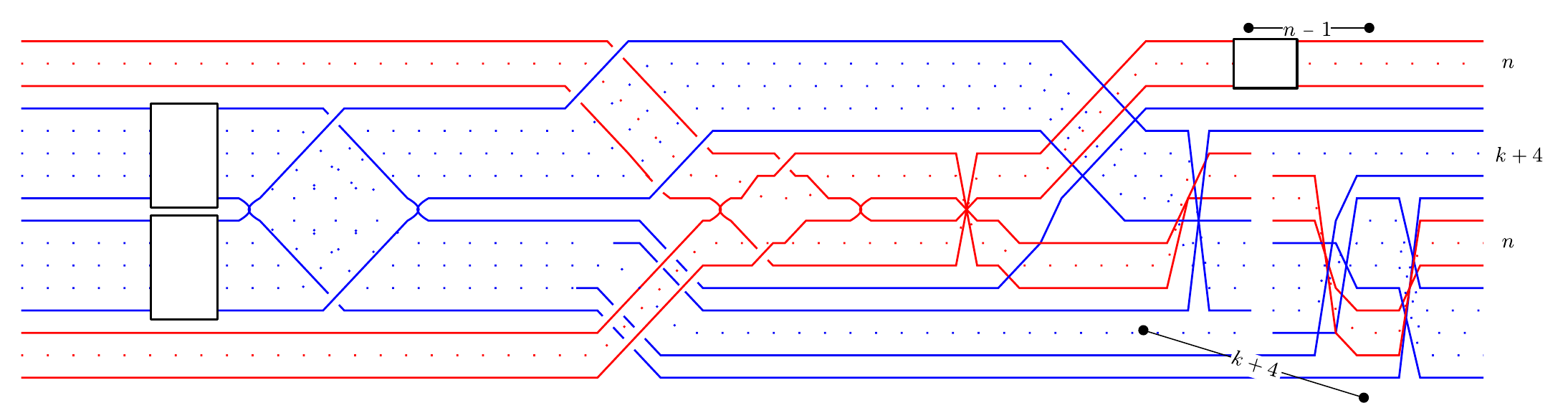}}%
    \put(0.79682079,0.21436808){\color[rgb]{0,0,0}\makebox(0,0)[lt]{\lineheight{1.25}\smash{\begin{tabular}[t]{l}$\I^{1}_{n}$\end{tabular}}}}%
    \put(0,0){\includegraphics[width=\unitlength,page=2]{Familieskeven8.pdf}}%
    \put(0.85777479,0.21436808){\color[rgb]{0,0,0}\makebox(0,0)[lt]{\lineheight{1.25}\smash{\begin{tabular}[t]{l}$\I^{1}_{n}$\end{tabular}}}}%
    \put(0,0){\includegraphics[width=\unitlength,page=3]{Familieskeven8.pdf}}%
    \put(0.89369807,0.07465682){\color[rgb]{0,0,0}\makebox(0,0)[lt]{\lineheight{1.25}\smash{\begin{tabular}[t]{l}$\I_{2s}^{s}$\end{tabular}}}}%
    \put(0,0){\includegraphics[width=\unitlength,page=4]{Familieskeven8.pdf}}%
    \put(0.825,0.09214286){\color[rgb]{0,0,0}\makebox(0,0)[lt]{\lineheight{1.25}\smash{\begin{tabular}[t]{l}$\I_{2s-1}^{s-1}$\end{tabular}}}}%
    \put(0,0){\includegraphics[width=\unitlength,page=5]{Familieskeven8.pdf}}%
    \put(0.56660342,0.12359748){\color[rgb]{0,0,0}\makebox(0,0)[lt]{\lineheight{1.25}\smash{\begin{tabular}[t]{l}$\I^{s+1-n}_{s+n}$\end{tabular}}}}%
    \put(0,0){\includegraphics[width=\unitlength,page=6]{Familieskeven8.pdf}}%
    \put(0.28388252,0.15941246){\color[rgb]{0,0,0}\makebox(0,0)[lt]{\lineheight{1.25}\smash{\begin{tabular}[t]{l}$\I_{s}^{n+1}$\end{tabular}}}}%
    \put(0,0){\includegraphics[width=\unitlength,page=7]{Familieskeven8.pdf}}%
    \put(0.73357143,0.12078425){\color[rgb]{0,0,0}\makebox(0,0)[lt]{\lineheight{1.25}\smash{\begin{tabular}[t]{l}$\I_{s+n+2}^{n+2}$\end{tabular}}}}%
    \put(0.09894994,0.1593905){\color[rgb]{0,0,0}\makebox(0,0)[lt]{\lineheight{1.25}\smash{\begin{tabular}[t]{l}$\Delta^{-1}$\end{tabular}}}}%
    \put(0.09894994,0.08796193){\color[rgb]{0,0,0}\makebox(0,0)[lt]{\lineheight{1.25}\smash{\begin{tabular}[t]{l}$\Delta$\end{tabular}}}}%
    \put(0,0){\includegraphics[width=\unitlength,page=8]{Familieskeven8.pdf}}%
    \put(0.1,0.02869801){\color[rgb]{0,0,0}\makebox(0,0)[lt]{\lineheight{1.25}\smash{\begin{tabular}[t]{l}$\Delta$\end{tabular}}}}%
    \put(0,0){\includegraphics[width=\unitlength,page=9]{Familieskeven8.pdf}}%
    \put(0.09961486,0.21441231){\color[rgb]{0,0,0}\makebox(0,0)[lt]{\lineheight{1.25}\smash{\begin{tabular}[t]{l}$\Delta^{-1}$\end{tabular}}}}%
    \put(0,0){\includegraphics[width=\unitlength,page=10]{Familieskeven8.pdf}}%
    \put(0.02338599,0.17724766){\color[rgb]{0,0,0}\makebox(0,0)[lt]{\lineheight{1.25}\smash{\begin{tabular}[t]{l}$\Delta^{-k-4}$\end{tabular}}}}%
    \put(0.02264315,0.07010481){\color[rgb]{0,0,0}\makebox(0,0)[lt]{\lineheight{1.25}\smash{\begin{tabular}[t]{l}$\Delta^{k+4}$\end{tabular}}}}%
  \end{picture}%
\endgroup%

}
Slide $\I^{n+1}_s$ over the crossings with red strands; it becomes $\I^1_{k+6}$. Use move~\eqref{eq:swap-tang-crossing}  to switch $\I^{s+1-n}_{s+n}$ to the left of the tangency nest on the red strands:

{
	\fontsize{9pt}{5pt}
	\def\svgwidth{\textwidth}
	%
\begingroup%
  \makeatletter%
  \providecommand\color[2][]{%
    \errmessage{(Inkscape) Color is used for the text in Inkscape, but the package 'color.sty' is not loaded}%
    \renewcommand\color[2][]{}%
  }%
  \providecommand\transparent[1]{%
    \errmessage{(Inkscape) Transparency is used (non-zero) for the text in Inkscape, but the package 'transparent.sty' is not loaded}%
    \renewcommand\transparent[1]{}%
  }%
  \providecommand\rotatebox[2]{#2}%
  \newcommand*\fsize{\dimexpr\f@size pt\relax}%
  \newcommand*\lineheight[1]{\fontsize{\fsize}{#1\fsize}\selectfont}%
  \ifx\svgwidth\undefined%
    \setlength{\unitlength}{1050bp}%
    \ifx\svgscale\undefined%
      \relax%
    \else%
      \setlength{\unitlength}{\unitlength * \real{\svgscale}}%
    \fi%
  \else%
    \setlength{\unitlength}{\svgwidth}%
  \fi%
  \global\let\svgwidth\undefined%
  \global\let\svgscale\undefined%
  \makeatother%
  \begin{picture}(1,0.26)%
    \lineheight{1}%
    \setlength\tabcolsep{0pt}%
    \put(0,0){\includegraphics[width=\unitlength,page=1]{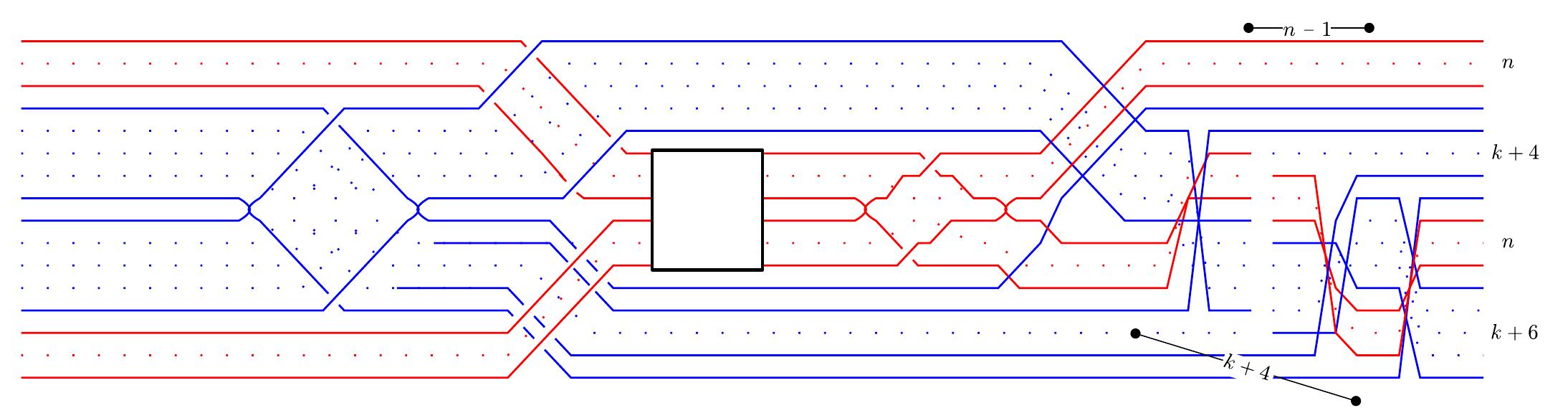}}%
    \put(0.42142857,0.12428571){\color[rgb]{0,0,0}\makebox(0,0)[lt]{\lineheight{1.25}\smash{\begin{tabular}[t]{l}$\I^{s+1-n}_{s+n}$\end{tabular}}}}%
    \put(0,0){\includegraphics[width=\unitlength,page=2]{Familieskeven9.pdf}}%
    \put(0.58214286,0.20285714){\color[rgb]{0,0,0}\makebox(0,0)[lt]{\lineheight{1.25}\smash{\begin{tabular}[t]{l}$\I_{k+6}^{1}$\end{tabular}}}}%
    \put(0,0){\includegraphics[width=\unitlength,page=3]{Familieskeven9.pdf}}%
    \put(0.79682079,0.21436808){\color[rgb]{0,0,0}\makebox(0,0)[lt]{\lineheight{1.25}\smash{\begin{tabular}[t]{l}$\I^{1}_{n}$\end{tabular}}}}%
    \put(0,0){\includegraphics[width=\unitlength,page=4]{Familieskeven9.pdf}}%
    \put(0.85714287,0.21436808){\color[rgb]{0,0,0}\makebox(0,0)[lt]{\lineheight{1.25}\smash{\begin{tabular}[t]{l}$\I^{1}_{n}$\end{tabular}}}}%
    \put(0,0){\includegraphics[width=\unitlength,page=5]{Familieskeven9.pdf}}%
    \put(0.88928571,0.07785714){\color[rgb]{0,0,0}\makebox(0,0)[lt]{\lineheight{1.25}\smash{\begin{tabular}[t]{l}$\I_{2s}^{s}$\end{tabular}}}}%
    \put(0,0){\includegraphics[width=\unitlength,page=6]{Familieskeven9.pdf}}%
    \put(0.82142857,0.08927175){\color[rgb]{0,0,0}\makebox(0,0)[lt]{\lineheight{1.25}\smash{\begin{tabular}[t]{l}$\I_{2s-1}^{s-1}$\end{tabular}}}}%
    \put(0,0){\includegraphics[width=\unitlength,page=7]{Familieskeven9.pdf}}%
    \put(0.73277631,0.11722081){\color[rgb]{0,0,0}\makebox(0,0)[lt]{\lineheight{1.25}\smash{\begin{tabular}[t]{l}$\I_{s+n+2}^{n+2}$\end{tabular}}}}%
    \put(0,0){\includegraphics[width=\unitlength,page=8]{Familieskeven9.pdf}}%
    \put(0.09854329,0.15986522){\color[rgb]{0,0,0}\makebox(0,0)[lt]{\lineheight{1.25}\smash{\begin{tabular}[t]{l}$\Delta^{-1}$\end{tabular}}}}%
    \put(0.09854329,0.08843665){\color[rgb]{0,0,0}\makebox(0,0)[lt]{\lineheight{1.25}\smash{\begin{tabular}[t]{l}$\Delta$\end{tabular}}}}%
    \put(0,0){\includegraphics[width=\unitlength,page=9]{Familieskeven9.pdf}}%
    \put(0.09880627,0.02869801){\color[rgb]{0,0,0}\makebox(0,0)[lt]{\lineheight{1.25}\smash{\begin{tabular}[t]{l}$\Delta$\end{tabular}}}}%
    \put(0,0){\includegraphics[width=\unitlength,page=10]{Familieskeven9.pdf}}%
    \put(0.09842113,0.21441231){\color[rgb]{0,0,0}\makebox(0,0)[lt]{\lineheight{1.25}\smash{\begin{tabular}[t]{l}$\Delta^{-1}$\end{tabular}}}}%
    \put(0,0){\includegraphics[width=\unitlength,page=11]{Familieskeven9.pdf}}%
    \put(0.02626757,0.17724766){\color[rgb]{0,0,0}\makebox(0,0)[lt]{\lineheight{1.25}\smash{\begin{tabular}[t]{l}$\Delta^{-k-4}$\end{tabular}}}}%
    \put(0.02552472,0.07010481){\color[rgb]{0,0,0}\makebox(0,0)[lt]{\lineheight{1.25}\smash{\begin{tabular}[t]{l}$\Delta^{k+4}$\end{tabular}}}}%
  \end{picture}%
\endgroup%

}

Combine $\I^1_{k+6}$ and one of the copies of $\I^1_n$ with double points of the $n$ red strands, 
using move~\eqref{eq:split-many} to create  $\I^1_s$:

{
	\fontsize{9pt}{5pt}
	\def\svgwidth{\textwidth}
	%
\begingroup%
  \makeatletter%
  \providecommand\color[2][]{%
    \errmessage{(Inkscape) Color is used for the text in Inkscape, but the package 'color.sty' is not loaded}%
    \renewcommand\color[2][]{}%
  }%
  \providecommand\transparent[1]{%
    \errmessage{(Inkscape) Transparency is used (non-zero) for the text in Inkscape, but the package 'transparent.sty' is not loaded}%
    \renewcommand\transparent[1]{}%
  }%
  \providecommand\rotatebox[2]{#2}%
  \newcommand*\fsize{\dimexpr\f@size pt\relax}%
  \newcommand*\lineheight[1]{\fontsize{\fsize}{#1\fsize}\selectfont}%
  \ifx\svgwidth\undefined%
    \setlength{\unitlength}{1050bp}%
    \ifx\svgscale\undefined%
      \relax%
    \else%
      \setlength{\unitlength}{\unitlength * \real{\svgscale}}%
    \fi%
  \else%
    \setlength{\unitlength}{\svgwidth}%
  \fi%
  \global\let\svgwidth\undefined%
  \global\let\svgscale\undefined%
  \makeatother%
  \begin{picture}(1,0.26)%
    \lineheight{1}%
    \setlength\tabcolsep{0pt}%
    \put(0,0){\includegraphics[width=\unitlength,page=1]{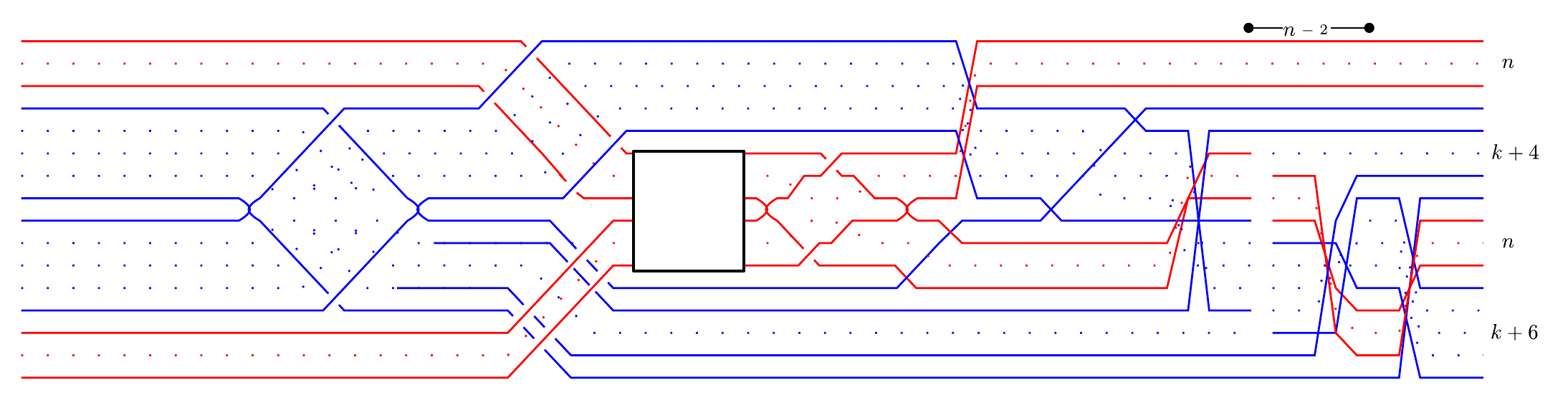}}%
    \put(0.40946057,0.12359748){\color[rgb]{0,0,0}\makebox(0,0)[lt]{\lineheight{1.25}\smash{\begin{tabular}[t]{l}$\I^{s+1-n}_{s+n}$\end{tabular}}}}%
    \put(0,0){\includegraphics[width=\unitlength,page=2]{Familieskeven10.pdf}}%
    \put(0.61032915,0.17785714){\color[rgb]{0,0,0}\makebox(0,0)[lt]{\lineheight{1.25}\smash{\begin{tabular}[t]{l}$\I_{s}^{1}$\end{tabular}}}}%
    \put(0,0){\includegraphics[width=\unitlength,page=3]{Familieskeven10.pdf}}%
    \put(0.79563193,0.21436808){\color[rgb]{0,0,0}\makebox(0,0)[lt]{\lineheight{1.25}\smash{\begin{tabular}[t]{l}$\I^{1}_{n}$\end{tabular}}}}%
    \put(0,0){\includegraphics[width=\unitlength,page=4]{Familieskeven10.pdf}}%
    \put(0.85714287,0.21436808){\color[rgb]{0,0,0}\makebox(0,0)[lt]{\lineheight{1.25}\smash{\begin{tabular}[t]{l}$\I^{1}_{n}$\end{tabular}}}}%
    \put(0,0){\includegraphics[width=\unitlength,page=5]{Familieskeven10.pdf}}%
    \put(0.88928571,0.07785714){\color[rgb]{0,0,0}\makebox(0,0)[lt]{\lineheight{1.25}\smash{\begin{tabular}[t]{l}$\I_{2s}^{s}$\end{tabular}}}}%
    \put(0,0){\includegraphics[width=\unitlength,page=6]{Familieskeven10.pdf}}%
    \put(0.82142857,0.08927176){\color[rgb]{0,0,0}\makebox(0,0)[lt]{\lineheight{1.25}\smash{\begin{tabular}[t]{l}$\I_{2s-1}^{s-1}$\end{tabular}}}}%
    \put(0,0){\includegraphics[width=\unitlength,page=7]{Familieskeven10.pdf}}%
    \put(0.73277631,0.11722081){\color[rgb]{0,0,0}\makebox(0,0)[lt]{\lineheight{1.25}\smash{\begin{tabular}[t]{l}$\I_{s+n+2}^{n+2}$\end{tabular}}}}%
    \put(0,0){\includegraphics[width=\unitlength,page=8]{Familieskeven10.pdf}}%
    \put(0.10211472,0.16343665){\color[rgb]{0,0,0}\makebox(0,0)[lt]{\lineheight{1.25}\smash{\begin{tabular}[t]{l}$\Delta^{-1}$\end{tabular}}}}%
    \put(0.10211472,0.09200808){\color[rgb]{0,0,0}\makebox(0,0)[lt]{\lineheight{1.25}\smash{\begin{tabular}[t]{l}$\Delta$\end{tabular}}}}%
    \put(0,0){\includegraphics[width=\unitlength,page=9]{Familieskeven10.pdf}}%
    \put(0.1023777,0.02865378){\color[rgb]{0,0,0}\makebox(0,0)[lt]{\lineheight{1.25}\smash{\begin{tabular}[t]{l}$\Delta$\end{tabular}}}}%
    \put(0,0){\includegraphics[width=\unitlength,page=10]{Familieskeven10.pdf}}%
    \put(0.10199256,0.21436808){\color[rgb]{0,0,0}\makebox(0,0)[lt]{\lineheight{1.25}\smash{\begin{tabular}[t]{l}$\Delta^{-1}$\end{tabular}}}}%
    \put(0,0){\includegraphics[width=\unitlength,page=11]{Familieskeven10.pdf}}%
    \put(0.02697645,0.17724766){\color[rgb]{0,0,0}\makebox(0,0)[lt]{\lineheight{1.25}\smash{\begin{tabular}[t]{l}$\Delta^{-k-4}$\end{tabular}}}}%
    \put(0.02623361,0.07010481){\color[rgb]{0,0,0}\makebox(0,0)[lt]{\lineheight{1.25}\smash{\begin{tabular}[t]{l}$\Delta^{k+4}$\end{tabular}}}}%
  \end{picture}%
\endgroup%

}

Using~\eqref{eq:move-strand-tang}, move (blue) strand $s+n+1$  that has double point intersections  with $n$ (red) strands $s+1, \dots s+n$ located to the right of the red tangency nest:   

{
	\fontsize{9pt}{5pt}
	\def\svgwidth{\textwidth}
	%
\begingroup%
  \makeatletter%
  \providecommand\color[2][]{%
    \errmessage{(Inkscape) Color is used for the text in Inkscape, but the package 'color.sty' is not loaded}%
    \renewcommand\color[2][]{}%
  }%
  \providecommand\transparent[1]{%
    \errmessage{(Inkscape) Transparency is used (non-zero) for the text in Inkscape, but the package 'transparent.sty' is not loaded}%
    \renewcommand\transparent[1]{}%
  }%
  \providecommand\rotatebox[2]{#2}%
  \newcommand*\fsize{\dimexpr\f@size pt\relax}%
  \newcommand*\lineheight[1]{\fontsize{\fsize}{#1\fsize}\selectfont}%
  \ifx\svgwidth\undefined%
    \setlength{\unitlength}{1050bp}%
    \ifx\svgscale\undefined%
      \relax%
    \else%
      \setlength{\unitlength}{\unitlength * \real{\svgscale}}%
    \fi%
  \else%
    \setlength{\unitlength}{\svgwidth}%
  \fi%
  \global\let\svgwidth\undefined%
  \global\let\svgscale\undefined%
  \makeatother%
  \begin{picture}(1,0.26)%
    \lineheight{1}%
    \setlength\tabcolsep{0pt}%
    \put(0,0){\includegraphics[width=\unitlength,page=1]{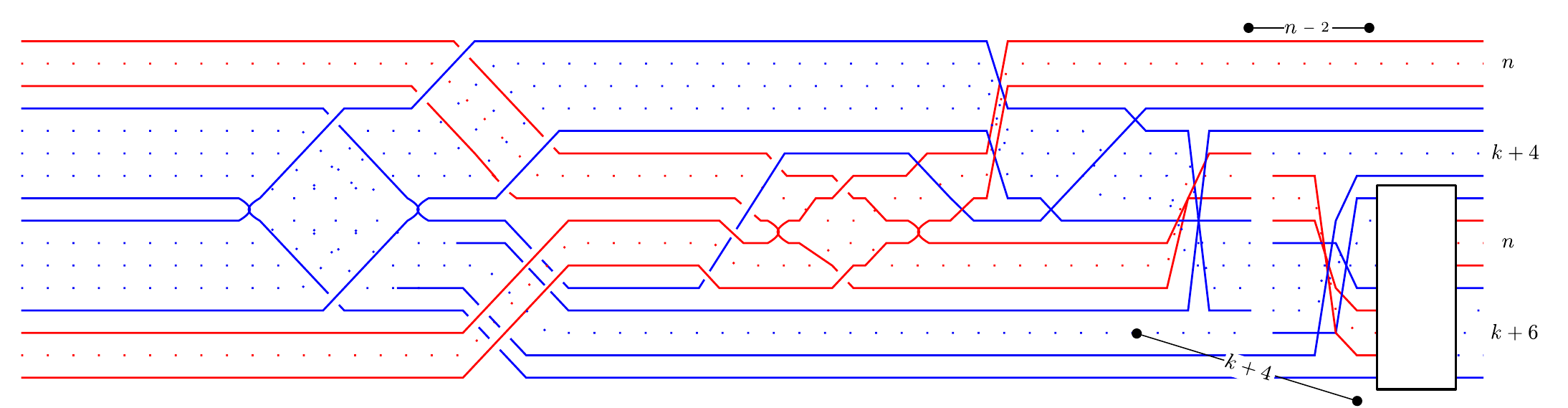}}%
    \put(0.89008084,0.07785714){\color[rgb]{0,0,0}\makebox(0,0)[lt]{\lineheight{1.25}\smash{\begin{tabular}[t]{l}$\I_{2s}^{s}$\end{tabular}}}}%
    \put(0,0){\includegraphics[width=\unitlength,page=2]{Familieskeven11.pdf}}%
    \put(0.82222369,0.08927175){\color[rgb]{0,0,0}\makebox(0,0)[lt]{\lineheight{1.25}\smash{\begin{tabular}[t]{l}$\I_{2s-1}^{s-1}$\end{tabular}}}}%
    \put(0,0){\includegraphics[width=\unitlength,page=3]{Familieskeven11.pdf}}%
    \put(0.73357144,0.11722081){\color[rgb]{0,0,0}\makebox(0,0)[lt]{\lineheight{1.25}\smash{\begin{tabular}[t]{l}$\I_{s+n+2}^{n+2}$\end{tabular}}}}%
    \put(0,0){\includegraphics[width=\unitlength,page=4]{Familieskeven11.pdf}}%
    \put(0.79642859,0.21436808){\color[rgb]{0,0,0}\makebox(0,0)[lt]{\lineheight{1.25}\smash{\begin{tabular}[t]{l}$\I^{1}_{n}$\end{tabular}}}}%
    \put(0,0){\includegraphics[width=\unitlength,page=5]{Familieskeven11.pdf}}%
    \put(0.85793952,0.21436808){\color[rgb]{0,0,0}\makebox(0,0)[lt]{\lineheight{1.25}\smash{\begin{tabular}[t]{l}$\I^{1}_{n}$\end{tabular}}}}%
    \put(0,0){\includegraphics[width=\unitlength,page=6]{Familieskeven11.pdf}}%
    \put(0.63449084,0.17778467){\color[rgb]{0,0,0}\makebox(0,0)[lt]{\lineheight{1.25}\smash{\begin{tabular}[t]{l}$\I_{s}^{1}$\end{tabular}}}}%
    \put(0,0){\includegraphics[width=\unitlength,page=7]{Familieskeven11.pdf}}%
    \put(0.37323539,0.12359748){\color[rgb]{0,0,0}\makebox(0,0)[lt]{\lineheight{1.25}\smash{\begin{tabular}[t]{l}$\I^{s+1-n}_{s+n}$\end{tabular}}}}%
    \put(0,0){\includegraphics[width=\unitlength,page=8]{Familieskeven11.pdf}}%
    \put(0.10211472,0.16343665){\color[rgb]{0,0,0}\makebox(0,0)[lt]{\lineheight{1.25}\smash{\begin{tabular}[t]{l}$\Delta^{-1}$\end{tabular}}}}%
    \put(0.10211472,0.09200808){\color[rgb]{0,0,0}\makebox(0,0)[lt]{\lineheight{1.25}\smash{\begin{tabular}[t]{l}$\Delta$\end{tabular}}}}%
    \put(0,0){\includegraphics[width=\unitlength,page=9]{Familieskeven11.pdf}}%
    \put(0.10357143,0.02869801){\color[rgb]{0,0,0}\makebox(0,0)[lt]{\lineheight{1.25}\smash{\begin{tabular}[t]{l}$\Delta$\end{tabular}}}}%
    \put(0,0){\includegraphics[width=\unitlength,page=10]{Familieskeven11.pdf}}%
    \put(0.10318629,0.21441231){\color[rgb]{0,0,0}\makebox(0,0)[lt]{\lineheight{1.25}\smash{\begin{tabular}[t]{l}$\Delta^{-1}$\end{tabular}}}}%
    \put(0,0){\includegraphics[width=\unitlength,page=11]{Familieskeven11.pdf}}%
    \put(0.02626757,0.17724766){\color[rgb]{0,0,0}\makebox(0,0)[lt]{\lineheight{1.25}\smash{\begin{tabular}[t]{l}$\Delta^{-k-4}$\end{tabular}}}}%
    \put(0.02552472,0.07010481){\color[rgb]{0,0,0}\makebox(0,0)[lt]{\lineheight{1.25}\smash{\begin{tabular}[t]{l}$\Delta^{k+4}$\end{tabular}}}}%
  \end{picture}%
\endgroup%

}

Use~\eqref{eq:split-one-middle} to combine double points from this last blue strand with $\I^1_{s}$ into $\I^1_{s+1}$:

{
	\fontsize{9pt}{5pt}
	\def\svgwidth{\textwidth}
	%
\begingroup%
  \makeatletter%
  \providecommand\color[2][]{%
    \errmessage{(Inkscape) Color is used for the text in Inkscape, but the package 'color.sty' is not loaded}%
    \renewcommand\color[2][]{}%
  }%
  \providecommand\transparent[1]{%
    \errmessage{(Inkscape) Transparency is used (non-zero) for the text in Inkscape, but the package 'transparent.sty' is not loaded}%
    \renewcommand\transparent[1]{}%
  }%
  \providecommand\rotatebox[2]{#2}%
  \newcommand*\fsize{\dimexpr\f@size pt\relax}%
  \newcommand*\lineheight[1]{\fontsize{\fsize}{#1\fsize}\selectfont}%
  \ifx\svgwidth\undefined%
    \setlength{\unitlength}{1050bp}%
    \ifx\svgscale\undefined%
      \relax%
    \else%
      \setlength{\unitlength}{\unitlength * \real{\svgscale}}%
    \fi%
  \else%
    \setlength{\unitlength}{\svgwidth}%
  \fi%
  \global\let\svgwidth\undefined%
  \global\let\svgscale\undefined%
  \makeatother%
  \begin{picture}(1,0.26)%
    \lineheight{1}%
    \setlength\tabcolsep{0pt}%
    \put(0,0){\includegraphics[width=\unitlength,page=1]{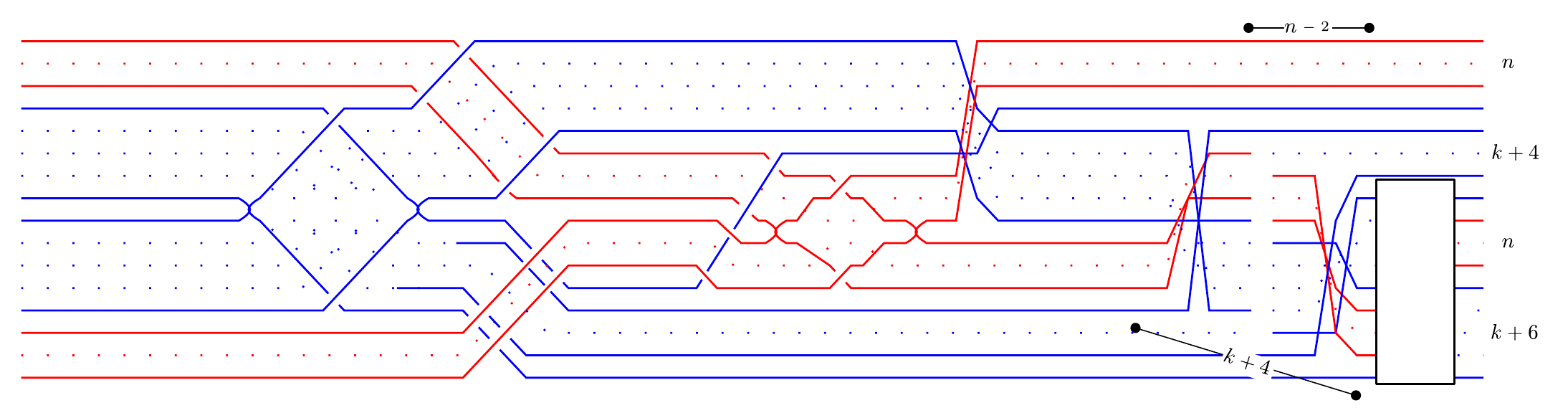}}%
    \put(0.88928572,0.08142857){\color[rgb]{0,0,0}\makebox(0,0)[lt]{\lineheight{1.25}\smash{\begin{tabular}[t]{l}$\I_{2s}^{s}$\end{tabular}}}}%
    \put(0,0){\includegraphics[width=\unitlength,page=2]{Familieskeven12.pdf}}%
    \put(0.82142857,0.09284318){\color[rgb]{0,0,0}\makebox(0,0)[lt]{\lineheight{1.25}\smash{\begin{tabular}[t]{l}$\I_{2s-1}^{s-1}$\end{tabular}}}}%
    \put(0,0){\includegraphics[width=\unitlength,page=3]{Familieskeven12.pdf}}%
    \put(0.73277631,0.12079224){\color[rgb]{0,0,0}\makebox(0,0)[lt]{\lineheight{1.25}\smash{\begin{tabular}[t]{l}$\I_{s+n+2}^{n+2}$\end{tabular}}}}%
    \put(0,0){\includegraphics[width=\unitlength,page=4]{Familieskeven12.pdf}}%
    \put(0.61032915,0.17071429){\color[rgb]{0,0,0}\makebox(0,0)[lt]{\lineheight{1.25}\smash{\begin{tabular}[t]{l}$\I_{s+1}^{1}$\end{tabular}}}}%
    \put(0,0){\includegraphics[width=\unitlength,page=5]{Familieskeven12.pdf}}%
    \put(0.79682079,0.21436808){\color[rgb]{0,0,0}\makebox(0,0)[lt]{\lineheight{1.25}\smash{\begin{tabular}[t]{l}$\I^{1}_{n}$\end{tabular}}}}%
    \put(0,0){\includegraphics[width=\unitlength,page=6]{Familieskeven12.pdf}}%
    \put(0.85833172,0.21436808){\color[rgb]{0,0,0}\makebox(0,0)[lt]{\lineheight{1.25}\smash{\begin{tabular}[t]{l}$\I^{1}_{n}$\end{tabular}}}}%
    \put(0,0){\includegraphics[width=\unitlength,page=7]{Familieskeven12.pdf}}%
    \put(0.37323539,0.12359748){\color[rgb]{0,0,0}\makebox(0,0)[lt]{\lineheight{1.25}\smash{\begin{tabular}[t]{l}$\I^{s+1-n}_{s+n}$\end{tabular}}}}%
    \put(0,0){\includegraphics[width=\unitlength,page=8]{Familieskeven12.pdf}}%
    \put(0.10353103,0.16343665){\color[rgb]{0,0,0}\makebox(0,0)[lt]{\lineheight{1.25}\smash{\begin{tabular}[t]{l}$\Delta^{-1}$\end{tabular}}}}%
    \put(0.10353103,0.09200808){\color[rgb]{0,0,0}\makebox(0,0)[lt]{\lineheight{1.25}\smash{\begin{tabular}[t]{l}$\Delta$\end{tabular}}}}%
    \put(0,0){\includegraphics[width=\unitlength,page=9]{Familieskeven12.pdf}}%
    \put(0.10498774,0.02869801){\color[rgb]{0,0,0}\makebox(0,0)[lt]{\lineheight{1.25}\smash{\begin{tabular}[t]{l}$\Delta$\end{tabular}}}}%
    \put(0,0){\includegraphics[width=\unitlength,page=10]{Familieskeven12.pdf}}%
    \put(0.1046026,0.21441231){\color[rgb]{0,0,0}\makebox(0,0)[lt]{\lineheight{1.25}\smash{\begin{tabular}[t]{l}$\Delta^{-1}$\end{tabular}}}}%
    \put(0,0){\includegraphics[width=\unitlength,page=11]{Familieskeven12.pdf}}%
    \put(0.02425715,0.1786312){\color[rgb]{0,0,0}\makebox(0,0)[lt]{\lineheight{1.25}\smash{\begin{tabular}[t]{l}$\Delta^{-k-4}$\end{tabular}}}}%
    \put(0.02351431,0.07148834){\color[rgb]{0,0,0}\makebox(0,0)[lt]{\lineheight{1.25}\smash{\begin{tabular}[t]{l}$\Delta^{k+4}$\end{tabular}}}}%
  \end{picture}%
\endgroup%

}

Use move~\eqref{eq:switch-multipts} to switch $\I^1_{s+1}$ with multipoints $\I^1_n$ on the right, moving these $(n-2)$ multipoints next to the red tangency nest:

{
	\fontsize{9pt}{5pt}
	\def\svgwidth{\textwidth}
	%
\begingroup%
  \makeatletter%
  \providecommand\color[2][]{%
    \errmessage{(Inkscape) Color is used for the text in Inkscape, but the package 'color.sty' is not loaded}%
    \renewcommand\color[2][]{}%
  }%
  \providecommand\transparent[1]{%
    \errmessage{(Inkscape) Transparency is used (non-zero) for the text in Inkscape, but the package 'transparent.sty' is not loaded}%
    \renewcommand\transparent[1]{}%
  }%
  \providecommand\rotatebox[2]{#2}%
  \newcommand*\fsize{\dimexpr\f@size pt\relax}%
  \newcommand*\lineheight[1]{\fontsize{\fsize}{#1\fsize}\selectfont}%
  \ifx\svgwidth\undefined%
    \setlength{\unitlength}{1050bp}%
    \ifx\svgscale\undefined%
      \relax%
    \else%
      \setlength{\unitlength}{\unitlength * \real{\svgscale}}%
    \fi%
  \else%
    \setlength{\unitlength}{\svgwidth}%
  \fi%
  \global\let\svgwidth\undefined%
  \global\let\svgscale\undefined%
  \makeatother%
  \begin{picture}(1,0.26)%
    \lineheight{1}%
    \setlength\tabcolsep{0pt}%
    \put(0,0){\includegraphics[width=\unitlength,page=1]{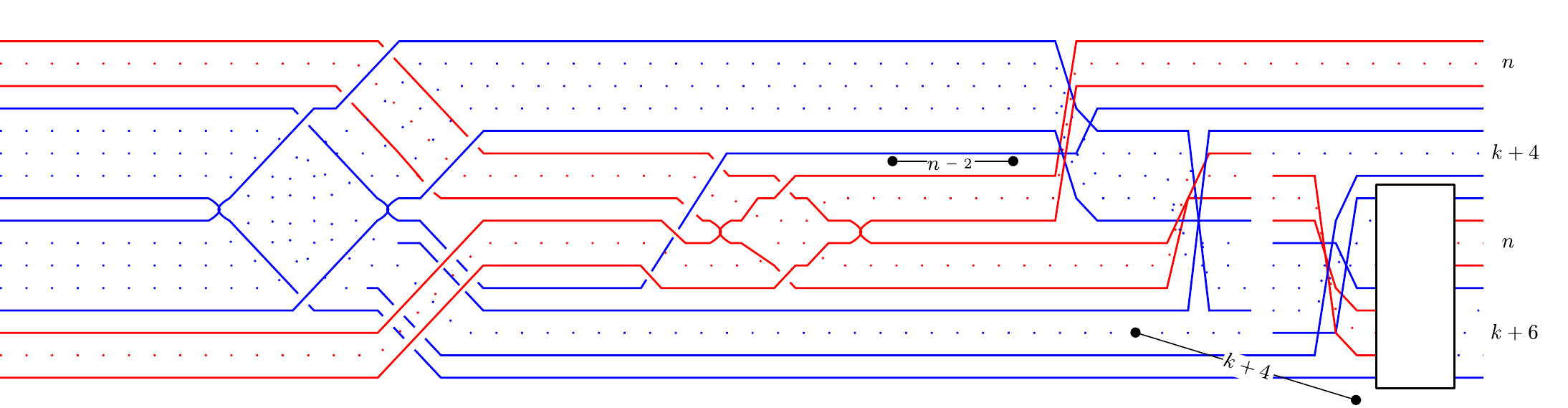}}%
    \put(0.88928572,0.07847951){\color[rgb]{0,0,0}\makebox(0,0)[lt]{\lineheight{1.25}\smash{\begin{tabular}[t]{l}$\I_{2s}^{s}$\end{tabular}}}}%
    \put(0,0){\includegraphics[width=\unitlength,page=2]{Familieskeven13.pdf}}%
    \put(0.82142857,0.08989413){\color[rgb]{0,0,0}\makebox(0,0)[lt]{\lineheight{1.25}\smash{\begin{tabular}[t]{l}$\I_{2s-1}^{s-1}$\end{tabular}}}}%
    \put(0,0){\includegraphics[width=\unitlength,page=3]{Familieskeven13.pdf}}%
    \put(0.73277631,0.11784318){\color[rgb]{0,0,0}\makebox(0,0)[lt]{\lineheight{1.25}\smash{\begin{tabular}[t]{l}$\I_{s+n+2}^{n+2}$\end{tabular}}}}%
    \put(0,0){\includegraphics[width=\unitlength,page=4]{Familieskeven13.pdf}}%
    \put(0.6781863,0.16803759){\color[rgb]{0,0,0}\makebox(0,0)[lt]{\lineheight{1.25}\smash{\begin{tabular}[t]{l}$\I_{s+1}^{1}$\end{tabular}}}}%
    \put(0,0){\includegraphics[width=\unitlength,page=5]{Familieskeven13.pdf}}%
    \put(0.56247201,0.12785714){\color[rgb]{0,0,0}\makebox(0,0)[lt]{\lineheight{1.25}\smash{\begin{tabular}[t]{l}$\I^{k+8}_{s+1}$\end{tabular}}}}%
    \put(0,0){\includegraphics[width=\unitlength,page=6]{Familieskeven13.pdf}}%
    \put(0.62390058,0.12785714){\color[rgb]{0,0,0}\makebox(0,0)[lt]{\lineheight{1.25}\smash{\begin{tabular}[t]{l}$\I^{k+8}_{s+1}$\end{tabular}}}}%
    \put(0,0){\includegraphics[width=\unitlength,page=7]{Familieskeven13.pdf}}%
    \put(0.32374627,0.12359748){\color[rgb]{0,0,0}\makebox(0,0)[lt]{\lineheight{1.25}\smash{\begin{tabular}[t]{l}$\I^{s+1-n}_{s+n}$\end{tabular}}}}%
    \put(0,0){\includegraphics[width=\unitlength,page=8]{Familieskeven13.pdf}}%
    \put(0.08856718,0.16127072){\color[rgb]{0,0,0}\makebox(0,0)[lt]{\lineheight{1.25}\smash{\begin{tabular}[t]{l}$\Delta^{-1}$\end{tabular}}}}%
    \put(0.08856718,0.08984215){\color[rgb]{0,0,0}\makebox(0,0)[lt]{\lineheight{1.25}\smash{\begin{tabular}[t]{l}$\Delta$\end{tabular}}}}%
    \put(0,0){\includegraphics[width=\unitlength,page=9]{Familieskeven13.pdf}}%
    \put(0.09002389,0.02653207){\color[rgb]{0,0,0}\makebox(0,0)[lt]{\lineheight{1.25}\smash{\begin{tabular}[t]{l}$\Delta$\end{tabular}}}}%
    \put(0,0){\includegraphics[width=\unitlength,page=10]{Familieskeven13.pdf}}%
    \put(0.08963876,0.21224638){\color[rgb]{0,0,0}\makebox(0,0)[lt]{\lineheight{1.25}\smash{\begin{tabular}[t]{l}$\Delta^{-1}$\end{tabular}}}}%
    \put(0,0){\includegraphics[width=\unitlength,page=11]{Familieskeven13.pdf}}%
    \put(0.01354287,0.17724766){\color[rgb]{0,0,0}\makebox(0,0)[lt]{\lineheight{1.25}\smash{\begin{tabular}[t]{l}$\Delta^{-k-4}$\end{tabular}}}}%
    \put(0.01280002,0.07010481){\color[rgb]{0,0,0}\makebox(0,0)[lt]{\lineheight{1.25}\smash{\begin{tabular}[t]{l}$\Delta^{k+4}$\end{tabular}}}}%
  \end{picture}%
\endgroup%

}

Move the red intersections through the red tangency nest using~\eqref{eq:TNIswap2}; this creates extra braiding $\Delta^{n-2}$. Using move~\eqref{eq:move-braid-multipt}, we can slide these half twists to the left side, past $\I^{s+1-n}_{s+n}$ and then under the blue strands. 

{
	\fontsize{9pt}{5pt}
	\def\svgwidth{\textwidth}
	%
\begingroup%
  \makeatletter%
  \providecommand\color[2][]{%
    \errmessage{(Inkscape) Color is used for the text in Inkscape, but the package 'color.sty' is not loaded}%
    \renewcommand\color[2][]{}%
  }%
  \providecommand\transparent[1]{%
    \errmessage{(Inkscape) Transparency is used (non-zero) for the text in Inkscape, but the package 'transparent.sty' is not loaded}%
    \renewcommand\transparent[1]{}%
  }%
  \providecommand\rotatebox[2]{#2}%
  \newcommand*\fsize{\dimexpr\f@size pt\relax}%
  \newcommand*\lineheight[1]{\fontsize{\fsize}{#1\fsize}\selectfont}%
  \ifx\svgwidth\undefined%
    \setlength{\unitlength}{1050bp}%
    \ifx\svgscale\undefined%
      \relax%
    \else%
      \setlength{\unitlength}{\unitlength * \real{\svgscale}}%
    \fi%
  \else%
    \setlength{\unitlength}{\svgwidth}%
  \fi%
  \global\let\svgwidth\undefined%
  \global\let\svgscale\undefined%
  \makeatother%
  \begin{picture}(1,0.26)%
    \lineheight{1}%
    \setlength\tabcolsep{0pt}%
    \put(0,0){\includegraphics[width=\unitlength,page=1]{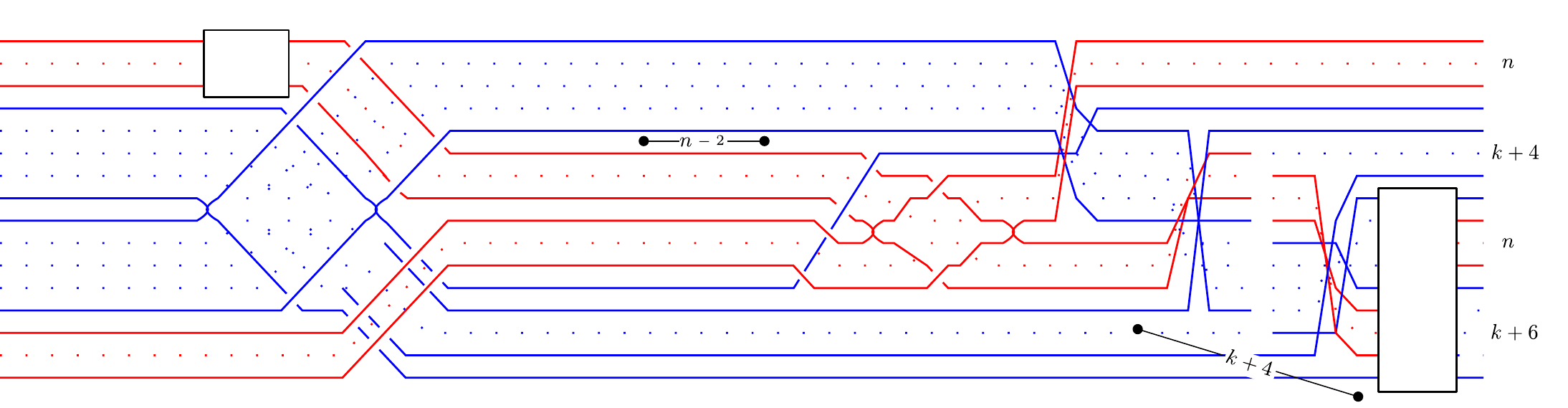}}%
    \put(0.8907074,0.07605747){\color[rgb]{0,0,0}\makebox(0,0)[lt]{\lineheight{1.25}\smash{\begin{tabular}[t]{l}$\I_{2s}^{s}$\end{tabular}}}}%
    \put(0,0){\includegraphics[width=\unitlength,page=2]{Familieskeven14.pdf}}%
    \put(0.82285025,0.09204807){\color[rgb]{0,0,0}\makebox(0,0)[lt]{\lineheight{1.25}\smash{\begin{tabular}[t]{l}$\I_{2s-1}^{s-1}$\end{tabular}}}}%
    \put(0,0){\includegraphics[width=\unitlength,page=3]{Familieskeven14.pdf}}%
    \put(0.73214286,0.11999711){\color[rgb]{0,0,0}\makebox(0,0)[lt]{\lineheight{1.25}\smash{\begin{tabular}[t]{l}$\I_{s+n+2}^{n+2}$\end{tabular}}}}%
    \put(0,0){\includegraphics[width=\unitlength,page=4]{Familieskeven14.pdf}}%
    \put(0.66727465,0.16805838){\color[rgb]{0,0,0}\makebox(0,0)[lt]{\lineheight{1.25}\smash{\begin{tabular}[t]{l}$\I_{s+1}^{1}$\end{tabular}}}}%
    \put(0,0){\includegraphics[width=\unitlength,page=5]{Familieskeven14.pdf}}%
    \put(0.29874627,0.12359748){\color[rgb]{0,0,0}\makebox(0,0)[lt]{\lineheight{1.25}\smash{\begin{tabular}[t]{l}$\I^{s+1-n}_{s+n}$\end{tabular}}}}%
    \put(0,0){\includegraphics[width=\unitlength,page=6]{Familieskeven14.pdf}}%
    \put(0.4029466,0.14214286){\color[rgb]{0,0,0}\makebox(0,0)[lt]{\lineheight{1.25}\smash{\begin{tabular}[t]{l}$\I^{k+7}_{s}$\end{tabular}}}}%
    \put(0,0){\includegraphics[width=\unitlength,page=7]{Familieskeven14.pdf}}%
    \put(0.46437517,0.14214286){\color[rgb]{0,0,0}\makebox(0,0)[lt]{\lineheight{1.25}\smash{\begin{tabular}[t]{l}$\I^{k+7}_{s}$\end{tabular}}}}%
    \put(0.13403009,0.21653336){\color[rgb]{0,0,0}\makebox(0,0)[lt]{\lineheight{1.25}\smash{\begin{tabular}[t]{l}$\Delta^{n-3}$\end{tabular}}}}%
    \put(0,0){\includegraphics[width=\unitlength,page=8]{Familieskeven14.pdf}}%
    \put(0.83045864,0.21790955){\color[rgb]{0,0,0}\makebox(0,0)[lt]{\lineheight{1.25}\smash{\begin{tabular}[t]{l}$\Delta^{-n+2}$\end{tabular}}}}%
    \put(0,0){\includegraphics[width=\unitlength,page=9]{Familieskeven14.pdf}}%
    \put(0.087829,0.16296389){\color[rgb]{0,0,0}\makebox(0,0)[lt]{\lineheight{1.25}\smash{\begin{tabular}[t]{l}$\Delta^{-1}$\end{tabular}}}}%
    \put(0.087829,0.09153532){\color[rgb]{0,0,0}\makebox(0,0)[lt]{\lineheight{1.25}\smash{\begin{tabular}[t]{l}$\Delta$\end{tabular}}}}%
    \put(0,0){\includegraphics[width=\unitlength,page=10]{Familieskeven14.pdf}}%
    \put(0.08928571,0.02822525){\color[rgb]{0,0,0}\makebox(0,0)[lt]{\lineheight{1.25}\smash{\begin{tabular}[t]{l}$\Delta$\end{tabular}}}}%
    \put(0,0){\includegraphics[width=\unitlength,page=11]{Familieskeven14.pdf}}%
    \put(0.01428571,0.17724766){\color[rgb]{0,0,0}\makebox(0,0)[lt]{\lineheight{1.25}\smash{\begin{tabular}[t]{l}$\Delta^{-k-4}$\end{tabular}}}}%
    \put(0.01354287,0.07010481){\color[rgb]{0,0,0}\makebox(0,0)[lt]{\lineheight{1.25}\smash{\begin{tabular}[t]{l}$\Delta^{k+4}$\end{tabular}}}}%
  \end{picture}%
\endgroup%

}

We now focus on the subdiagram in the middle of the figure, formed by  the red strands with tangency nest, the multipoints $\I^{s+1-n}_{s+n}$ and $(n-2)$ copies of $\I^{k+7}_s$.  This subconfiguration is similar to the previous ones, and we can work with it separately: although the blue strand passing through the middle cannot be isotoped away, it does not interfere with any moves we need. 
The indices in the pictures  are shifted by $k+6$ for legibility. The subdiagram on the red strands is the Scott deformation of a star-shaped graph in $\mathcal C$, and the subdiagram we produce is the $\mathbb Q$HD arrangement for that singularity.
We break up the multipoint $\I^1_{2n}$ as $\I^n_{2n} \X^{1, n-1}_{n, 2n} \I^{n+2}_{2n}$ by~\eqref{eq:split-many}, and follow our previous strategy. Namely, we perform several moves as in the beginning of this section (the rightmost elements of the previous diagrams): combine $\I^{n+2}_{2n}$ with the strand above it and the corresponding double points to make $\I^{n+1}_{2n}$. Then, slide copies of $\I^1_{n}$ down through $\X$ one by one,  combining each with double points on one strand.   This creates the
 multipoints $\I^{n}_{2n}, \dots, \I^2_{n+2}$.  Finally, we slide $\I_{2n}^{n+1}$ up through the tangency nest, at the cost of extra braiding, and combine with the remaining strand.
 
\begin{figure*}[ht!]
\centering
\includegraphics[scale=.6]{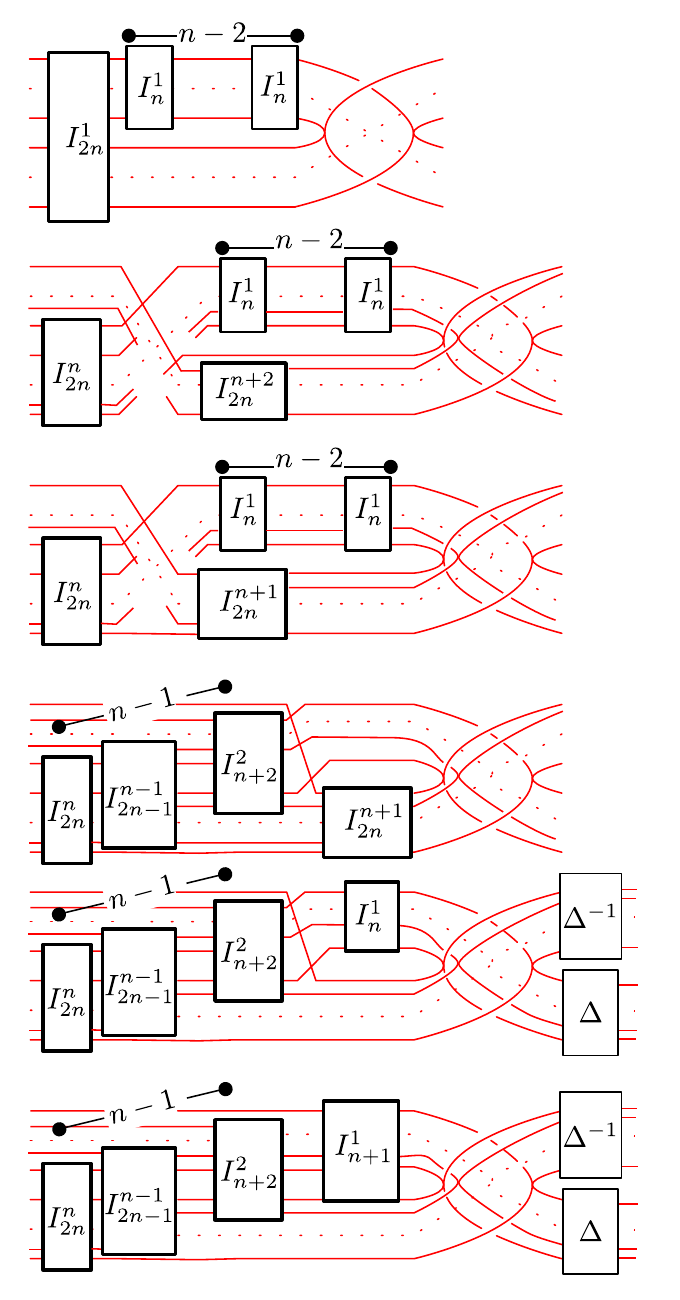}
\end{figure*}

After incorporating this last diagram as a subconfiguration into the previous diagram on the red and blue strands and decorating all intersections with marked points,  we get a diagram representing a DJVS arrangement that produces  a rational homology disk.  To get a simpler picture, we can move the braiding in the subconfiguration to the edge of the large diagram by~~\eqref{eq:move-braid-multipt}, and then cancel the braiding on the edges as in Subsection~\ref{ss:braid-at-edge} to arrive at the final diagram. \hfill\qedsymbol
\begin{center}
{
	\fontsize{9pt}{5pt}
	\def\svgwidth{\textwidth}
	%
\begingroup%
  \makeatletter%
  \providecommand\color[2][]{%
    \errmessage{(Inkscape) Color is used for the text in Inkscape, but the package 'color.sty' is not loaded}%
    \renewcommand\color[2][]{}%
  }%
  \providecommand\transparent[1]{%
    \errmessage{(Inkscape) Transparency is used (non-zero) for the text in Inkscape, but the package 'transparent.sty' is not loaded}%
    \renewcommand\transparent[1]{}%
  }%
  \providecommand\rotatebox[2]{#2}%
  \newcommand*\fsize{\dimexpr\f@size pt\relax}%
  \newcommand*\lineheight[1]{\fontsize{\fsize}{#1\fsize}\selectfont}%
  \ifx\svgwidth\undefined%
    \setlength{\unitlength}{930bp}%
    \ifx\svgscale\undefined%
      \relax%
    \else%
      \setlength{\unitlength}{\unitlength * \real{\svgscale}}%
    \fi%
  \else%
    \setlength{\unitlength}{\svgwidth}%
  \fi%
  \global\let\svgwidth\undefined%
  \global\let\svgscale\undefined%
  \makeatother%
  \begin{picture}(1,0.32258065)%
    \lineheight{1}%
    \setlength\tabcolsep{0pt}%
    \put(0,0){\includegraphics[width=\unitlength,page=1]{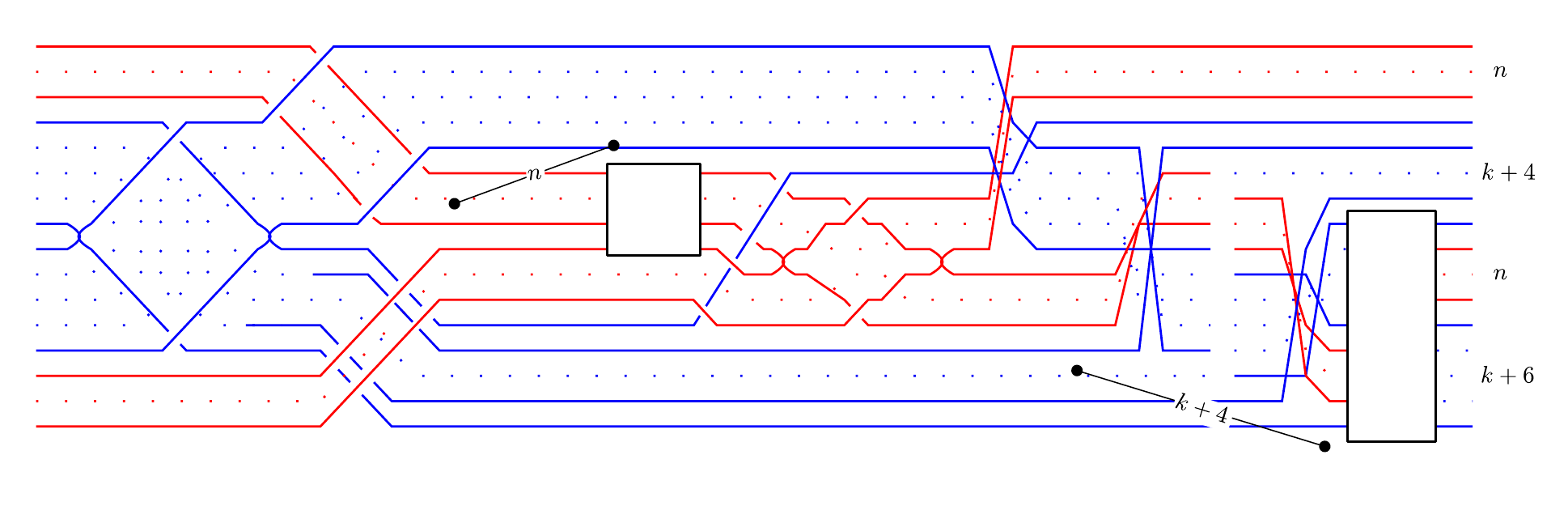}}%
    \put(0.87257287,0.11580132){\color[rgb]{0,0,0}\makebox(0,0)[lt]{\lineheight{1.25}\smash{\begin{tabular}[t]{l}$\I_{2s}^{s}$\end{tabular}}}}%
    \put(0,0){\includegraphics[width=\unitlength,page=2]{Familieskeven15.pdf}}%
    \put(0.79596001,0.13385522){\color[rgb]{0,0,0}\makebox(0,0)[lt]{\lineheight{1.25}\smash{\begin{tabular}[t]{l}$\I_{2s-1}^{s-1}$\end{tabular}}}}%
    \put(0,0){\includegraphics[width=\unitlength,page=3]{Familieskeven15.pdf}}%
    \put(0.69714582,0.16541059){\color[rgb]{0,0,0}\makebox(0,0)[lt]{\lineheight{1.25}\smash{\begin{tabular}[t]{l}$\I_{s+n+2}^{n+2}$\end{tabular}}}}%
    \put(0,0){\includegraphics[width=\unitlength,page=4]{Familieskeven15.pdf}}%
    \put(0.6233555,0.21956641){\color[rgb]{0,0,0}\makebox(0,0)[lt]{\lineheight{1.25}\smash{\begin{tabular}[t]{l}$\I_{s+1}^{1}$\end{tabular}}}}%
    \put(0,0){\includegraphics[width=\unitlength,page=5]{Familieskeven15.pdf}}%
    \put(0.29351615,0.15322581){\color[rgb]{0,0,0}\makebox(0,0)[lt]{\lineheight{1.25}\smash{\begin{tabular}[t]{l}$\I^{s}_{s+n}$\end{tabular}}}}%
    \put(0.3943226,0.18751612){\color[rgb]{0,0,0}\makebox(0,0)[lt]{\lineheight{1.25}\smash{\begin{tabular}[t]{l}$\I^{k+7}_{s+1}$\end{tabular}}}}%
  \end{picture}%
\endgroup%

}
\end{center}

\bibliography{references} 

@book{nemethi2022normal,
	title={Normal surface singularities},
	author={N{\'e}methi, Andr{\'a}s},
	volume={74},
	year={2022},
	publisher={Springer}
}

@article{gompf2005locally,
	title={Locally holomorphic maps yield symplectic structures},
	author={Gompf, Robert E},
	journal={Communications in Analysis and Geometry},
	volume={13},
	number={3},
	pages={511--525},
	year={2005},
	publisher={International Press of Boston, Inc. Somerville, MA 02143, USA}
}

@book{gompf19994,
	title={4-manifolds and Kirby calculus},
	author={Gompf, Robert E and Stipsicz, Andr{\'a}s},
	volume={20},
	year={1999},
	publisher={American Mathematical Society Providence, RI}
}

@article {Ger,
    AUTHOR = {Gervais, Sylvain},
     TITLE = {A finite presentation of the mapping class group of a
              punctured surface},
   JOURNAL = {Topology},
    VOLUME = {40},
      YEAR = {2001},
    NUMBER = {4},
     PAGES = {703--725},
}

@book {BrKnor,
    AUTHOR = {Brieskorn, Egbert and Kn\"orrer, Horst},
     TITLE = {Plane algebraic curves},
    SERIES = {Modern Birkh\"auser Classics},
      NOTE = {Translated from the German original by John Stillwell,
              [2012] reprint of the 1986 edition},
 PUBLISHER = {Birkh\"auser/Springer Basel AG, Basel},
      YEAR = {1986},
}

@article {PS1,
    AUTHOR = {Plamenevskaya, Olga and Starkston, Laura},
     TITLE = {Unexpected {S}tein fillings, rational surface singularities
              and plane curve arrangements},
   JOURNAL = {Geom. Topol.},
  FJOURNAL = {Geometry \& Topology},
    VOLUME = {27},
      YEAR = {2023},
    NUMBER = {3},
     PAGES = {1083--1202},
     }

@article {Wahl,
    AUTHOR = {Wahl, Jonathan},
     TITLE = {Smoothings of normal surface singularities},
   JOURNAL = {Topology},
  FJOURNAL = {Topology. An International Journal of Mathematics},
    VOLUME = {20},
      YEAR = {1981},
    NUMBER = {3},
     PAGES = {219--246},
}

@article {FintushelStern,
    AUTHOR = {Fintushel, Ronald and Stern, Ronald J.},
     TITLE = {Rational blowdowns of smooth {$4$}-manifolds},
   JOURNAL = {J. Differential Geom.},
  FJOURNAL = {Journal of Differential Geometry},
    VOLUME = {46},
      YEAR = {1997},
    NUMBER = {2},
     PAGES = {181--235},
}

@article {Symington,
    AUTHOR = {Symington, Margaret},
     TITLE = {Symplectic rational blowdowns},
   JOURNAL = {J. Differential Geom.},
  FJOURNAL = {Journal of Differential Geometry},
    VOLUME = {50},
      YEAR = {1998},
    NUMBER = {3},
     PAGES = {505--518},
}

@article {JPark,
    AUTHOR = {Park, Jongil},
     TITLE = {Simply connected symplectic 4-manifolds with {$b^+_2=1$} and
              {$c^2_1=2$}},
   JOURNAL = {Invent. Math.},
  FJOURNAL = {Inventiones Mathematicae},
    VOLUME = {159},
      YEAR = {2005},
    NUMBER = {3},
     PAGES = {657--667},
}

@article {StipSzabo,
    AUTHOR = {Stipsicz, Andr\'as I. and Szab\'o, Zolt\'an},
     TITLE = {An exotic smooth structure on {$\Bbb C\Bbb
              P^2\#6\overline{\Bbb C\Bbb P^2}$}},
   JOURNAL = {Geom. Topol.},
  FJOURNAL = {Geometry and Topology},
    VOLUME = {9},
      YEAR = {2005},
     PAGES = {813--832},
}

@article{Beke-inprep,
title={Minimal Rational Graphs Admitting a {$\mathbb Q$HD} smoothing},
author={Beke, M{\'a}rton},
journal={Pacific Journal of Mathematics},
volume={344},
number={1},
pages={1--25},
year={2026},
publisher={Mathematical Sciences Publishers}}

@misc{LVHMW2,
      title={On symplectic fillings of spinal open book decompositions {II}: Holomorphic curves and classification}, 
      author={Samuel Lisi and Jeremy Van Horn-Morris and Chris Wendl},
      year={2020},
      eprint={2010.16330},
      archivePrefix={arXiv},
      primaryClass={math.SG},
      url={https://arxiv.org/abs/2010.16330}, 
}

@misc{LVHMW,
      title={On symplectic fillings of spinal open book decompositions {I}: Geometric constructions}, 
      author={Samuel Lisi and Jeremy Van Horn-Morris and Chris Wendl},
      year={2018},
      eprint={1810.12017},
      archivePrefix={arXiv},
      primaryClass={math.SG},
      url={https://arxiv.org/abs/1810.12017}, 
}

@misc{PS2,
Author = {Plamenevskaya,  Olga  and Starkston, Laura},
Title = {Sandwiched singularities and nearly {L}efschetz fibrations},
year = {2025},
Note = {arXiv:2507.21293}}

@article {BaHa,
    AUTHOR = {Baykur, R. \.Inan\c{c} and Hayano, Kenta},
     TITLE = {Multisections of {L}efschetz fibrations and topology of
              symplectic 4-manifolds},
   JOURNAL = {Geom. Topol.},
  FJOURNAL = {Geometry \& Topology},
    VOLUME = {20},
      YEAR = {2016},
    NUMBER = {4},
     PAGES = {2335--2395}}

@article {SSW,
    AUTHOR = {Stipsicz, Andr\'as I. and Szab\'o, Zolt\'an and Wahl,
              Jonathan},
     TITLE = {Rational blowdowns and smoothings of surface singularities},
   JOURNAL = {J. Topol.},
  FJOURNAL = {Journal of Topology},
    VOLUME = {1},
      YEAR = {2008},
    NUMBER = {2},
     PAGES = {477--517}}

@misc{HRW,
Author = {Min,  Hyunki  and Roy, Agniva and Wang, Luya},
Title = {Spinal open books and symplectic fillings with exotic fibers},
Note = {arXiv:2410.10697},
Year = {2024}}

@article {ACampo,
    AUTHOR = {A'Campo, Norbert},
     TITLE = {Le groupe de monodromie du d\'{e}ploiement des singularit\'{e}s
              isol\'{e}es de courbes planes. {I}},
   JOURNAL = {Math. Ann.},
  FJOURNAL = {Mathematische Annalen},
    VOLUME = {213},
      YEAR = {1975},
     PAGES = {1--32}}

@article {Lauf,
    AUTHOR = {Laufer, Henry B.},
     TITLE = {Taut two-dimensional singularities},
   JOURNAL = {Math. Ann.},
  FJOURNAL = {Mathematische Annalen},
    VOLUME = {205},
      YEAR = {1973},
     PAGES = {131--164}}

@article {CS,
	AUTHOR = {Cohen, Daniel C. and Suciu, Alexander I.},
	TITLE = {The braid monodromy of plane algebraic curves and hyperplane
	arrangements},
	JOURNAL = {Comment. Math. Helv.},
	FJOURNAL = {Commentarii Mathematici Helvetici},
	VOLUME = {72},
	YEAR = {1997},
	NUMBER = {2},
	PAGES = {285--315},
	ISSN = {0010-2571},
	MRCLASS = {52B30 (14H30 20F36 57M05)},
	MRNUMBER = {1470093},
	MRREVIEWER = {Lee Rudolph},
	DOI = {10.1007/s000140050017},
	URL = {https://doi.org/10.1007/s000140050017},
}

@article {djvs,
    AUTHOR = {de Jong, T. and van Straten, D.},
     TITLE = {Deformation theory of sandwiched singularities},
   JOURNAL = {Duke Math. J.},
  FJOURNAL = {Duke Mathematical Journal},
    VOLUME = {95},
      YEAR = {1998},
    NUMBER = {3},
     PAGES = {451--522},
      ISSN = {0012-7094},
   MRCLASS = {14B07 (14J17 32S30)},
  MRNUMBER = {1658768},
MRREVIEWER = {Gerhard Pfister},
       DOI = {10.1215/S0012-7094-98-09513-8},
       URL = {https://doi.org/10.1215/S0012-7094-98-09513-8}}

@phdthesis{fowler2013rational,
	title={Rational homology disk smoothing components of weighted homogeneous surface singularities},
	author={Fowler, Jacob R},
	year={2013},
	school={The University of North Carolina at Chapel Hill}
}

@article{wahl2011rational,
	title={On rational homology disk smoothings of valency 4 surface singularities},
	author={Wahl, Jonathan},
	journal={Geometry \& Topology},
	volume={15},
	number={2},
	pages={1125--1156},
	year={2011},
	publisher={Mathematical Sciences Publishers}
}

@article{Tes,
	title={The hunting of invariants in the geometry of discriminants},
	author={Teissier, Bernard},
	journal={Real and complex singularities, Oslo},
	pages={565--677},
	year={1976}
}

@article{ARVOLA,
	title = {The fundamental group of the complement of an arrangement of complex hyperplanes},
	journal = {Topology},
	volume = {31},
	number = {4},
	pages = {757-765},
	year = {1992},
	issn = {0040-9383},
	doi = {https://doi.org/10.1016/0040-9383(92)90006-4},
	url = {https://www.sciencedirect.com/science/article/pii/0040938392900064},
	author = {William A. Arvola}
}

@article{KorOz,
	title={On sections of elliptic fibrations},
	author={Korkmaz, Mustafa and Ozbagci, Burak},
	journal={Michigan Mathematical Journal},
	volume={56},
	number={1},
	pages={77--87},
	year={2008},
	publisher={University of Michigan, Department of Mathematics}
}

@book{HarKir,
	title={Handlebody decompositions of complex surfaces},
	author={Harer, John and Kas, Arnold and Kirby, Robion C},
	volume={350},
	year={1986},
	publisher={American Mathematical Soc.}
}

@article{stipsicz07,
	title={Singular fibers in elliptic fibrations on the rational elliptic surface},
	author={Stipsicz, Andr{\'a}s I and Szab{\'o}, Zolt{\'a}n and Szil{\'a}rd, {\'A}gnes},
	journal={Periodica Mathematica Hungarica},
	volume={54},
	number={2},
	pages={137--162},
	year={2007},
	publisher={Springer}
}

@article{CPP,
	title = {On the contact boundaries of normal surface singularities},
	journal = {Comptes Rendus Mathematique},
	volume = {339},
	number = {1},
	pages = {43-48},
	year = {2004},
	issn = {1631-073X},
	doi = {https://doi.org/10.1016/j.crma.2004.04.023},
	url = {https://www.sciencedirect.com/science/article/pii/S1631073X04002316},
	author = {Clément Caubel and Patrick Popescu-Pampu},
}

@article{bhupal2011weighted,
	title={Weighted homogeneous singularities and rational homology disk smoothings},
	author={Bhupal, Mohan and Stipsicz, Andr{\'a}s I},
	journal={American journal of mathematics},
	volume={133},
	number={5},
	pages={1259--1297},
	year={2011},
	publisher={Johns Hopkins University Press}
}
\bibliographystyle{alpha}

\end{document}